\documentclass[12pt,reqno]{amsart}
\usepackage{amssymb}
\usepackage{amsmath}
\usepackage{bm}
\usepackage{xcolor,datetime}
\usepackage{times}
\usepackage[
bookmarks=true,unicode,bookmarks=true,colorlinks,
citecolor=blue,linkcolor=blue]{hyperref}
\usepackage{fouriernc}
\usepackage[top=1.2in,bottom=1.2in,left=1.2in,right=1.2in]{geometry}
\usepackage{graphicx}
\usepackage{subfigure}
\usepackage{multirow}
\usepackage{amsthm}
\usepackage{enumitem}
\numberwithin{equation}{section}

\newtheorem{theorem}{Theorem}[section]
\newtheorem{lem}{Lemma}[section]
\newtheorem{pro}{Proposition}[section]

\newtheorem{rmk}{Remark}[section]

\newtheorem{assumption}{Assumption}
\numberwithin{equation}{section}

\renewcommand{\P}{\mathbb{P}}
\newcommand{\E}{\mathbb{E}}
\newcommand{\R}{\mathbb{R}}

\allowdisplaybreaks

\begin{document}

\title{Central limit theorem and Cram\'{e}r-type moderate deviations for Milstein scheme}

\author[P.~Chen]{Peng Chen}
\address[P.~Chen]{School of Mathematics, Nanjing University of Aeronautics and
Astronautics, Nanjing 211106, China}
\email{chenpengmath@nuaa.edu.cn}

\author[H.~Jiang]{Hui Jiang}
\address[H.~Jiang]{School of Mathematics, Nanjing University of Aeronautics and
Astronautics, Nanjing 211106, China}
\email{huijiang@nuaa.edu.cn}

\author[J.~Wang]{Jing Wang}
\address[J.~Wang]{School of Mathematics, Nanjing University of Aeronautics and
Astronautics, Nanjing 211106, China}
\email{wangjingmath@nuaa.edu.cn}

\maketitle

\begin{abstract}
In this paper, we investigate the Milstein numerical scheme with step size $\eta$ for a stochastic differential equation driven by multiplicative Brownian motion. Under some appropriate coefficient conditions, the continuous-time system and its discrete Milstein scheme approximation each possess unique invariant measures, which we denote by $\pi$ and $\pi_\eta$ respectively. We first establish a central limit theorem for the empirical measure $\Pi_{\eta}$, a statistical consistent estimator of $\pi_{\eta}$. Subsequently, we derive both normalized and self-normalized Cram\'{e}r-type moderate deviations.
\end{abstract}

\maketitle


\noindent{\bf 2000MSC:} primary 60F10; 60E05; secondary 60H10; 62F12.

\section{Introduction}

We consider the following stochastic differential equation (SDE) on $\mathbb{R}^d$:
\begin{equation}\label{SDE}
\mathrm{d}X_t=b(X_t)\mathrm{d}t+\sigma(X_t)\mathrm{d}B_t,\quad X_0=x_0,
\end{equation}
where $b:\mathbb{R}^d\to\mathbb{R}^d$ and $\sigma:\mathbb{R}^d\to\mathbb{R}^{d\times d}$, and $(B_t)_{t\geq0}$ is a $d$-dimensional standard Brownian motion.

Under appropriate coefficient conditions, the existence and uniqueness of solution to the SDE (\ref{SDE}) have been thoroughly established in \cite{CZZ,Weinan,Evan,RWZ} and the references therein. Now, recent researches have increasingly focused on studying the underlying invariant measure, and investigating the corresponding numerical invariant measures derived from various discretization schemes with their convergence rate. The popular numerical methods include the Euler-Maruyama scheme, Milstein scheme and other high-order discretization schemes, see more details in \cite{Abdulle,Weinan,Szpruch,Talay}.

Given a step size $\eta\in(0,1)$, the Euler-Maruyama scheme of (\ref{SDE}) reads as
\begin{align}\label{EM}
    \theta_{k+1}=\theta_k+\eta b(\theta_k)+\sqrt{\eta}\sigma(\theta_k)\xi_{k+1},\quad k\geq 0,
\end{align}
where $(\xi_k)_{k\geq 1}$ are i.i.d. $d$-dimensional standard normal random vectors. The convergence properties of the Euler-Maruyama scheme have been extensively studied across various stochastic systems, including SDEs driven by Brownian motion, Markov switching and $\alpha$-stable L\'evy processes in \cite{Bao,Chen-2023,Chen-2024,Lu,Mao,Yuan}. For the backward Euler-Maruyama method, \cite{Liu,Liu-2023,Li} investigated invariant measures for SDEs with Markov switching, with nonlinear and super-linear coefficients respectively. Meanwhile, \cite{Jiang,Jiang-2020} employed the stochastic $\theta$ method to study SDEs with Markov switching and nonlinear structures. In \cite{Chen-2024}, variable step-size Euler-Maruyama scheme was applied to approximate the invariant measure of regime-switching jump-diffusion processes.

For the convergence rate and Cram\'{e}r-type moderate deviations of Euler-Maruyama scheme, in the case of additive noise (i.e. $\sigma(x)\equiv\sigma$), \cite{Fang} has proved that the Wasserstein-1 distance between $\pi$ and $\pi_{\eta}^{EM}$ (the unique invariant measure of the Euler-Maruyama scheme) is in order of $\eta^{1/2}$, up to a logarithmic correction. Moreover, \cite{Lu} has obtained the central limit theorem and normalized Cram\'er-type moderate deviation for Euler-Maruyama scheme (\ref{EM}), while \cite{Fan} extends the range of Cram\'er-type moderate deviations by the martingale methods. For more results on Cram\'er-type moderate deviation for dependent time series, we refer to \cite{Chen,Fan-2020,FSZ21,GJZ23,JPW24,SZZ21} and the references therein.

In this paper, we focus on the Milstein scheme for SDE (\ref{SDE}). Given a step size $\eta\in(0,1)$, the Milstein scheme can be given as
\begin{equation}\label{Milstein}
\theta_{k+1}=\theta_k+\eta b(\theta_k)+\sqrt{\eta}\sigma(\theta_k)\xi_{k+1}+\frac{1}{2}\eta\mathcal{R}(\theta_k,\xi_{k+1}),\quad k\geq 0,
\end{equation}
where
\begin{align*}
\mathcal{R}(\theta_k,\xi_{k+1}):=\left(\nabla_{\sigma(\theta_k)\xi_{k+1}}\sigma(\theta_k)\right)\xi_{k+1}-\mathbb{E}\left(\left(\nabla_{\sigma(\theta_k)\xi_{k+1}}\sigma(\theta_k)\right)\xi_{k+1}\right),
\end{align*}
and $\nabla_v f(x)$ denotes the directional derivative of $f \in \mathcal{C}^2(\mathbb{R}^d,\mathbb{R}^{d\times d})$ along $v\in \mathbb{R}^d$, defined in Section \ref{preli}. Here $\mathcal{C}^k(\R^d,\R^{d\times d})$ with $k\geq 1$ denotes the collection of all $k$-th order continuous differentiable functions.

Compared to the Euler-Maruyama scheme, the invariant measure 
$\pi_{\eta}$ of the Milstein scheme \eqref{Milstein} has received less attention. \cite{WL19} showed the invariant measure of the Milstein scheme for SDE with commutative noise, while \cite{Gao} proved that the Milstein scheme admits a unique invariant measure for stochastic differential delay equation with exponential convergence to the underlying one in the Wasserstein metric. Here, compared with the Euler-Maruyama scheme, the study of Cram\'er-type moderate deviations for the Milstein scheme is notably fewer.

In this paper, we first construct an empirical measure $\Pi_{\eta}$ of the Milstein scheme as a statistic of $\pi_{\eta}$. And then, we apply Stein's method established in \cite{Fang} to study the corresponding  Cram\'er-type moderate deviations for $\eta^{-1/2}(\Pi_{\eta}(\cdot)-\pi(\cdot))$. The motivation of this paper has two folds:
\begin{enumerate}
    \item establish the central limit theorem of $\Pi_{\eta}$;
    \item derive both normalized and self-normalized Cram\'er-type moderate deviations of $\eta^{-1/2}(\Pi_{\eta}(\cdot)-\pi(\cdot))$.
\end{enumerate}

The remainder of this paper is organized as follows. Some assumptions and main theorems are stated in Section 2. In Section 3, we introduce key notations, outline the proof strategies for the main results, and provide essential propositions, while the technical proofs are deferred to Appendix \ref{Appendix} and \ref{AppendixB}. The detailed proofs of our main theorems are systematically developed in Sections \ref{ProofCLT}--\ref{ProofS}, respectively.

\section{Assumptions and main theorems}

\subsection{Assumptions and main framework}

We first state the main assumptions.
\begin{assumption}\label{conditions}
Suppose $\sigma:\mathbb{R}^d\to\mathbb{R}^{d\times d}$ and $b:\mathbb{R}^d\to\mathbb{R}^d$ are second order differentiable. There exist $L$, $K_1>0$ and $K_2\geq0$ such that for every $x,y\in\R^d$,
\begin{align}\label{Lip}
\Arrowvert b(x)-b(y)\Arrowvert_2 \vee \Arrowvert\sigma(x)-\sigma(y)\Arrowvert \leq L\Arrowvert x-y\Arrowvert_2,
\end{align}
\begin{align}\label{Haosan}
\big<b(x)-b(y),x-y\big>\leq-K_1\Arrowvert x-y\Arrowvert_2^2+K_2.
\end{align}
Moreover, $\sigma$ is bounded and positive definite, and $\nabla\sigma\not\equiv 0$.
\end{assumption}
Here, $\left<\cdot,\cdot\right>$ and $\Arrowvert x\Arrowvert_2$ denote the inner product on $\R^d$ and the Euclidean norm of a vector $x\in\R^d$, respectively. For $A\in\R^{d\times
 d}$, we denote the operator norm $\Arrowvert A\Arrowvert$ by
$$
\Arrowvert A\Arrowvert=\sup_{v\in\R^d,\Arrowvert v\Arrowvert_2=1}\Arrowvert Av\Arrowvert_2.
$$

\begin{assumption}\label{conditions2}
    For the random variable $\theta_0$ (the initial value of the Milstein scheme \eqref{Milstein}) and for $\gamma>0$ depending on $K_1$ and $L$, there exists a positive constant $C$ such that
    \begin{align*}
        \E\exp\left\{\gamma\Arrowvert\theta_0\Arrowvert_2^2\right\}\leq C.
    \end{align*}
\end{assumption}

\begin{rmk}
It is easy to see that the condition (\ref{Lip}) implies
\begin{align}\label{|g(x)|^2}
\Arrowvert b(x)\Arrowvert_2^2\leq 2L^2\Arrowvert x\Arrowvert_2^2+2\Arrowvert b(0)\Arrowvert_2^2,\quad \Arrowvert\nabla b\Arrowvert\leq L.
\end{align}
Then, the condition (\ref{Haosan}) with Young's inequality imply
\begin{align}\label{|x|^2}
\left<x,b(x)\right>
&=\left<x-0,b(x)-b(0)\right>+\left<x,b(0)\right>\nonumber\\
&\leq-K_1\Arrowvert x\Arrowvert_2^2+K_2+\frac{K_1}{2}\Arrowvert x\Arrowvert_2^2+\frac{1}{2K_1}\Arrowvert b(0)\Arrowvert_2^2=-\frac{K_1}{2}\Arrowvert x\Arrowvert_2^2+\left(K_2+\frac{1}{2K_1}\Arrowvert b(0)\Arrowvert_2^2\right).
\end{align}
\end{rmk}

For a small $\eta\in(0,1)$, define the empirical invariant measure of $\pi_{\eta}$ as
\begin{align*}
\Pi_\eta(\cdot)=\frac{1}{[\eta^{-2}]}\sum_{k=0}^{[\eta^{-2}]-1}\delta_{\theta_k}(\cdot),
\end{align*}
where $\delta_y(\cdot)$ is the Dirac measure of $y$ and $[\eta^{-2}]$ denote the integer part of $\eta^{-2}$. Then $\Pi_\eta$ is a consistent statistic of $\pi$ as $\eta\to0$.

For a matrix $A\in \mathbb{R}^{d \times d}$, we write $A^{\rm{T}}$ for the transpose of $A$. Let $\pi(h)=\int_{\R^d}h(x)\pi(\mathrm{d}x)$ and define
\begin{equation}\label{Y}
W_\eta=\frac{\eta^{-1/2}(\Pi_\eta(h)-\pi(h))}{\sqrt{\mathcal{Y}_\eta}} \quad \text{and} \quad S_\eta=\frac{\eta^{-1/2}(\Pi_\eta(h)-\pi(h))}{\sqrt{\mathcal{V}_\eta}}, 
\end{equation}
where
\begin{gather}\label{V}
\mathcal{Y}_\eta=\frac{1}{[\eta^{-2}]}\sum_{k=0}^{[\eta^{-2}]-1}\big\Arrowvert \sigma(\theta_k)^{\rm T}\nabla f_{h}(\theta_k)\big\Arrowvert_2^2;\nonumber\\
\mathcal{V}_\eta=\frac{1}{\eta[\eta^{-2}]}\sum_{k=0}^{[\eta^{-2}]-1}  \big\Arrowvert\big(\theta_{k+1}-\theta_k-\eta b(\theta_k)-\frac{1}{2}\eta\mathcal{R}(\theta_k,\xi_{k+1})\big)^{\rm T}\nabla f_{h}(\theta_k)\big\Arrowvert_2^2.
\end{gather}
Here, $f_h$ is the solution to the following Stein equation (\cite{Lu}):
\begin{align}\label{Stein}
h-\pi(h)=\mathcal{A}f,\quad h\in \mathcal{C}_b^2(\R^d,\R),
\end{align}
where $\mathcal{C}_b^k(\R^d,\R)$ with $k\geq 1$ denotes the collection of all bounded $k$-th order continuously differentiable functions, and $\mathcal{A}$ is the generator of SDE (\ref{SDE}) defined as
\begin{align}\label{Af(x)}
\mathcal{A}f(x)=\big<b(x),\nabla f(x)\big>+\frac{1}{2}\big<\sigma(x)\sigma(x)^{\rm T},\nabla^2f(x)\big>_{\rm{HS}},\qquad f\in \mathcal{C}_b^2(\R^d,\R).
\end{align}
Here, for the matrices $A, B \in \mathbb{R}^{d \times d}$, define $\langle A, B\rangle_{{\rm HS}}:=\sum_{i,j=1}^d A_{ij} B_{ij}$. Notably, Lemma 3.1 in \cite{Lu} provides the key regularity estimates for $f_{h}$, i.e.
\begin{align}\label{boundedness}
    \Arrowvert\nabla^kf_{h}\Arrowvert\leq C,\quad k=0,1,2,3,4,
\end{align}
where the positive constant $C$ depends on $b$ and $\sigma$.

\subsection{Main results}

First, we give the central limit theorem of $\Pi_{\eta}(h)$.
\begin{theorem}\label{CLT}
Let {\bf Assumption \ref{conditions}} hold and $h\in \mathcal{C}_b^2(\R^d,\R)$. Then we have
\begin{align*}
\frac{1}{\sqrt{\eta}}\left(\Pi_\eta(h)-\pi(h)\right)\xrightarrow{\mathcal{L}} N\left(0,\pi\left(\left\Arrowvert\sigma^{{\rm T}}\nabla f_{h}\right\Arrowvert_2^{2}\right)\right),
\end{align*}
where $\xrightarrow{\mathcal{L}}$ denotes the convergence in distribution.
\end{theorem}

Second, we state the normalized Cram\'{e}r-type moderate deviations.
\begin{theorem}\label{W}
Let {\bf Assumptions} {\bf\ref{conditions}} and {\bf\ref{conditions2}} hold, and $h\in \mathcal{C}_b^2(\R^d,\R)$. Then we have for all $c\Arrowvert\nabla\sigma\Arrowvert_{\infty}^{3/4}\eta^{1/8}\leq x=o(\Arrowvert\nabla\sigma\Arrowvert_{\infty}^{-3/4}\eta^{-1/8})$,
$$
\left\lvert\ln\frac{\P\big(W_\eta>x\big)}{1-\Phi(x)}\right\rvert\leq c\left(x^3\eta^{1/2}+x\Arrowvert\nabla\sigma\Arrowvert_{\infty}^{3/4}\eta^{1/8}+\eta^{1/2}\big|\ln\eta\big|\right).
$$
In particular, it implies that
$$
\sup_{0\leq x=o(\Arrowvert\nabla\sigma\Arrowvert_{\infty}^{-3/4}\eta^{-1/8})}\left\lvert\frac{\P\big(W_\eta>x\big)}{1-\Phi(x)}-1\right\rvert\to 0,\quad as \quad\eta\to0.
$$

Moreover, the same results also hold when $W_\eta$ is replaced by $-W_\eta$.
\end{theorem}

Finally, we present the self-normalized Cram\'{e}r-type moderate deviations for the Milstein scheme for SDE (\ref{SDE}).
\begin{theorem}\label{S}
Let {\bf Assumptions} {\bf\ref{conditions}} and {\bf\ref{conditions2}} hold, and $h\in \mathcal{C}_b^2(\R^d,\R)$. Then we have for all $c\Arrowvert\nabla\sigma\Arrowvert_{\infty}^{3/4}\eta^{1/8}
\leq x=o(\Arrowvert\nabla\sigma\Arrowvert_{\infty}^{-3/4}\eta^{-1/8})$,
$$
\left\lvert\ln\frac{\P\big(S_\eta>x\big)}{1-\Phi(x)}\right\rvert\leq c\left(x^3\eta^{1/2}+x^2\eta^{1/2}+x\Arrowvert\nabla\sigma\Arrowvert_{\infty}^{3/4}\eta^{1/8}+\eta^{1/2}\big|\ln\eta\big|\right).
$$
In particular, it implies that
$$
\sup_{0\leq x=o(\Arrowvert\nabla\sigma\Arrowvert_{\infty}^{-3/4}\eta^{-1/8})}\left\lvert\frac{\P\big(S_\eta>x\big)}{1-\Phi(x)}-1\right\rvert\to 0,\quad as\quad \eta\to 0.
$$

Moreover, the same results also hold when $S_\eta$ is replaced by $-S_\eta$.
\end{theorem}

\section{Preliminary propositions and Notations}\label{preli}

To ensure a rigorous analytical framework, this section first introduces some notations, then articulates the principal proof techniques, and finally establishes auxiliary propositions which will be employed in later sections.

{\bf Notations.} We now introduce the following notations.
\begin{enumerate}
\item For $f\in\mathcal{C}^2(\R^d,\R)$ and $v, v_1, v_2,x\in \R^d$, the directional derivative $\nabla_v f(x)$ and $\nabla_{v_2}\nabla_{v_1} f(x)$ are  defined by
\begin{align*}
    \nabla_v f(x)\ = \ \lim_{\epsilon \rightarrow 0} \frac{f(x+\epsilon v)-f(x)}{\epsilon}\quad \text{and} \quad \nabla_{v_2} \nabla_{v_1} f(x)\ = \ \lim_{\epsilon \rightarrow 0} \frac{\nabla_{v_1}f(x+\epsilon v_2)-\nabla_{v_1}f(x)}{\epsilon}.
\end{align*}
Let $\nabla f(x)\in \mathbb{R}^d$ and $\nabla^2 f(x)\in \mathbb{R}^{d \times d}$ denote the gradient and the Hessian of $f$, respectively.
It is known that
$\nabla_v f(x) = \left\langle \nabla f(x), v\right\rangle$ and
$\nabla_{v_2} \nabla_{v_1} f(x)=\left\langle \nabla^2 f(x), v_1 v^{\rm T}_2\right\rangle_{{\rm HS}}$.

\item Similarly, for a second-order differentiable function $f=\left(f_1,\dots, f_d\right)^{\rm T}:\R^d \rightarrow \R^d$, define $\nabla f(x)=\left(\nabla f_1(x), \dots, \nabla f_d(x)\right) \in \mathbb{R}^{d \times d}$ and $\nabla^2 f(x)=\left\{\nabla^2 f_i(x)\right\}_{i=1}^d \in \mathbb{R}^{d \times d \times d}$. In this case, we have $\nabla_v f(x)=[\nabla f(x)]^{\rm T} v$,
$$
\nabla_{v_2} \nabla_{v_1} f(x)=\left\{ \left\langle \nabla^2 f_1(x), v_1 v^{\rm T}_2\right\rangle_{{\rm HS}},\dots, \left\langle \nabla^2 f_d(x), v_1 v^{\rm T}_2\right\rangle_{{\rm HS}}\right\}^{\rm T},
$$
and for any tensor $A\in\mathbb{R}^{d\times d\times d},$ $\left\langle \left\langle A,v_{1}v_{2}^{\rm T}\right\rangle_{{\rm {HS}}},v_{3}\right\rangle=\sum_{i,j,k=1}^{d}A_{ijk}v_{1}^{(i)}v_{2}^{(j)}v_{3}^{(k)}$ with $v_{l}^{(i)}$ is the $i$-th component of the vector $v_{l},$ $l=1,2,3.$

\item For $M\in\mathcal{C}^{2}(\mathbb{R}^{d},\mathbb{R}^{d\times d})$ and $v,v_{1},v_{2},x\in\mathbb{R}^{d},$ the directional derivative $\nabla_{v}M(x)$ and $\nabla_{v_{2}}\nabla_{v_{1}}M(x)$ are defined by
\begin{align*}
\nabla_{v}M(x)=\lim_{\epsilon\rightarrow0}\frac{M(x+\epsilon v)-M(x)}{\epsilon}\quad \text{and} \quad \nabla_{v_{2}}\nabla_{v_{1}}M(x)=\lim_{\epsilon\rightarrow0}\frac{\nabla_{v_{1}}M(x+\epsilon v_{2})-\nabla_{v_{1}}M(x)}{\epsilon}.
\end{align*}

\item For $f\in\mathcal{C}^2(\R^d,\R)$, define the operator norm of $\nabla^{2} f(x)$ by
\begin{align*}
\|\nabla^{2} f(x)\|_{\rm op}= \sup_{|v_{1}|,|v_{2}|=1} |\nabla_{v_{2}} \nabla_{v_{1}} f(x)|\quad \text{and}\quad \|\nabla^{2} f\|_{{\rm op}, \infty}= \sup_{x \in \mathbb{R}^{d}} \|\nabla^{2} f(x)\|_{\rm op}.
\end{align*}
We often drop the subscript "{\rm op}" in the definitions above and simply write $\|\nabla^{2} f(x)\|=\|\nabla^{2} f(x)\|_{\rm op}$ and  $\|\nabla^{2} f\|_{\infty}=\|\nabla^{2} f\|_{{\rm op}, \infty}$ if no confusions arise. For higher rank tensors, we can define them analogously.

\item Let $\{a_{n}\}_{n\geq1}$ and $\{b_{n}\}_{n\geq1}$ be two nonnegative real number sequences, if there exists some $C>0$ such that $a_{n}\leq Cb_{n}$, we write $a_{n}=O(b_{n})$. If $\lim_{n\rightarrow\infty}\frac{a_{n}}{b_{n}}=0$, we write $a_{n}=o(b_{n})$.
\end{enumerate}

\begin{rmk}\label{Re}
According to the definition of directional derivative, for any $x\in\mathbb{R}^{d}$ and $k\geq0$, we have
\begin{align*}
\left(\nabla_{\sigma(x)\xi_{k+1}}\sigma(x)\right)\xi_{k+1}=\sum_{j_{1}=1}^{d}\sum_{j_{2}=1}^{d}\sum_{l=1}^{d}\sigma_{l,j_{1}}(x)\frac{\partial \sigma_{j_{2}}(x)}{\partial x^{l}}\xi_{k+1}^{j_{1}}\xi_{k+1}^{j_{2}},
\end{align*}
where $x=(x^{1},\cdots,x^{d})\in\mathbb{R}^{d}$, $\sigma(x)=\left(\sigma_{i,j}(x)\right)_{i,j\in\{1,\cdots,d\}}$ and $\sigma_{j}(x)=\left(\sigma_{1,j}(x),\cdots,\sigma_{d,j}(x)\right)^{\rm T}$ for any $j=1,\cdots,d$. Notice that $\xi_{k+1}$ is a $d$-dimensional standard normal random vectors, we have $\mathbb{E}\left(\xi_{k+1}^{j_{1}}\xi_{k+1}^{j_{2}}\right)=0$ and $\mathbb{E}\left(\big(\xi_{k+1}^{j_{1}}\big)^{2}\right)=1$, $j_{1}\neq j_{2}$. Consequently,
\begin{align*}
\mathbb{E}\left(\left(\nabla_{\sigma(\theta_k)\xi_{k+1}}\sigma(\theta_k)\right)\xi_{k+1}\right)=\sum_{j=1}^{d}\sum_{l=1}^{d}\sigma_{l,j}(x)\frac{\partial \sigma_{j}(x)}{\partial x^{l}}.
\end{align*}
Hence, the Milstein scheme \eqref{Milstein} is equivalent to the form studied in \cite{WL19}.
\end{rmk}

\subsection{The key decomposition of $\Pi_{\eta}(\cdot)$}

To establish our main results, we employ a decomposition strategy, in which we split $\eta^{-\frac{1}{2}}\left(\Pi_{\eta}(h)-\pi(h)\right)$ into the sum of a martingale difference sequence and some asymptotically negligible remainder terms, as shown in \eqref{decom} below.

Without loss of generality, we assume from now on that $m=\eta^{-2}$ is an integer. By Stein's equation \eqref{Stein}, we have
\begin{align*}
\Pi_\eta(h)-\pi(h)=\eta\big(f_h(\theta_m)-f_h(\theta_0)\big)+\eta\sum_{k=0}^{m-1}\big(\mathcal{A}f_h(\theta_k)\eta-\big(f_h(\theta_{k+1})-f_h(\theta_k)\big)\big).
\end{align*}
Notice that $\triangle\theta_k=\theta_{k+1}-\theta_k=\eta b(\theta_k)+\sqrt{\eta}\sigma(\theta_k)\xi_{k+1}+\frac{1}{2}\eta\mathcal{R}(\theta_k,\xi_{k+1})$. Via the Taylor expansion, we have
\begin{align*}
&\mathcal{A}f_h(\theta_k)\eta-\big(f_h(\theta_{k+1})-f_h(\theta_k)\big)\\
=&\eta\big<b(\theta_k),\nabla f_h(\theta_k)\big>+\frac{1}{2}\eta\langle\sigma(\theta_k)\sigma(\theta_k)^{\rm T},\nabla^2 f_h(\theta_k)\rangle_{\rm{HS}}-\big<\nabla f_h(\theta_k),\triangle\theta_k\big>\\&
-\frac{1}{2}\big<\nabla^2 f_h(\theta_k),(\triangle\theta_k)(\triangle\theta_k)^{\rm T}\big>_{\rm{HS}}-\frac{1}{2}\int_{0}^{1}\left(1-t\right)^2\nabla_{\triangle\theta_k}\nabla_{\triangle\theta_k}\nabla_{\triangle\theta_k} f_h(\theta_k+t\triangle\theta_k)\mathrm{d}t\\
=&-\sqrt{\eta}\left<\nabla f_h(\theta_k),\sigma(\theta_k)\xi_{k+1}\right>-\frac{1}{2}\eta\left<\nabla f_h(\theta_k),\mathcal{R}(\theta_k,\xi_{k+1})\right>+\frac{1}{2}\eta\big<\sigma(\theta_k)\sigma(\theta_k)^{\rm T},\nabla^2 f_h(\theta_k)\big>_{\rm{HS}}
\\&-\frac{1}{2}\big<\nabla^2 f_h(\theta_k),(\triangle\theta_k)(\triangle\theta_k)^{\rm T}\big>_{\rm{HS}}-\frac{1}{2}\int_{0}^{1}\left(1-t\right)^2\nabla_{\triangle\theta_k}\nabla_{\triangle\theta_k}\nabla_{\triangle\theta_k} f_h(\theta_k+t\triangle\theta_k)\mathrm{d}t.
\end{align*}
Then, we can obtain
\begin{align}\label{decom}
    \eta^{-1/2}(\Pi_\eta(h)-\pi(h))=\mathcal{H}_\eta+\mathcal{R}_\eta,
\end{align}
where, as we shall see below, $\mathcal{H}_\eta$ is a martingale and $\mathcal{R}_\eta$ is a remainder term, given by
$$
\mathcal{H}_\eta=-\eta\sum_{k=0}^{m-1}\big<\nabla f_h(\theta_k),\sigma(\theta_k)\xi_{k+1}\big>,\quad \mathcal{R}_\eta=-\sum_{i=1}^{13}\mathcal{R}_{\eta,i},
$$
with
\begin{align*}
\mathcal{R}_{\eta,1}=&\sqrt{\eta}\left(f_h(\theta_0)-f_h(\theta_m)\right),\\
\mathcal{R}_{\eta,2}=&\frac{\eta^{\frac{3}{2}}}{2}\sum_{k=0}^{m-1}\left<\nabla^2f_h(\theta_k),\left(\sigma(\theta_k)\xi_{k+1}\right)\left(\sigma(\theta_k)\xi_{k+1}\right)^{\rm T}-\sigma(\theta_k)\sigma(\theta_k)^{\rm T}\right>_{\rm{HS}},\\
\mathcal{R}_{\eta,3}=&\frac{\eta^2}{2}\sum_{k=0}^{m-1}\left(\left<\nabla^2f_h(\theta_k),b(\theta_k)\left(\sigma(\theta_k)\xi_{k+1}\right)^{\rm T}\right>_{\rm{HS}}+\left<\nabla^2f_h(\theta_k),\sigma(\theta_k)\xi_{k+1}\left(b(\theta_k)\right)^{\rm T}\right>_{\rm{HS}}\right),\\
\mathcal{R}_{\eta,4}=&\frac{\eta^2}{2}\sum_{k=0}^{m-1}\int_0^1\left(1-t\right)^2\nabla_{\sigma(\theta_k)\xi_{k+1}}\nabla_{\sigma(\theta_k)\xi_{k+1}}\nabla_{\sigma(\theta_k)\xi_{k+1}}f_h(\theta_k+t\triangle\theta_k)\mathrm{d}t,\\
\mathcal{R}_{\eta,5}=&\frac{\eta^{\frac{5}{2}}}{2}\sum_{k=0}^{m-1}\left<\nabla^2f_h(\theta_k),b(\theta_k)\left(b(\theta_k)\right)^{\rm T}\right>_{\rm{HS}}
\\&+\frac{\eta^{\frac{7}{2}}}{2}\sum_{k=0}^{m-1}\int_0^1\left(1-t\right)^2\nabla_{b(\theta_k)}\nabla_{b(\theta_k)}\nabla_{b(\theta_k)}f_h(\theta_k+t\triangle\theta_k)\mathrm{d}t,\\
\mathcal{R}_{\eta,6}=&\frac{3\eta^{\frac{5}{2}}}{2}\sum_{k=0}^{m-1}\int_0^1\left(1-t\right)^2\left(\nabla_{b(\theta_k)}\nabla_{\sigma(\theta_k)\xi_{k+1}}\nabla_{\sigma(\theta_k)\xi_{k+1}}f_h(\theta_k+t\triangle\theta_k)\right.\\&\left.\qquad\qquad\qquad\qquad\qquad\qquad\qquad\qquad+\sqrt{\eta}\nabla_{b(\theta_k)}\nabla_{b(\theta_k)}\nabla_{\sigma(\theta_k)\xi_{k+1}}f_h(\theta_k+t\triangle\theta_k)\right)\mathrm{d}t,\\
\mathcal{R}_{\eta,7}=&\frac{1}{2}\eta^{\frac{3}{2}}\sum_{k=0}^{m-1}\big<\nabla f_h(\theta_k),\mathcal{R}(\theta_k,\xi_{k+1})\big>,\\
\mathcal{R}_{\eta,8}=&\frac{1}{4}\eta^{\frac{5}{2}}\sum_{k=0}^{m-1}\Big(\big<\nabla^2f_h(\theta_k),b(\theta_k)\big(\mathcal{R}(\theta_k,\xi_{k+1})\big)^{\rm T}\big>_{\rm{HS}}+\big<\nabla^2f_h(\theta_k),\mathcal{R}(\theta_k,\xi_{k+1})\big(b(\theta_k)\big)^{\rm T}\big>_{\rm{HS}}\Big),\\
\mathcal{R}_{\eta,9}=&\frac{1}{4}\eta^2\sum_{k=0}^{m-1}\Big(\big<\nabla^2f_h(\theta_k),\sigma(\theta_k)\xi_{k+1}\big(\mathcal{R}(\theta_k,\xi_{k+1})\big)^{\rm T}\big>_{\rm{HS}}\\&\qquad\qquad\qquad\qquad\qquad\qquad\qquad\qquad+\big<\nabla^2f_h(\theta_k),\mathcal{R}(\theta_k,\xi_{k+1})\big(\sigma(\theta_k)\xi_{k+1}\big)^{\rm T}\big>_{\rm{HS}}\Big),\\
\mathcal{R}_{\eta,10}=&\frac{1}{8}\eta^{\frac{5}{2}}\sum_{k=0}^{m-1}\big<\nabla^2f_h(\theta_k),\mathcal{R}(\theta_k,\xi_{k+1})\big(\mathcal{R}(\theta_k,\xi_{k+1})\big)^{\rm T}\big>_{\rm{HS}}\\
&+\frac{1}{16}\eta^{\frac{7}{2}}\sum_{k=0}^{m-1}\int_{0}^{1}\left(1-t\right)^2\nabla_{\mathcal{R}(\theta_k,\xi_{k+1})}\nabla_{\mathcal{R}(\theta_k,\xi_{k+1})}\nabla_{\mathcal{R}(\theta_k,\xi_{k+1})}f_h\big(\theta_k+t\triangle\theta_k\big)\mathrm{d}t,\\
\mathcal{R}_{\eta,11}=&\frac{3}{4}\eta^{\frac{7}{2}}\sum_{k=0}^{m-1}\int_{0}^{1}\left(1-t\right)^2\Big(\nabla_{b(\theta_k)}\nabla_{b(\theta_k)}\nabla_{\mathcal{R}(\theta_k,\xi_{k+1})}f_h\big(\theta_k+t\triangle\theta_k\big)\\&\qquad\qquad\qquad\qquad\qquad\qquad\qquad+\frac{1}{2}\nabla_{b(\theta_k)}\nabla_{\mathcal{R}(\theta_k,\xi_{k+1})}\nabla_{\mathcal{R}(\theta_k,\xi_{k+1})}f_h\big(\theta_k+t\triangle\theta_k\big)\Big)\mathrm{d}t,\\
\mathcal{R}_{\eta,12}=&\frac{3}{2}\eta^3\sum_{k=0}^{m-1}\int_{0}^{1}\left(1-t\right)^2\nabla_{b(\theta_k)}\nabla_{\sigma(\theta_k)\xi_{k+1}}\nabla_{\mathcal{R}(\theta_k,\xi_{k+1})}f_h\big(\theta_k+t\triangle\theta_k\big)\mathrm{d}t,\\
\mathcal{R}_{\eta,13}=&\frac{3}{4}\eta^{\frac{5}{2}}\sum_{k=0}^{m-1}\int_{0}^{1}\left(1-t\right)^2\Big(\nabla_{\sigma(\theta_k)\xi_{k+1}}\nabla_{\sigma(\theta_k)\xi_{k+1}}\nabla_{\mathcal{R}(\theta_k,\xi_{k+1})}f_h\big(\theta_k+t\triangle\theta_k\big)\\
&\qquad\qquad\qquad\qquad\qquad\qquad+\frac{1}{2}\sqrt{\eta}\nabla_{\sigma(\theta_k)\xi_{k+1}}\nabla_{\mathcal{R}(\theta_k,\xi_{k+1})}\nabla_{\mathcal{R}(\theta_k,\xi_{k+1})}f_h\big(\theta_k+t\triangle\theta_k\big)\Big)\mathrm{d}t.
\end{align*}

\subsection{Preliminary propositions}

In this subsection, we present several key propositions that play crucial roles in our analysis. First, under {\bf Assumption} {\bf\ref{conditions}}, by Lemma 2.3 in \cite{Lu}, SDE \eqref{SDE} is ergodic with a unique invariant measure $\pi$.

Let $\E_{\theta_k}(\cdot)$ and $\P_{\theta_k}(\cdot)$ be respectively the conditional expectation $\E(\cdot|\theta_k)$ and conditional probability $\P(\cdot|\theta_k)$. For the ergodicity of the process $\theta_{k}$, we have the following lemma.

\begin{lem}\label{ergo}
 Choose the Lyapunov function $V(x)=1+\Arrowvert x\Arrowvert_2^2,\,x\in \R^d$, and suppose {\bf Assumption} {\bf\ref{conditions}} holds. If $\theta_0=x_0$ is fixed, we denote the process $(\theta_k)_{k\geq 0}$ by $(\theta_k^{x_0})_{k\geq 0}$. Then $(\theta_{k}^{x_0})_{k\geq0}$ is ergodic with a unique invariant measure $\pi_{\eta}$ such that
\begin{align}\label{CO2}
\sup_{|h|\leq V}\left|\mathbb{E}h(\theta_{k}^{x_0})-\pi_{\eta}(h)\right|\leq u_{1}\left(V(x_0)+\pi_{\eta}(V)\right)(1+\eta^{-1})e^{-u_{2}\eta},
\end{align}
for some positive constants $u_1$ and $u_2$ independent of $\eta$. In addition, for any $k\geq0$ and $x_0\in\mathbb{R}^{d}$, there exists a constant $C>0$ which is independent of $\eta$ such that 
\begin{align*}
\E\big(\Arrowvert\theta_k^{x_0}\Arrowvert_2^4\big)\leq C\big(1+\Arrowvert x_0\Arrowvert_2^{4}\big).
\end{align*}
\end{lem}

\begin{proof}
    We can write by \eqref{Milstein},
$$\delta=\theta_{k+1}-\theta_k=\eta b(\theta_k)+\sqrt{\eta}\sigma(\theta_k)\xi_{k+1}+\frac{1}{2}\eta\mathcal{R}(\theta_k,\xi_{k+1}).$$
Since $\nabla V(x)=2x$ and $\nabla^2V(x)=2I_d$, we have
\begin{align*}
&\E_{\theta_k} \big(V(\theta_{k+1})\big)-V(\theta_k)\\
=&\E_{\theta_k}\Arrowvert\theta_{k+1}\Arrowvert_2^2-\Arrowvert\theta_k\Arrowvert_2^2=\E_{\theta_k}\Arrowvert\delta+\theta_k\Arrowvert_2^2-\Arrowvert \theta_k\Arrowvert_2^2=2\E_{\theta_k}\left<\theta_k,\delta\right>+\E_{\theta_k}\delta\delta^{\rm T}\\
=&\eta\left(\left\langle b(\theta_k),2\theta_k\right\rangle+\left\langle\sigma(\theta_k)\sigma(\theta_k)^{\rm T},I_d\right\rangle_{\rm{HS}}\right)+\eta^2\Arrowvert b(\theta_k)\Arrowvert_2^2+\frac{1}{2}\eta^2\Arrowvert\sigma(\theta_k)\Arrowvert^2\Arrowvert\nabla\sigma(\theta_k)\Arrowvert^2\\
=&\eta\mathcal{A}V(\theta_k)+\eta^2\Arrowvert b(\theta_k)\Arrowvert_2^2+\frac{1}{2}\eta^2\Arrowvert\sigma(\theta_k)\Arrowvert^2\Arrowvert\nabla\sigma(\theta_k)\Arrowvert^2,
\end{align*}
and \eqref{|x|^2} further implies that
\begin{align*}
\mathcal{A}V(\theta_k)=&2\left\langle b(\theta_k),\theta_k\right\rangle+\frac{1}{2}\left\langle\sigma(\theta_k)\sigma(\theta_k)^{\rm T},2I_d\right\rangle_{\rm{HS}}
\\\leq&2\left(-\frac{K_1}{2}\Arrowvert \theta_k\Arrowvert_2^2+\left(K_2+\frac{1}{2K_1}\Arrowvert b(0)\Arrowvert_2^2\right)\right)+\Arrowvert\sigma\Arrowvert_{\infty}^2
\leq-\frac{K_{1}}{2}V(\theta_k)+C_1I_A(\theta_k),
\end{align*}
where $C_1=\frac{K_{1}}{2}+2K_2+\frac{1}{K_1}\Arrowvert b(0)\Arrowvert_2^2+\Arrowvert\sigma\Arrowvert_{\infty}^2$ and $A=\left\{x:\Arrowvert x\Arrowvert_2^2\leq\frac{2C_2}{K_1}-1\right\}$ with $C_{2}=K_{1}+2K_2+\frac{1}{K_1}\Arrowvert b(0)\Arrowvert_2^2+\Arrowvert\sigma\Arrowvert_{\infty}^2$.
Thus there exists a positive constant $C_3>\frac{K_{1}}{2}$ independent of $\eta$ such that for sufficiently small $\eta>0$,
\begin{align*}
&\E_{\theta_k} \big(V(\theta_{k+1})\big)-V(\theta_k)\\
\leq&\eta\left(-K_1V(\theta_k)+C_2\right)+\eta^2\Arrowvert b(\theta_k)\Arrowvert_2^2+\frac{1}{2}\eta^2\Arrowvert\sigma(\theta_k)\Arrowvert^2\Arrowvert\nabla\sigma(\theta_k)\Arrowvert^2\\
\leq&-K_1\eta V(\theta_k)+C_2\eta+\eta^2\left(2L^2\Arrowvert\theta_k\Arrowvert_2^2+2\Arrowvert b(0)\Arrowvert_2^2\right)+\frac{1}{2}\eta^2\Arrowvert\sigma\Arrowvert_{\infty}^2\Arrowvert\nabla\sigma\Arrowvert_{\infty}^2\\
=&\left(-K_1\eta+2L^2\eta^2\right)V(\theta_k)+C_2\eta+\left(2\Arrowvert b(0)\Arrowvert_2^2+\frac{1}{2}\Arrowvert\sigma\Arrowvert_{\infty}^2\Arrowvert\nabla\sigma\Arrowvert_{\infty}^2-2L^2\right)\eta^2\\
\leq&-\frac{K_1}{2}\eta V(\theta_k)+C_3\eta\leq-\frac{K_1}{4}\eta V(\theta_k)+C_3\eta I_B(\theta_k),
\end{align*}
where $C_{3}=C_{2}+2\Arrowvert b(0)\Arrowvert_2^2+\frac{1}{2}\Arrowvert\sigma\Arrowvert_{\infty}^2\Arrowvert\nabla\sigma\Arrowvert_{\infty}^2$ and $B=\left\{x:\Arrowvert x\Arrowvert_2^2\leq\frac{4C_3}{K_1}-1\right\}$.
So we have
\begin{align}\label{CO1}
\E_{\theta_k} \big(V(\theta_{k+1})\big)\leq\left(1-\frac{K_1}{4}\eta\right)V(\theta_k)+C_3\eta I_B(\theta_k).
\end{align}
By (29) in \cite{RT96}, we deduce that for sufficiently small $\eta>0$, $\theta_k$ is ergodic with a unique invariance measure. Next by the same argument as the proof of (A.3) in \cite{Xing}, one can derive from \eqref{CO1} that
\begin{align*}
\sup_{|h|\leq V}\left|\mathbb{E}h(\theta_{k}^{x_0})-\pi_{\eta}(h)\right|\leq u_{1}\left(V(x_0)+\pi_{\eta}(V)\right)(1+\eta^{-1})e^{-u_{2}\eta}
\end{align*}
for some positive constants $u_1$ and $u_2$ independent of $\eta$.

Moreover, let $\widetilde{V}(x)=\Arrowvert x\Arrowvert_2^4+1$, according to (\ref{Milstein}), we can get
\begin{align*}
&\E_{\theta_k}\left(\widetilde{V}(\theta_{k+1})\right)=\E_{\theta_k}\left(\Arrowvert\theta_k+\eta b(\theta_k)+\sqrt{\eta}\sigma(\theta_k)\xi_{k+1}+\frac{1}{2}\eta\mathcal{R}(\theta_k,\xi_{k+1})\Arrowvert_2^4+1\right)\\
=&\Arrowvert\theta_k+\eta b(\theta_k)\Arrowvert_2^4+\E_{\theta_k}\Arrowvert\sqrt{\eta}\sigma(\theta_k)\xi_{k+1}+\frac{1}{2}\eta\mathcal{R}(\theta_k,\xi_{k+1})\Arrowvert_2^4\\
&+6\Arrowvert\theta_k+\eta b(\theta_k)\Arrowvert_2^2\E_{\theta_k}\Arrowvert\sqrt{\eta}\sigma(\theta_k)\xi_{k+1}+\frac{1}{2}\eta\mathcal{R}(\theta_k,\xi_{k+1})\Arrowvert_2^2\\
&+4\Arrowvert\theta_k+\eta b(\theta_k)\Arrowvert_2\E_{\theta_k}\Arrowvert\sqrt{\eta}\sigma(\theta_k)\xi_{k+1}+\frac{1}{2}\eta\mathcal{R}(\theta_k,\xi_{k+1})\Arrowvert_2^3\\
\leq& \left(1-2K_1\eta+c_1\eta^2\right)\Arrowvert\theta_k\Arrowvert_2^4+c_2\eta\Arrowvert\theta_k\Arrowvert_2^2+c_3\eta^{3/2}\Arrowvert\theta_k\Arrowvert_2+c_4\eta^2+1\\
\leq& \left(1-K_1\eta+c_1\eta^2\right)\widetilde{V}(\theta_k)-K_1\eta\Arrowvert\theta_k\Arrowvert_2^4+c_2\eta\Arrowvert\theta_k\Arrowvert_2^2+c_4\eta^2+K_1\eta,
\end{align*}
where the last inequality comes from $x\leq x^2+\frac{1}{4}$ and $c_1,c_2,c_3,c_4$ depend on $\sigma$, $\Arrowvert b(0)\Arrowvert_2$, $K_1$ and $L$. Then for any sufficiently small $\eta>0$, we have
\begin{align}\label{|x|^4+1}
\E_{\theta_k}\left(\widetilde{V}(\theta_{k+1})\right)\leq\left(1-\frac{K_1}{2}\eta\right)\widetilde{V}(\theta_k)+\tilde{b}\eta,
\end{align}
where $\tilde{b}=\frac{c_2^2}{4K_1}+K_1+c_4$. Thus by induction, we have
\begin{align*}
\E_{\theta_k}\left(\widetilde{V}(\theta_{k+1})\right)\leq \left(1-\frac{K_1}{2}\eta\right)^{k+1}\widetilde{V}(x_0)+\tilde{b}\eta\sum_{i=0}^{k}\left(1-\frac{K_1}{2}\eta\right)^{i}
\leq(1+\Arrowvert x_0\Arrowvert_2^{4})+\frac{2\tilde{b}}{K_{1}},
\end{align*} 
and the desired result follows.
\end{proof}

We now establish the following auxiliary proposition.
\begin{pro}\label{EY}
Suppose that {\bf Assumptions} {\bf\ref{conditions}} and {\bf\ref{conditions2}} hold. For any $k\in\mathbb{N}$, it holds for all $y>0$,
\begin{align*}
&\P\left(\left\lvert\sum_{i=0}^{k-1}\Arrowvert\sigma(\theta_i)^{\rm T}\nabla f_{h}(\theta_i)\Arrowvert_2^2-\sum_{i=0}^{k-1}\E\Arrowvert\sigma(\theta_i)^{\rm T}\nabla f_{h}(\theta_i)\Arrowvert_2^2\right\rvert\geq y\right)\\
\leq&
2e^{-cy^2k^{-1}}+c_3\left(1+\eta^{-1}\right)\left(e^{-c_1\eta y}+\frac{e^{-c_2y}}{1-e^{-u_1\eta}}\right),
\end{align*}
here $u_1$, $c_1$, $c_2$ and $c_3$ are positive constants and both are independent of the variables.
\end{pro}
\begin{proof}
    Recall the proof of Lemma \ref{ergo}, we can obtain $\lambda=1-\frac{K_1}{4}\eta$ and $b=C_3\eta$ in (A1) of \cite{Durmus}. Combining the set $A=\left\{x:\Arrowvert x\Arrowvert_2^2\leq\frac{2C_2}{K_1}-1\right\}$ and set $B=\left\{x:\Arrowvert x\Arrowvert_2^2\leq\frac{4C_3}{K_1}-1\right\}$, we can take $d=1+\eta$ in (A2) of \cite{Durmus}, i.e. set $\ell=\left\{x\in\R^d:\Arrowvert x\Arrowvert_2^2\leq\eta\right\}$. Next by \eqref{CO2} and \eqref{CO1}, we obtain the compact set $\ell$ is petite and take $m=1$ and $\epsilon$ be a constant small enough, by (11) in \cite{Durmus} we can obtain $\rho=e^{-u_2\eta}$ and $c=u_1\left(1+\eta^{-1}\right)$ for some positive constants $u_1$ and $u_2$ independent of $\eta$. Then, let $\gamma=0$ and $\theta_0$ satisfy {\bf Assumption} {\bf\ref{conditions2}}, by Theorem 6 in \cite{Durmus}, we have for all $y>0$ that
\begin{align*}
&\P\left(\left\lvert\sum_{i=0}^{k-1}\Arrowvert\sigma(\theta_i)^{\rm T}\nabla f_h(\theta_i)\Arrowvert_2^2-\sum_{i=0}^{k-1}\E\Arrowvert\sigma(\theta_i)^{\rm T}\nabla f_h(\theta_i)\Arrowvert_2^2\right\rvert\geq y\right)\\
\leq&
\P_{\pi_{\eta}}\left(\left\lvert\sum_{i=0}^{k-1}\Arrowvert\sigma(\theta_i)^{\rm T}\nabla f_h(\theta_i)\Arrowvert_2^2-k\pi_\eta\big(\Arrowvert\sigma^{\rm T}\nabla f_h\Arrowvert_2^2\big)\right\rvert\geq y/4\right)
\\&+c\left(e^{-c_1\eta y}+\frac{e^{-c_2y}}{1-e^{-u_1\eta}}\right)u_2\left(1+\eta^{-1}\right)\left(\E\left( V(\theta_0)\right)+\pi_{\eta}(V)\right),
\end{align*}
where $\P_{\pi_{\eta}}(\cdot)$ denotes the probability measure when $\theta_0\sim\pi_{\eta}$. By \eqref{CO1} for $\theta_0\sim\pi_{\eta}$, we have $\sup_{\R}\pi_{\eta}(V)\leq \frac{4C_3}{K_1}$, then together with Lemma 3.3 in \cite{Fan}, we have
\begin{align*}
    &\P\left(\left\lvert\sum_{i=0}^{k-1}\Arrowvert\sigma(\theta_i)^{\rm T}\nabla f_h(\theta_i)\Arrowvert_2^2-\sum_{i=0}^{k-1}\E\Arrowvert\sigma(\theta_i)^{\rm T}\nabla f_h(\theta_i)\Arrowvert_2^2\right\rvert\geq y\right)\\
    \leq& 2e^{-cy^2k^{-1}}+c_3\left(1+\eta^{-1}\right)\left(e^{-c_1\eta y}+\frac{e^{-c_2y}}{1-e^{-u_1\eta}}\right).
\end{align*}
This completes the proof of Proposition \ref{EY}.
\end{proof}

\begin{lem}\label{E_k}
Let $\Psi_1:\R^d\to\R^d$ and $\Psi_2:\R^{2d}\to\R$ be measurable functions. Denote $A=\{\left|\xi_i\right|\leq \mathbf{R},i=1,\cdots,m\}$, we have
\begin{align*}
&\E\exp\left\{\left(\sum_{k=0}^{m-1}\big(\left<\Psi_1(\theta_k),\sigma(\theta_k)\xi_{k+1}\right>+\Psi_2(\theta_k,\xi_{k+1})\big)\right)I_A\right\}\\
\leq&\left(\E\exp\left\{\left(\sum_{k=0}^{m-1}2\left(\left\Arrowvert\Psi_1(\theta_k)\right\Arrowvert_2^2\Arrowvert\sigma\Arrowvert_{\infty}^2+\Psi_2(\theta_k,\xi_{k+1})\right)\right)I_A\right\}\right)^{\frac{1}{2}}
\end{align*}
and
\begin{align*}
&\E_{\theta_0}\exp\left\{\left(\sum_{k=0}^{m-1}\big(\left<\Psi_1(\theta_k),\sigma(\theta_k)\xi_{k+1}\right>+\Psi_2(\theta_k,\xi_{k+1})\big)\right)I_A\right\}\\
\leq&\left(\E_{\theta_0}\exp\left\{\left(\sum_{k=0}^{m-1}2\left(\left\Arrowvert\Psi_1(\theta_k)\right\Arrowvert_2^2\Arrowvert\sigma\Arrowvert_{\infty}^2+\Psi_2(\theta_k,\xi_{k+1})\right)\right)I_A\right\}\right)^{\frac{1}{2}}.
\end{align*}
\end{lem}

\begin{proof}
    By H\"older's inequality, we can get
\begin{align*}
&\E\exp\left\{\left(\sum_{k=0}^{m-1}\big(\left<\Psi_1(\theta_k),\sigma(\theta_k)\xi_{k+1}\right>+\Psi_2(\theta_k,\xi_{k+1})\big)\right)I_A\right\}\\
=&\E\exp\left\{\Bigg(\sum_{k=0}^{m-1}\big(\left<\Psi_1(\theta_k),\sigma(\theta_k)\xi_{k+1}\right>-\big\Arrowvert\sigma(\theta_k)^{\rm T}\Psi_1(\theta_k)\big\Arrowvert_2^2\right.\\&\left.\qquad\qquad\qquad\qquad\qquad\qquad\qquad\;+\big\Arrowvert\sigma(\theta_k)^{\rm T}\Psi_1(\theta_k)\big\Arrowvert_2^2+\Psi_2(\theta_k,\xi_{k+1})\big)\Bigg)I_A\right\}\\
\leq&\left(\E\exp\left\{\left(\sum_{k=0}^{m-1}2\big(\left<\Psi_1(\theta_k),\sigma(\theta_k)\xi_{k+1}\right>-\big\Arrowvert\sigma(\theta_k)^{\rm T}\Psi_1(\theta_k)\big\Arrowvert_2^2\big)\right)I_A\right\}\right)^{\frac{1}{2}}\\
&\cdot\left(\E\exp\left\{\left(\sum_{k=0}^{m-1}2\big(\big\Arrowvert\sigma(\theta_k)^{\rm T}\Psi_1(\theta_k)\big\Arrowvert_2^2+\Psi_2(\theta_k,\xi_{k+1})\big)\right)I_A\right\}\right)^{\frac{1}{2}}\\
\leq&\left(\E\exp\left\{\left(\sum_{k=0}^{m-1}2\big(\Arrowvert\sigma(\theta_k)^{\rm T}\Psi_1(\theta_k)\big\Arrowvert_2^2+\Psi_2(\theta_k,\xi_{k+1})\big)\right)I_A\right\}\right)^{\frac{1}{2}},
\end{align*}
where the last inequality is by a standard conditional argument as follows: let $\E_{m-1}$ be the conditional expectation $\E(\cdot|\theta_0,\xi_1,\cdots,\xi_{m-1})$,
since $\xi_{k+1}$ is gaussian distributed and independent of $\theta_k$, a straightforward calculation gives
\begin{align*}
&\E\exp\left\{\left(\sum_{k=0}^{m-1}2\big(\left<\Psi_1(\theta_k),\sigma(\theta_k)\xi_{k+1}\right>-\big\Arrowvert\sigma(\theta_k)^{\rm T}\Psi_1(\theta_k)\big\Arrowvert_2^2\big)\right)I_A\right\}\\
\leq&\E\exp\left\{\sum_{k=0}^{m-1}2\big(\left<\Psi_1(\theta_k),\sigma(\theta_k)\xi_{k+1}\right>-\big\Arrowvert\sigma(\theta_k)^{\rm T}\Psi_1(\theta_k)\big\Arrowvert_2^2\big)\right\}\\
=&\E\left(\exp\left\{\sum_{k=0}^{m-2}2\big(\left<\Psi_1(\theta_k),\sigma(\theta_k)\xi_{k+1}\right>-\big\Arrowvert\sigma(\theta_k)^{\rm T}\Psi_1(\theta_k)\big\Arrowvert_2^2\big)\right\}\right.\\&\left.\qquad\qquad\E_{m-1}\left(e^{2\left(\left<\Psi_1(\theta_{m-1}),\sigma(\theta_{m-1})\xi_m\right>-\left\Arrowvert\sigma(\theta_{m-1})^{\rm T}\Psi_1(\theta_{m-1})\right\Arrowvert_2^2\right)}\right)\right)\\
=&\E\left(\exp\left\{\sum_{k=0}^{m-2}2\big(\left<\Psi_1(\theta_k),\sigma(\theta_k)\xi_{k+1}\right>-\big\Arrowvert\sigma(\theta_k)^{\rm T}\Psi_1(\theta_k)\big\Arrowvert_2^2\big)\right\}\right).
\end{align*}
By taking the conditional expectations successively,
\begin{align*}
    \E\exp\left\{\sum_{k=0}^{m-1}2\big(\left<\Psi_1(\theta_k),\sigma(\theta_k)\xi_{k+1}\right>-\big\Arrowvert\sigma(\theta_k)^{\rm T}\Psi_1(\theta_k)\big\Arrowvert_2^2\big)\right\}=1.
\end{align*}
A similar calculation gives the second inequality. Thus the desired result is established.
\end{proof}

\begin{pro}\label{g}
Suppose that {\bf Assumption} {\bf\ref{conditions}} holds and denote $A=\{\left|\xi_i\right|\leq \mathbf{R},i=1,\cdots,m\}$. Then there exist positive constants $c$ and $c_{\cdot}$ depending on $L, K_1, K_2, \Arrowvert b(0)\Arrowvert_2$ and $\sigma$, whose values may vary from line to line, such that for any $a>0$ and $\mathbf{R}$ satisfying
\begin{align}\label{R}
    c_1\eta^{-a}<\mathbf{R}<c_2\Arrowvert\nabla\sigma\Arrowvert^{-1}_{\infty}\eta^{-\frac{1}{2}},
\end{align}
we have
$$
\E_{\theta_0}\left(\exp\left\{\left(c\eta\sum_{k=0}^{m-1}\Arrowvert b(\theta_k)\Arrowvert_2^2\right)I_A\right\}\right)\leq Ce^{\left(c_1\Arrowvert\theta_0\Arrowvert_2^2+c_2\Arrowvert\sigma\Arrowvert_{\infty}^2\eta^{-1}\left(1+\Arrowvert\nabla\sigma\Arrowvert_{\infty}^2\left(\mathbf{R}^2+1\right)\right)\right)},
$$
and for all $x>0$,
\begin{align}\label{gP0}
    \P_{\theta_0}\left(\eta\sum_{k=0}^{m-1}\Arrowvert b(\theta_k)\Arrowvert_2^2>x\right)\leq Ce^{\left(c_1\Arrowvert\theta_0\Arrowvert_2^2+c_2\Arrowvert\sigma\Arrowvert_{\infty}^2\eta^{-1}\left(1+\Arrowvert\nabla\sigma\Arrowvert_{\infty}^2\left(\mathbf{R}^2+1\right)\right)\right)}e^{-c_3x}+c_4e^{-\frac{\mathbf{R}^2}{2}}.
\end{align}
Moreover, if {\bf Assumption} {\bf\ref{conditions2}} holds, we can obtain
$$
\E\left(\exp\left\{\left(c\eta\sum_{k=0}^{m-1}\Arrowvert b(\theta_k)\Arrowvert_2^2\right)I_A\right\}\right)\leq c_1e^{c_2\Arrowvert\sigma\Arrowvert_{\infty}^2\eta^{-1}\left(1+\Arrowvert\nabla\sigma\Arrowvert_{\infty}^2\left(\mathbf{R}^2+1\right)\right)},
$$
and for all $x>0$,
\begin{align}\label{gP}
    \P\left(\eta\sum_{k=0}^{m-1}\Arrowvert b(\theta_k)\Arrowvert_2^2>x\right)\leq c_1e^{c_2\Arrowvert\sigma\Arrowvert_{\infty}^2\eta^{-1}\left(1+\Arrowvert\nabla\sigma\Arrowvert_{\infty}^2\left(\mathbf{R}^2+1\right)\right)}e^{-c_3x}+c_4e^{-\frac{\mathbf{R}^2}{2}}.
\end{align}
\end{pro}
\begin{proof}
    See Appendix \ref{Appendix}.
\end{proof}

In order to estimate the tail probability of $\mathcal{R}_{\eta}$, we also need Lemma 3.4 in Fan, Hu and Xu \cite{Fan}.
\begin{pro}\label{martingale differences}
Let $(\zeta_i,\mathcal{F}_i)_{i\geq1}$ be a sequence of martingale differences. Assume there exist positive constants $c$ and $\alpha\in(0,1]$ such that
$$
1\leq u_n:=\sum_{i=1}^{n}\big\Arrowvert\E\big(\zeta_i^2\exp\{c\lvert\zeta_i\rvert^{\alpha}\}|\mathcal{F}_{i-1}\big)\big\Arrowvert_{\infty}<\infty.
$$
Then there exists a positive constant $c_{\alpha}$ such that for all $x>0$,
$$
\P\left(\sum_{i=1}^{n}\zeta_i\geq x\right)\leq c_{\alpha}\exp\left\{-\frac{x^2}{c_{\alpha}\big(u_n+x^{2-\alpha}\big)}\right\}.
$$
\end{pro}

Then, we can establish the following deviation inequality for the martingale difference $\big(\psi_i,\mathcal{F}_i\big)_{i\geq1}$, where $\psi_{i+1}=\big\Arrowvert\big(\sigma(\theta_i)\xi_{i+1}\big)^{\rm T}\nabla f_h(\theta_i)\big\Arrowvert_2^2-\big\Arrowvert\sigma(\theta_i)^{\rm T}\nabla f_h(\theta_i)\big\Arrowvert_2^2$ and $\mathcal{F}_i=\sigma\big(\theta_0,\xi_k,1\leq k\leq i\big)$.
\begin{pro}\label{VY}
Under {\bf Assumption} {\bf\ref{conditions}}, it holds for all $y>0$,
\begin{align*}
&\P\left(\left\lvert\sum_{i=0}^{[\eta^{-2}]-1}\Big(\eta^{-1}\big\Arrowvert\big(\theta_{i+1}-\theta_i-\eta b(\theta_i)-\frac{1}{2}\eta\mathcal{R}(\theta_{i},\xi_{i+1})\big)^{\rm T}\nabla f_{h}(\theta_i)\big\Arrowvert_2^2-\big\Arrowvert\sigma(\theta_i)^{\rm T}\nabla f_{h}(\theta_i)\big\Arrowvert_2^2\Big)\right\rvert>y\right)\\
\leq& c_1\exp\left\{-\frac{y^2}{c_1\big(\eta^{-2}+cy\big)}\right\},
\end{align*}
where $c_1$ and $c$ depend on $b$ and $\sigma$.
\end{pro}

\begin{proof}
By \eqref{boundedness}, we deduce that
\begin{align*}
\big\lvert\psi_{k+1}\big\rvert
&\leq\Arrowvert\nabla f_h(\theta_k)\Arrowvert_2^2\cdot\Arrowvert(\sigma(\theta_k)\xi_{k+1})(\sigma(\theta_k)\xi_{k+1})^{\rm T}-\sigma(\theta_k)\sigma(\theta_k)^{\rm T}\Arrowvert\\
&\leq c\big(1+\Arrowvert\xi_{k+1}\Arrowvert_2^2\big),
\end{align*}
where $c=\Arrowvert\nabla f_h(\theta_k)\Arrowvert_2^2\Arrowvert\sigma(\theta_k)\Arrowvert^2$. Then, there exists some positive constant $c$ such that
$$
\E\big(\psi_{k+1}^2\exp{\big\{c\lvert \psi_{k+1}\rvert\big\}}\big\lvert\mathcal{F}_k\big)<\infty.
$$
Therefore, by \eqref{Milstein} and Proposition \ref{martingale differences} with $\alpha=1$, we have for all $y>0$,
\begin{align*}
&\P\left(\Bigg\lvert\sum_{i=0}^{[\eta^{-2}]-1}\Big(\eta^{-1}\big\Arrowvert\big(\theta_{i+1}-\theta_i-\eta b(\theta_i)-\frac{1}{2}\eta\mathcal{R}(\theta_i,\xi_{i+1})\big)^{\rm T}\nabla f_h(\theta_i)\big\Arrowvert_2^2-\big\Arrowvert\sigma(\theta_i)^{\rm T}\nabla f_h(\theta_i)\big\Arrowvert_2^2\Big)\Bigg\rvert>y\right)\\
=&\P\left(\Bigg\lvert\sum_{i=0}^{[\eta^{-2}]-1}\Big(\big\Arrowvert\big(\sigma(\theta_i)\xi_{i+1}\big)^{\rm T}\nabla f_h(\theta_i)\big\Arrowvert_2^2-\big\Arrowvert\sigma(\theta_i)^{\rm T}\nabla f_h(\theta_i)\big\Arrowvert_2^2\Big)\Bigg\rvert>y\right)\\
\leq&\P\left(\Bigg\lvert\sum_{i=0}^{[\eta^{-2}]-1}\psi_{i+1}\Bigg\rvert>y\right)
\leq c\exp\left\{-\frac{y^2}{c_1\big(\eta^{-2}+cy\big)}\right\},
\end{align*}
which completes the proof of Proposition \ref{VY}.
\end{proof}

Now, we present the deviation inequality for $\mathcal{R}_{\eta}$, whose detailed proof is deferred to the Appendix \ref{AppendixB}.
\begin{pro}\label{remainder}
    Suppose that {\bf Assumptions} {\bf\ref{conditions}} and {\bf\ref{conditions2}} hold and for any $a>0$, we choose $c_1\eta^{-a}<\mathbf{R}<c\Arrowvert\nabla\sigma\Arrowvert_{\infty}^{-1}\eta^{-1/3}$. Then for all
    $$\max\big\{c\eta^{1/2},\eta^{1/2}\mathbf{R}^3\Arrowvert\nabla\sigma\Arrowvert_{\infty}^3\big\}<y=o\big(\eta^{-1/2}\big),$$
    we have
    \begin{align}\nonumber
    \P\Big(\big\lvert\mathcal{R}_\eta\big\rvert\geq y\Big)\leq c_1\left(\exp\left\{-\frac{y^2\eta^{-1}}{c_2(1+\eta^{1/2}y)}\right\}+\exp\left\{-\frac{\mathbf{R}^2}{2}\right\}+\exp\left\{-c_3y^{1/3}\eta^{-7/6}\right\}\right).
    \end{align}
\end{pro}

\begin{rmk}
In comparison to the Euler-Maruyama scheme \eqref{EM}, the Milstein scheme \eqref{Milstein} incorporates an additional term $\mathcal{R}(\theta_k,\xi_{k+1})$, originating from $\nabla_{\sigma(\theta_k)\xi_{k+1}}\sigma(\theta_k)\not\equiv0$ in Assumption {\bf\ref{conditions}}. As a result, this extra term not only complicates the structure of the remainder but, more importantly, leads to a heavy-tailed distribution when proving the Proposition \ref{g} which is a key step in deriving the remainder's upper bound. More precisely, the proof of Proposition \ref{g} relies on a truncation of the standard normal random vectors $(\xi_k)_{k\geq1}$ via the event $A=\left\{|\xi_i|\leq\mathbf{R},i=1,\cdots,m\right\}$. Such truncation is essential to bound exponential moment expression of the form $\E_{\theta_0}\exp\left\{c\sum_{k=0}^{m-1}\eta\Arrowvert\mathcal{R}(\theta_k,\xi_{k+1})\Arrowvert_2^2\right\}$, which inherently involves the term $\E_{\theta_0}\exp\left\{c\sum_{k=0}^{m-1}\eta\Arrowvert\xi_{k+1}\Arrowvert_2^4\right\}$. Thus, compared with Lemma 3.2 in \cite{Fan}, this truncation deteriorates the outcome of Proposition \ref{g}, i.e.
\begin{align*}
    \P_{\theta_0}\left(\eta\sum_{k=0}^{m-1}\Arrowvert b(\theta_k)\Arrowvert_2^2>x\right)\leq Ce^{\left(c_1\Arrowvert\theta_0\Arrowvert_2^2+c_2\Arrowvert\sigma\Arrowvert_{\infty}^2\eta^{-1}\left(1+\Arrowvert\nabla\sigma\Arrowvert_{\infty}^2\left(\mathbf{R}^2+1\right)\right)\right)}e^{-c_3x}+c_4e^{-\frac{\mathbf{R}^2}{2}},
\end{align*}
which ultimately leading to a suboptimal probability estimates for the remainder terms. Here, by choosing $y=c\Arrowvert\nabla\sigma\Arrowvert_{\infty}^{3/4}\eta^{1/8}$, $\mathbf{R}=c\Arrowvert\nabla\sigma\Arrowvert_{\infty}^{-3/4}\eta^{-1/8}$, we have
\begin{align*}
\frac{\P\left(\mathcal{R}_\eta/\sqrt{\mathcal{Y}_\eta}\geq y\right)}{1-\Phi(x)}\leq Ce^{-c\Arrowvert\nabla\sigma\Arrowvert_{\infty}^{-3/2}\eta^{-1/4}}.
\end{align*}
Consequently, it leads to the upper bound in the Theorem \ref{W} and Theorem \ref{S}, and narrows the range of $x$ for which the result holds.
\end{rmk}

\section{Proof of Theorem \ref{CLT}}\label{ProofCLT}

We begin by stating a preparatory lemma which will enable us to study the convergence rate of the martingale $\mathcal{H}_{\eta}$.
\begin{lem}\label{convergence}
Suppose that {\bf Assumption} {\bf\ref{conditions}} holds. Let $(X_{t})_{t\geq0}$ and $(\theta_{k})_{k\geq0}$ be defined by \eqref{SDE} and \eqref{Milstein}, respectively. Then we have
\begin{align*}
\left|\pi\left(\left\Arrowvert\sigma^{T}\nabla f_{h}\right\Arrowvert_2^{2}\right)-\pi_{\eta}\left(\left\Arrowvert\sigma^{T}\nabla f_{h}\right\Arrowvert_2^{2}\right)\right|\leq C\eta.
\end{align*}
\end{lem}

\begin{proof}
Let the initial value $\theta_0\thicksim\pi_{\eta}$ and $\triangle\theta_0=\theta_1-\theta_0$. We have by (\ref{Milstein}) that
\begin{gather}\label{Lem4.1-1}
\E_{\theta_0}\big(\triangle\theta_0\big)=\eta b(\theta_0),\nonumber\\ \E_{\theta_0}\big(\triangle\theta_0(\triangle\theta_0)^{\rm T}\big)=\eta^2 b(\theta_0)b(\theta_0)^{\rm T}+\eta\sigma(\theta_0)\sigma(\theta_0)^{\rm T}+\frac{1}{4}\eta^2\E_{\theta_0}\left(\mathcal{R}(\theta_0,\xi_1)\mathcal{R}(\theta_0,\xi_1)^{\rm T}\right).
\end{gather}
Consider the Stein's equation
\begin{equation}\label{Lem4.1-2}
\left\Arrowvert\sigma^{\rm T}\nabla f_h\right\Arrowvert_2^2-\pi\left(\left\Arrowvert\sigma^{\rm T}\nabla f_h\right\Arrowvert_2^2\right)=\mathcal{A}\bar{f_h}.
\end{equation}
By \eqref{boundedness}, the test function $\left\Arrowvert\sigma^{\rm T}\nabla f_h\right\Arrowvert_2^2\in\mathcal{C}_b^2(\R^d,\R)$, and $\bar{f_h}$ exists and satisfies $\Arrowvert\nabla^k\bar{f_h}\Arrowvert\leq C$, $k=0,1,2,3,4$. Using the Taylor expansion and stationarity of $(\theta_k)_{k\geq 0}$ that
\begin{align}\label{Lem4.1-3}
0
=&\E\left(\bar{f_h}(\theta_1)-\bar{f_h}(\theta_0)\right)\nonumber\\
=&\E\left(\left\langle\nabla\bar{f_h}(\theta_0),\triangle\theta_0\right\rangle\right)+\frac{1}{2}\E\left(\left\langle\nabla^2\bar{f_h}(\theta_0),\triangle\theta_0(\triangle\theta_0)^{\rm T}\right\rangle_{\rm{HS}}\right)\nonumber\\&+\frac{1}{6}\E\left(\left\langle\left\langle\nabla^3\bar{f_h}(\theta_0),\triangle\theta_0(\triangle\theta_0)^{\rm T}\right\rangle_{\rm{HS}},\triangle\theta_0\right\rangle\right)\nonumber\\
&+\frac{1}{6}\int_0^1(1-t)^3\E\left(\nabla_{\triangle\theta_0}\nabla_{\triangle\theta_0}\nabla_{\triangle\theta_0}\nabla_{\triangle\theta_0}\bar{f_h}(\theta_0+t\triangle\theta_0)\right)\mathrm{d}t.
\end{align}
By (\ref{Lem4.1-1}), we have
$$
\E\left(\left\langle\nabla\bar{f_h}(\theta_0),\triangle\theta_0\right\rangle\right)=\E\left(\left\langle\nabla\bar{f_h}(\theta_0),\eta b(\theta_0)\right\rangle\right),
$$
and
\begin{align*}
&\E\left(\left\langle\nabla^2\bar{f_h}(\theta_0),\triangle\theta_0(\triangle\theta_0)^{\rm T}\right\rangle_{\rm{HS}}\right)\\
=&\E\left(\left\langle\nabla^2\bar{f_h}(\theta_0),\eta^2 b(\theta_0)b(\theta_0)^{\rm T}+\eta\sigma(\theta_0)\sigma(\theta_0)^{\rm T}+\frac{1}{4}\eta^2\E_{\theta_0}\big(\mathcal{R}(\theta_0,\xi_1)\mathcal{R}(\theta_0,\xi_1)^{\rm T}\big)\right\rangle_{\rm{HS}}\right).
\end{align*}
Together with (\ref{Af(x)}) and (\ref{Lem4.1-3}), it holds
\begin{align*}
\E\left(\mathcal{A}(\bar{f_h}(\theta_0))\right)=&-\frac{1}{2}\E\left(\left\langle\nabla^2\bar{f_h}(\theta_0),\eta b(\theta_0)b(\theta_0)^{\rm T}+\frac{1}{4}\eta\mathcal{R}(\theta_0,\xi_1)\mathcal{R}(\theta_0,\xi_1)^{\rm T}\right\rangle_{\rm{HS}}\right)\\
&-\frac{1}{6\eta}\E\left(\left\langle\left\langle\nabla^3\bar{f_h}(\theta_0),\triangle\theta_0(\triangle\theta_0)^{\rm T}\right\rangle_{\rm{HS}},\triangle\theta_0\right\rangle\right)\nonumber\\&-\frac{1}{6\eta}\int_0^1(1-t)^3\E\left(\nabla_{\triangle\theta_0}\nabla_{\triangle\theta_0}\nabla_{\triangle\theta_0}\nabla_{\triangle\theta_0}\bar{f_h}(\theta_0+t\triangle\theta_0)\right)\mathrm{d}t.
\end{align*}
For the first term, \eqref{CO1} implies $\pi_{\eta}(V)\leq\frac{4C_3}{K_1}$, together with \eqref{|g(x)|^2}, \eqref{boundedness} and Remark \ref{Re}, we have
\begin{align*}
    &\left|\frac{1}{2}\E\left(\left\langle\nabla^2\bar{f_h}(\theta_0),\eta b(\theta_0)b(\theta_0)^{\rm T}+\frac{1}{4}\eta\mathcal{R}(\theta_0,\xi_1)\mathcal{R}(\theta_0,\xi_1)^{\rm T}\right\rangle_{\rm{HS}}\right)\right|\\
    \leq &C\left(\eta\pi_{\eta}(\Arrowvert b\Arrowvert_2^2)+\eta\Arrowvert\sigma\Arrowvert_{\infty}^2\Arrowvert\nabla\sigma\Arrowvert_{\infty}^2\right)\leq C\eta.
\end{align*}
Using \eqref{boundedness} and Lemma \ref{ergo} again, we have
\begin{align*}
&\left|-\frac{1}{6\eta}\E\left(\left\langle\left\langle\nabla^3\bar{f_h}(\theta_0),\triangle\theta_0(\triangle\theta_0)^{\rm T}\right\rangle_{\rm{HS}},\triangle\theta_0\right\rangle\right)\right|\\
\leq&\frac{C}{\eta}\Big|\E\Big(\nabla_{\eta b(\theta_0)}\nabla_{\eta b(\theta_0)}\nabla_{\eta b(\theta_0)}\bar{f_h}(\theta_0)+\nabla_{\eta\mathcal{R}(\theta_0,\xi_1)}\nabla_{\eta\mathcal{R}(\theta_0,\xi_1)}\nabla_{\eta\mathcal{R}(\theta_0,\xi_1)}\bar{f_h}(\theta_0)\\&\qquad+\nabla_{\eta b(\theta_0)}\nabla_{\sqrt{\eta}\sigma(\theta_0)\xi_1}\nabla_{\sqrt{\eta}\sigma(\theta_0)\xi_1}\bar{f_h}(\theta_0)\\&\qquad+\nabla_{\eta b(\theta_0)}\nabla_{\eta\mathcal{R}(\theta_0,\xi_1)}\nabla_{\eta\mathcal{R}(\theta_0,\xi_1)}\bar{f_h}(\theta_0)+\nabla_{\sqrt{\eta}\sigma(\theta_0)\xi_1}\nabla_{\sqrt{\eta}\sigma(\theta_0)\xi_1}\nabla_{\eta\mathcal{R}(\theta_0,\xi_1)}\bar{f_h}(\theta_0)\Big)\Big|\\
\leq& \frac{C}{\eta}\left|\E\big(\eta^3\Arrowvert b(\theta_0)\Arrowvert_2^3+\eta^3\Arrowvert\mathcal{R}(\theta_0,\xi_1)\Arrowvert_2^3+\eta^2\Arrowvert\sigma\Arrowvert_{\infty}^2\Arrowvert b(\theta_0)\Arrowvert_2\right.\\&\left.\qquad+\eta^3\Arrowvert\sigma\Arrowvert_{\infty}^2\Arrowvert\nabla\sigma\Arrowvert_{\infty}^2\Arrowvert b(\theta_0)\Arrowvert_2+\eta^2\Arrowvert\sigma\Arrowvert_{\infty}^3\Arrowvert\nabla\sigma\Arrowvert_{\infty}\big)\right|\\
\leq& C\eta^2\E\Arrowvert b(\theta_0)\Arrowvert_2^3+C\eta\E\Arrowvert b(\theta_0)\Arrowvert_2+C\eta
\\\leq &C\eta^2\left(\E\Arrowvert b(\theta_0)\Arrowvert_2^4\right)^{3/4}+C\eta\left(\E\Arrowvert b(\theta_0)\Arrowvert_2^2\right)^{1/2}+C\eta\leq C\eta,
\end{align*}
and
\begin{align*}
&\left|\frac{1}{24\eta}\int_0^1\E\big(\nabla_{\triangle\theta_0}\nabla_{\triangle\theta_0}\nabla_{\triangle\theta_0}\nabla_{\triangle\theta_0}\bar{f_h}(\theta_0+t\triangle\theta_0)\big)\mathrm{d}t\right|\\
\leq&\frac{C}{\eta}\E\Arrowvert\eta b(\theta_0)+\sqrt{\eta}\sigma(\theta_0)\xi_1+\frac{1}{2}\eta\mathcal{R}(\theta_0,\xi_1)\Arrowvert_2^4
\\\leq &C\left(\eta^3\E\Arrowvert b(\theta_0)\Arrowvert_2^4+\eta\Arrowvert\sigma\Arrowvert_{\infty}^4+\eta^3\Arrowvert\sigma\Arrowvert_{\infty}^4\Arrowvert\nabla\sigma\Arrowvert_{\infty}^4\right)\leq C\eta.
\end{align*}
Hence, it holds that $\Big|\E\Big(\mathcal{A}\big(\bar{f_h}(\theta_0)\big)\Big)\Big|\leq C\eta$.

Finally, we deduce from Stein's equation (\ref{Stein}) that
\begin{align*}
\big|\pi_{\eta}\big(\Arrowvert\sigma^{\rm T}\nabla f_h\Arrowvert_2^2\big)-\pi\big(\Arrowvert\sigma^{\rm T}\nabla f_h\Arrowvert_2^2\big)\big|=&\Big|\E\Big(\Arrowvert\sigma(\theta_0)^{\rm T}\nabla f_h(\theta_0)\Arrowvert_2^2-\pi\big(\Arrowvert\sigma^{\rm T}\nabla f_h\Arrowvert_2^2\big)\Big)\Big|\\
=&\Big|\E\Big(\mathcal{A}\big(\bar{f_h}(\theta_0)\big)\Big)\Big|\leq C\eta.
\end{align*}
\end{proof}
\begin{proof}[Proof of Theorem \ref{CLT}]
According to Lemma \ref{convergence}, by the same argument as the proof of Theorem 2.4 in \cite{Lu}, one can derive that
\begin{align*}
\frac{1}{\sqrt{\eta}}\left(\Pi_\eta(h)-\pi(h)\right)\xrightarrow{\mathcal{L}} N\left(0,\pi\left(\left\Arrowvert\sigma^{{\rm T}}\nabla f_{h}\right\Arrowvert_2^{2}\right)\right).
\end{align*}
\end{proof}

\section{Proof of Theorem \ref{W}}\label{ProofW}

\subsection{Cram\'{e}r-type moderate deviations of martingale difference}
In the proof of Theorem \ref{W}, we also need the following normalized Cram\'{e}r-type moderate deviations for martingales \cite{Grama,Shao}. Explicitly, let $(\xi_i,\mathcal{F}_i)_{i=0,\cdots,n}$ be a finite sequence of martingale differences. Set $X_k=\sum_{i=1}^k\xi_i,k=1,\cdots,n$. Denote by $\big<X\big>$ the predictable quadratic of the martingale $X=(X_k,\mathcal{F}_k)_{k=0,\cdots,n}$, that is
$$
\big<X\big>_0=0,\quad\big<X\big>_k=\sum_{i=1}^k\E(\xi_i^2|\mathcal{F}_{i-1}),\quad k=1,\cdots,n.
$$
In the sequel, we shall use the following conditions:
\begin{itemize}
\item[(A1)] There exists a number $\epsilon_n\in\left(0,\frac{1}{2}\right]$ such that
$$
\left\lvert\E\left(\xi_i^k\big|\mathcal{F}_{i-1}\right)\right\rvert\leq\frac{1}{2}k!\epsilon_n^{k-2}\E\left(\xi_i^2\big|\mathcal{F}_{i-1}\right),
\text{\quad for all $k\geq2$ and $1\leq i\leq n$};
$$
\item[(A2)] There exist a number $\delta_n\in\left(0,\frac{1}{2}\right]$ and a positive constant $C_1$ such that for all $x>0$,
$$
\P\left(\left\lvert\big<X\big>_n-1\right\rvert\geq x\right)\leq C_1\exp\big\{-x^2\delta_n^{-2}\big\}.
$$
\item[(A2')] There exist a number $\delta_n\in\left(0,\frac{1}{2}\right]$ and a positive constant $C_2$ such that for all $x>0$,
$$
\P\left(\left\lvert\big<X\big>_n-1\right\rvert\geq x\right)\leq C_2\exp\big\{-x\delta_n^{-2}\big\}.
$$
\end{itemize}

\begin{pro}\label{CMD}
(Theorem 2.2 in \cite{Grama}, Theorem 2.1 in \cite{Shao}) Assume that conditions (A1), (A2) or (A2') are satisfied. Then the following inequality holds for all $0\leq x=o\big(\min\big\{\epsilon_n^{-1},\delta_n^{-1}\big\}\big)$,
\begin{align*}
\left\lvert\ln{\frac{\P\big(X_n/\sqrt{\big<X\big>_n}>x\big)}{1-\Phi(x)}}\right\rvert\leq C\Big(x^3\big(\epsilon_n+\delta_n\big)+(1+x)\big(\delta_n|\ln{\delta_n}|+\epsilon_n|\ln{\epsilon_n}|\big)\Big).
\end{align*}
\end{pro}

\subsection{Proof of Theorem \ref{W}}
Without loss of generality, we assume from now on that $\E\mathcal{Y}_\eta=1$. Notice that for all $0\leq x=o(\eta^{-1/2})$ and $0\leq x-y=o(\eta^{-1/2})$, we have
\begin{align*}
\P\big(W_\eta\geq x\big)=\P\left(\frac{\mathcal{H}_\eta+\mathcal{R}_\eta}{\sqrt{\mathcal{Y}_\eta}}\geq x\right)
\leq\P\left(\frac{\mathcal{H}_\eta}{\sqrt{\mathcal{Y}_\eta}}\geq x-y\right)
+\P\left(\frac{\mathcal{R}_\eta}{\sqrt{\mathcal{Y}_\eta}}\geq y\right).
\end{align*}

First, to estimate $\P\left(\mathcal{H}_{\eta}/\sqrt{\mathcal{Y}_{\eta}}\geq x-y\right)$, set the filtration $\mathcal{F}_n=\sigma(\theta_0,\xi_k,1\leq k\leq n)$. Then $\big(-\eta\big<\nabla f_h(\theta_k),\sigma(\theta_k)\xi_{k+1}\big>,\mathcal{F}_{k+1}\big)_{k\geq 0}$ is a sequence of martingale differences. Using $(k-1)!!\leq\frac{1}{2}k!,\, k\geq2$, we have 
$$\left|\E\big((-\eta\big<\nabla f_h(\theta_k),\sigma(\theta_k)\xi_{k+1}\big>)^k|\mathcal{F}_k\big)\right|\leq \frac{1}{2}k!\left(\Arrowvert\nabla f_h(\theta_k)\Arrowvert_2\Arrowvert\sigma(\theta_k)\Arrowvert\eta\right)^k.$$
Let $\lambda_{\min}>0$ denote the smallest eigenvalue of positive definite matrix $\sigma$ in {\bf Assumption} {\bf \ref{conditions2}}, we can obtain 
$$\E\big((-\eta\big<\nabla f_h(\theta_k),\sigma(\theta_k)\xi_{k+1}\big>)^2|\mathcal{F}_k\big)\geq \lambda_{\min}^2\Arrowvert\nabla f_h(\theta_k)\Arrowvert_2^2\eta^2.$$ 
Thus we have the Bernstein condition (see (4) in \cite{Grama}), by the boundedness of $\Arrowvert\nabla f_h(\theta_k)\Arrowvert_2$ and $\Arrowvert\sigma(\theta_k)\Arrowvert$, it holds for all $k\geq 2$, 
\begin{align*}
&\left|\E\big((-\eta\big<\nabla f_h(\theta_k),\sigma(\theta_k)\xi_{k+1}\big>)^k|\mathcal{F}_k\big)\right|\\
\leq &\frac{1}{2}k!\left(\Arrowvert\nabla f_h(\theta_k)\Arrowvert_2\left(\Arrowvert\sigma(\theta_k)\Arrowvert^k/\lambda_{\min}^2\right)^{\frac{1}{k-2}}\eta\right)^{k-2}\E\big((-\eta\big<\nabla f_h(\theta_k),\sigma(\theta_k)\xi_{k+1}\big>)^2|\mathcal{F}_k\big).
\end{align*}
Since
\begin{align*}
\big<\mathcal{H}_\eta\big>_m=\sum_{k=0}^{m-1}\E\big((-\eta\big<\nabla f_h(\theta_k),\sigma(\theta_k)\xi_{k+1}\big>)^2|\mathcal{F}_k\big)=\mathcal{Y}_\eta,
\end{align*}
by Proposition \ref{EY} with $k=\eta^{-2}$ and $\E\mathcal{Y}_{\eta}=1$, we have for all $x>0$,
\begin{align*}
\P\big(|\big<\mathcal{H}_\eta\big>_m-1|\geq x\big)=\P\big(|\mathcal{Y}_\eta-\E\mathcal{Y}_\eta|\geq x\big)\leq 2e^{-cx^2\eta^{-2}}+c_3\left(1+\eta^{-1}\right)\left(e^{-c_1x\eta^{-1}}+\frac{e^{-c_2x\eta^{-2}}}{1-e^{-u_1\eta}}\right),
\end{align*}
so we have $$\P\big(|\big<\mathcal{H}_\eta\big>_m-1|\geq x\big)\leq ce^{-c_1x\eta^{-1}},\quad x\geq \eta;$$
and
$$\P\big(|\big<\mathcal{H}_\eta\big>_m-1|\geq x\big)\leq ce^{-c_1x^2\eta^{-2}},\quad x<\eta,$$
i.e. we have $\delta_n=c\eta^{1/2}, x\geq \eta$ and $\delta_n=c\eta, x<\eta$.

Using the fact that
\begin{align}\label{normal}
\frac{1}{\sqrt{2\pi}(1+x)}e^{-x^2/2}\leq1-\Phi(x)\leq\frac{1}{\sqrt{\pi}(1+x)}e^{-x^2/2},\quad x\geq0,
\end{align}
we have $\left(1-\Phi(x-y)\right)/\left(1-\Phi(x)\right)=O\left(\exp\{xy-\frac{1}{2}y^2\}\right)$. By Proposition \ref{CMD} with $\epsilon_n=c\eta$ and $\delta_n=c\eta^{1/2}$, we get for all $0<y\leq x=o(\eta^{-1/2})$,
\begin{align*}
\frac{\P\big(\mathcal{H}_\eta/\sqrt{\mathcal{Y}_\eta}\geq x-y\big)}{1-\Phi(x)}
&=\frac{\P\big(\mathcal{H}_\eta/\sqrt{\mathcal{Y}_\eta}\geq x-y\big)}{1-\Phi(x-y)}\frac{1-\Phi(x-y)}{1-\Phi(x)}\\
&\leq \exp\left\{c_1\left(x^3\eta^{1/2}+(1+x)\eta^{1/2}|\ln\eta^{1/2}|\right)\right\}\exp\left\{c_2xy\right\}\\
&\leq \exp\left\{c_3\left(x^3\eta^{1/2}+(1+x)\eta^{1/2}|\ln\eta^{1/2}|+xy\right)\right\}.
\end{align*}

Second, we can obtain for all $0<y\leq x=o(\eta^{-1/2})$,
\begin{align}\label{main probability 1}
\frac{\P\left(W_\eta\geq x\right)}{1-\Phi(x)}
&\leq \frac{\P\big(\mathcal{H}_\eta/\sqrt{\mathcal{Y}_\eta}\geq x-y\big)}{1-\Phi(x)}+\frac{\P\left(\mathcal{R}_{\eta}/\sqrt{\mathcal{Y_{\eta}}}\geq y\right)}{1-\Phi(x)}\nonumber\\
&\leq \exp\left\{c_3\left(x^3\eta^{1/2}+(1+x)\eta^{1/2}|\ln\eta^{1/2}|+xy\right)\right\}+\frac{\P\left(\mathcal{R}_{\eta}/\sqrt{\mathcal{Y_{\eta}}}\geq y\right)}{1-\Phi(x)}.
\end{align}

Similarly, we have for all $0<y\leq x=o(\eta^{-1/2})$,
\begin{align}\label{main probability 2}
\frac{\P\left(W_\eta\geq x\right)}{1-\Phi(x)}
&\geq \frac{\P\big(\mathcal{H}_\eta/\sqrt{\mathcal{Y}_\eta}\geq x+y\big)}{1-\Phi(x)}-\frac{\P\left(\mathcal{R}_{\eta}/\sqrt{\mathcal{Y_{\eta}}}\leq -y\right)}{1-\Phi(x)}\nonumber\\
&\geq \exp\left\{-c_3\left(x^3\eta^{1/2}+(1+x)\eta^{1/2}|\ln\eta^{1/2}|+xy\right)\right\}-\frac{\P\left(\mathcal{R}_{\eta}/\sqrt{\mathcal{Y_{\eta}}}\leq -y\right)}{1-\Phi(x)}.
\end{align}

Third, we give an estimation for $\P\left(\frac{\mathcal{R}_\eta}{\sqrt{\mathcal{Y}_\eta}}\geq y\right)\big/\big(1-\Phi(x)\big)$. By $\nabla_{\sigma(\theta_k)\xi_{k+1}}\sigma(\theta_k)\not\equiv 0$ in Assumption {\bf\ref{conditions}} with bounded $\Arrowvert\nabla\sigma\Arrowvert_{\infty}$, Proposition \ref{remainder} implies the following: for all $\eta^{1/2}\mathbf{R}^3\Arrowvert\nabla\sigma\Arrowvert_{\infty}^3<y=o(\eta^{-1/2})$ and $\mathbf{R}$ satisfying $c_1\eta^{-a}<\mathbf{R}<c_2\Arrowvert\nabla\sigma\Arrowvert_{\infty}^{-1}\eta^{-1/3},\; a>0$,
$$
\P\left(\left\lvert\mathcal{R}_{\eta}\right\rvert\geq y\right)\leq c_1\exp\left\{-\frac{\mathbf{R}^2}{2}\right\}.
$$
Thus we have
\begin{align*}
\P\left(\frac{\mathcal{R}_\eta}{\sqrt{\mathcal{Y}_\eta}}\geq y\right)
\leq&\P\left(\frac{\mathcal{R}_\eta}{\sqrt{\mathcal{Y}_\eta}}\geq y,\mathcal{Y}_\eta>\E\mathcal{Y}_\eta-\frac{1}{2}\E\mathcal{Y}_\eta\right)+\P\left(\mathcal{Y}_\eta\leq\E\mathcal{Y}_\eta-\frac{1}{2}\E\mathcal{Y}_\eta\right)\\
\leq&\P\left(\frac{\mathcal{R}_\eta}{\sqrt{\E\mathcal{Y}_\eta/2}}\geq y\right)+\P\left(\left|\E\mathcal{Y}_\eta-\mathcal{Y}_\eta\right|\geq\frac{1}{2}\E\mathcal{Y}_\eta\right)\\
\leq& c_1\exp\left\{-\frac{\mathbf{R}^2}{2}\right\}+2e^{-c\eta^{-2}}+c_3\left(1+\eta^{-1}\right)\left(e^{-c_1\eta^{-1}}+\frac{e^{-c_2\eta^{-2}}}{1-e^{-u_1\eta}}\right)
\leq ce^{-\frac{\mathbf{R}^2}{2}}.
\end{align*}
Furthermore, together with \eqref{normal}, we can obtain for all $\eta^{1/2}\mathbf{R}^3\Arrowvert\nabla\sigma\Arrowvert_{\infty}^3<y=o(\eta^{-1/2})$,
\begin{align}\label{tail probability}
\frac{\P\left(\mathcal{R}_\eta/\sqrt{\mathcal{Y}_\eta}\geq y\right)}{1-\Phi(x)}
\leq C(1+x)e^{\frac{1}{2}x^2}\P\left(\mathcal{R}_\eta/\sqrt{\mathcal{Y}_\eta}\geq y\right)\leq C(1+x)\exp\left\{\frac{1}{2}\left(x^2-\mathbf{R}^2\right)\right\},
\end{align}
which converges to 0 as $\eta\to 0$ uniformly for
$$
\eta^{1/2}\mathbf{R}^3\Arrowvert\nabla\sigma\Arrowvert_{\infty}^3<y\leq x<\mathbf{R}.
$$
Together with \eqref{main probability 1} and \eqref{main probability 2}, we can choose $y=\epsilon x^{-1}$ and $\mathbf{R}=\epsilon^{-1}x$, where $\epsilon>0$ is taken to be sufficiently small. And then we can take $y=c\Arrowvert\nabla\sigma\Arrowvert_{\infty}^{3/4}\eta^{1/8}$, $\mathbf{R}=c\Arrowvert\nabla\sigma\Arrowvert_{\infty}^{-3/4}\eta^{-1/8}$ and we have
\begin{align*}
\frac{\P\left(\mathcal{R}_\eta/\sqrt{\mathcal{Y}_\eta}\geq y\right)}{1-\Phi(x)}\leq Ce^{-c\Arrowvert\nabla\sigma\Arrowvert_{\infty}^{-3/2}\eta^{-1/4}}\to 0,
\end{align*}
uniformly for $c\Arrowvert
\nabla\sigma\Arrowvert_{\infty}^{3/4}\eta^{1/8}\leq x=o(\Arrowvert\nabla\sigma\Arrowvert_{\infty}^{-3/4}\eta^{-1/8})$ as $\eta$ vanishes.

Similarly, we can analyze $\P\left(\mathcal{R}_\eta/\sqrt{\mathcal{Y}_\eta}\leq -y\right)/\left({1-\Phi(x)}\right).$ Furthermore, this together with (\ref{main probability 1}) and (\ref{main probability 2}), the proof of Theorem \ref{W} is complete.

\qed

\section{Proof of Theorem \ref{S}}\label{ProofS}

Assume that $\varepsilon_x\in(0,1/2]$. It is easy to see that for all $x\geq0$,
\begin{align}\label{P_u}
&\P\big(S_\eta>x\big)\nonumber\\
=&\P\big(\eta^{-1/2}(\Pi_\eta(h)-\pi(h))>x\sqrt{\mathcal{V}_\eta}\big)\nonumber\\
=&\P\big(\eta^{-1/2}(\Pi_\eta(h)-\pi(h))>x\sqrt{\mathcal{V}_\eta},\mathcal{V}_\eta\geq(1-\varepsilon_x)\mathcal{Y}_\eta\big)\nonumber\\&+\P\big(\eta^{-1/2}(\Pi_\eta(h)-\pi(h))>x\sqrt{\mathcal{V}_\eta},\mathcal{V}_\eta<(1-\varepsilon_x)\mathcal{Y}_\eta\big)\nonumber\\
\leq&\P\big(W_\eta\geq x\sqrt{1-\varepsilon_x}\big)+\P\big(\mathcal{V}_\eta-\mathcal{Y}_\eta<-\varepsilon_x\mathcal{Y}_\eta,\mathcal{Y}_\eta\geq\frac{1}{2}\E\mathcal{Y}_\eta\big)+\P\big(\mathcal{V}_\eta-\mathcal{Y}_\eta<-\varepsilon_x\mathcal{Y}_\eta,\mathcal{Y}_\eta<\frac{1}{2}\E\mathcal{Y}_\eta\big)\nonumber\\
\leq&\P\big(W_\eta\geq x\sqrt{1-\varepsilon_x}\big)+\P\big(\mathcal{V}_\eta-\mathcal{Y}_\eta<-\frac{1}{2}\varepsilon_x\E\mathcal{Y}_\eta\big)+\P\big(\mathcal{Y}_\eta-\E\mathcal{Y}_\eta<-\frac{1}{2}\E\mathcal{Y}_\eta\big)\nonumber\\:=&P_1+P_2+P_3.
\end{align}
First, by Theorem \ref{W} and \eqref{normal}, for $c\Arrowvert\nabla\sigma\Arrowvert_{\infty}^{3/4}\eta^{1/8}\leq x=o(\Arrowvert\nabla\sigma\Arrowvert_{\infty}^{-3/4}\eta^{-1/8})$, we have
\begin{align}\label{P_1}
P_1
&\leq\Big(1-\Phi\big(x\sqrt{1-\varepsilon_x}\big)\Big)\exp\left\{c\left(x^3\eta^{1/2}+(1+x)\eta^{1/2}|\ln\eta^{1/2}|+x\Arrowvert\nabla\sigma\Arrowvert_{\infty}^{3/4}\eta^{1/8}\right)\right\}\nonumber\\
&\leq\Big(1-\Phi\big(x\big)\Big)\exp\left\{c\left(x^2\varepsilon_x+x^3\eta^{1/2}+(1+x)\eta^{1/2}|\ln\eta^{1/2}|+x\Arrowvert\nabla\sigma\Arrowvert_{\infty}^{3/4}\eta^{1/8}\right)\right\}.
\end{align}
Second, for $0\leq k\leq[\eta^{-2}]-1$, denote
\begin{align*}
    \Psi_{k+1}=\eta^{-1}\big\Arrowvert\big(\theta_{k+1}-\theta_k-\eta b(\theta_k)-\frac{1}{2}\eta\mathcal{R}(\theta_k,\xi_{k+1})\big)^{\rm T}\nabla f_h(\theta_k)\big\Arrowvert_2^2-\big\Arrowvert \sigma(\theta_k)^{\rm T}\nabla f_h(\theta_k)\big\Arrowvert_2^2,
\end{align*}
by Proposition \ref{VY} and $\E\mathcal{Y}_{\eta}=1$, we get for all $x\geq0$,
\begin{align}\label{P_2}
P_2
=&\P\big(\mathcal{V}_\eta-\mathcal{Y}_\eta<-\frac{1}{2}\varepsilon_x\E\mathcal{Y}_\eta\big)
\leq\P\big(\lvert\mathcal{V}_\eta-\mathcal{Y}_\eta\rvert>\frac{1}{2}\varepsilon_x\E\mathcal{Y}_\eta\big)\nonumber\\
=&\P\Bigg(\Bigg\lvert\sum_{k=0}^{[\eta^{-2}]-1}\Psi_{k+1}\Bigg\rvert>\frac{1}{2}\varepsilon_x\eta^{-2}\Bigg)\nonumber\\
\leq& c_1\exp\left\{-\frac{\frac{1}{4}\varepsilon_x^2\eta^{-4}}{c_1\big(\eta^{-2}+c\varepsilon_x\eta^{-2}\big)}\right\}
\leq c_1\exp\big\{-c\varepsilon_x^2\eta^{-2}\big\}.
\end{align}
Third, using Proposition \ref{EY} and $\E\mathcal{Y}_{\eta}=1$, we deduce that for all $x\geq0$,
\begin{align}\label{P_3}
P_3
&=\P\big(\mathcal{Y}_\eta-\E\mathcal{Y}_\eta<-\frac{1}{2}\E\mathcal{Y}_\eta\big)=\P\big(\lvert\mathcal{Y}_\eta-\E\mathcal{Y}_\eta\rvert>\frac{1}{2}\E\mathcal{Y}_\eta\big)\nonumber\\
&=\P\left(\left\lvert\sum_{k=0}^{m-1}\big\Arrowvert\sigma(\theta_k)^{\rm T}\nabla f_h(\theta_k)\big\Arrowvert_2^2-\sum_{k=0}^{m-1}\E\big\Arrowvert\sigma(\theta_k)^{\rm T}\nabla f_h(\theta_k)\big\Arrowvert_2^2\right\rvert>\frac{1}{2}m\right)\nonumber\\
&\leq 2e^{-c\eta^{-2}}+c_3\left(1+\eta^{-1}\right)\left(e^{-c_1\eta^{-1}}+\frac{e^{-c_2\eta^{-2}}}{1-e^{-u_1\eta}}\right)\leq c\left(1+\eta^{-1}\right)e^{-c_1\eta^{-1}}.
\end{align}
Finally, taking $\varepsilon_x=c_0\eta^{1/2}$ with $c_0>0$, by (\ref{P_u})-(\ref{P_3}) and \eqref{normal}, we deduce that for all $c\Arrowvert\nabla\sigma\Arrowvert_{\infty}^{3/4}\eta^{1/8}\leq x=o(\Arrowvert\nabla\sigma\Arrowvert_{\infty}^{-3/4}\eta^{-1/8})$,
\begin{align*}
\P\big(S_\eta>x\big)\leq&\left(1-\Phi(x)\right)\exp\left\{c\left(x^3\eta^{1/2}+x^2\eta^{1/2}+(1+x)\eta^{1/2}|\ln\eta^{1/2}|+x\Arrowvert\nabla\sigma\Arrowvert_{\infty}^{3/4}\eta^{1/8}\right)\right\}\\
&+c_1\exp\left\{-cc_0^2\eta^{-1}\right\}+c\left(1+\eta^{-1}\right)e^{-c_1\eta^{-1}}\\
\leq&\left(1-\Phi(x)\right)\exp\left\{c\left(x^3\eta^{1/2}+x^2\eta^{1/2}+(1+x)\eta^{1/2}|\ln\eta^{1/2}|+x\Arrowvert\nabla\sigma\Arrowvert_{\infty}^{3/4}\eta^{1/8}\right)\right\}.
\end{align*}
Thus we obtain the upper bound for the tail probability $\P\big(S_\eta>x\big)$ with $x\geq0$.

Next, we estimate the lower bound. Notice that for all $x\geq0$,
\begin{align}\label{P_l}
&\P\big(S_\eta>x\big)\nonumber\\
\geq&\P\big(\eta^{-1/2}(\Pi_\eta(h)-\pi(h))>x\sqrt{\mathcal{V}_\eta},\mathcal{V}_\eta<(1+\varepsilon_x)\mathcal{Y}_\eta\big)\nonumber\\
\geq&\P\big(W_\eta\geq x\sqrt{1+\varepsilon_x},\mathcal{V}_\eta<(1+\varepsilon_x)\mathcal{Y}_\eta\big)\geq\P\big(W_\eta\geq x\sqrt{1+\varepsilon_x}\big)-\P\big(\mathcal{V}_\eta\geq(1+\varepsilon_x)\mathcal{Y}_\eta\big)\nonumber\\
=&\P\big(W_\eta\geq x\sqrt{1+\varepsilon_x}\big)-\P\big(\mathcal{V}_\eta\geq(1+\varepsilon_x)\mathcal{Y}_\eta,\mathcal{Y}_\eta\leq\frac{1}{2}\E\mathcal{Y}_\eta\big)-\P\big(\mathcal{V}_\eta\geq(1+\varepsilon_x)\mathcal{Y}_\eta,\mathcal{Y}_\eta>\frac{1}{2}\E\mathcal{Y}_\eta\big)\nonumber\\
\geq&\P\big(W_\eta\geq x\sqrt{1+\varepsilon_x}\big)-\P\big(\mathcal{Y}_\eta-\E\mathcal{Y}_\eta\leq-\frac{1}{2}\E\mathcal{Y}_\eta\big)-\P\big(\mathcal{V}_\eta-\mathcal{Y}_\eta\geq\frac{1}{2}\varepsilon_x\E\mathcal{Y}_\eta\big)
\nonumber\\:=&P_4-P_5-P_6.
\end{align}
First, by Theorem \ref{W} and \eqref{normal}, we have for all $c\Arrowvert\nabla\sigma\Arrowvert_{\infty}^{3/4}\eta^{1/8}\leq x=o(\Arrowvert\nabla\sigma\Arrowvert_{\infty}^{-3/4}\eta^{-1/8})$,
\begin{align}\label{P_4}
P_4
&\geq\Big(1-\Phi\big(x\sqrt{1+\varepsilon_x}\big)\Big)\exp\left\{-c_3\left(x^3\eta^{1/2}+(1+x)\eta^{1/2}|\ln\eta^{1/2}|+x\Arrowvert\nabla\sigma\Arrowvert_{\infty}^{3/4}\eta^{1/8}\right)\right\}\nonumber\\
&\geq\left(1-\Phi(x)\right)\exp\left\{-c\left(x^2\varepsilon_x+x^3\eta^{1/2}+(1+x)\eta^{1/2}|\ln\eta^{1/2}|+x\Arrowvert\nabla\sigma\Arrowvert_{\infty}^{3/4}\eta^{1/8}\right)\right\}.
\end{align}
Second, using Proposition \ref{EY} with $y=c\eta^{-2}$, we get for all $x\geq0$,
\begin{align}\label{P_5}
P_5
=&\P\big(\mathcal{Y}_\eta-\E\mathcal{Y}_\eta\leq-\frac{1}{2}\E\mathcal{Y}_\eta\big)
\nonumber\\\leq&2e^{-c\eta^{-2}}+c_3\left(1+\eta^{-1}\right)\left(e^{-c_1\eta^{-1}}+\frac{e^{-c_2\eta^{-2}}}{1-e^{-u_1\eta}}\right)\leq c\left(1+\eta^{-1}\right)e^{-c_1\eta^{-1}}.
\end{align}
Third, using Proposition \ref{VY}, we get for all $x\geq0$,
\begin{align}\label{P_6}
P_6
=\P\big(\mathcal{V}_\eta-\mathcal{Y}_\eta\geq\frac{1}{2}\varepsilon_x\E\mathcal{Y}_\eta\big)
\leq c_1\exp\Big\{-\frac{\frac{1}{4}\varepsilon_x^2\eta^{-4}}{c_1\big(\eta^{-2}+c\varepsilon_x\eta^{-2}\big)}\Big\}
\leq c_1\exp\big\{-c\varepsilon_x^2\eta^{-2}\big\}.
\end{align}
Finally, taking $\varepsilon_x=c_0\eta^{1/2}$ with $c_0>0$, by (\ref{P_l})-(\ref{P_6}), we deduce that for all $c\Arrowvert\nabla\sigma\Arrowvert_{\infty}^{3/4}\eta^{1/8}\leq x=o(\Arrowvert\nabla\sigma\Arrowvert_{\infty}^{-3/4}\eta^{-1/8})$,
\begin{align*}
\P\big(S_\eta>x\big)
\geq&\Big(1-\Phi(x)\Big)\exp\left\{-c\left(x^3\eta^{1/2}+x^2\eta^{1/2}+(1+x)\eta^{1/2}|\ln\eta^{1/2}|+x\Arrowvert\nabla\sigma\Arrowvert_{\infty}^{3/4}\eta^{1/8}\right)\right\}\\
&-c_1\exp\left\{-cc_0^2\eta^{-1}\right\}-c\left(1+\eta^{-1}\right)e^{-c_1\eta^{-1}},
\end{align*}
applying (\ref{normal}) to the last inequality, we obtain for all $c\Arrowvert\nabla\sigma\Arrowvert_{\infty}^{3/4}\eta^{1/8}\leq x=o(\Arrowvert\nabla\sigma\Arrowvert_{\infty}^{-3/4}\eta^{-1/8})$,
\begin{align*}
\P\big(S_\eta>x\big)\geq\Big(1-\Phi(x)\Big)\exp\left\{-c\left(x^3\eta^{1/2}+x^2\eta^{1/2}+(1+x)\eta^{1/2}|\ln\eta^{1/2}|+x\Arrowvert\nabla\sigma\Arrowvert_{\infty}^{3/4}\eta^{1/8}\right)\right\}.
\end{align*}
Thus we obtain the lower bound for the tail probability $\P\big(S_\eta>x\big)$ for $x\geq0$. The proof for $-S_\eta$ follows by a similar argument and the proof of Theorem \ref{S} is complete.

\qed

\begin{appendix}

\section{Proof of Proposition \ref{g}}\label{Appendix}

Since $A=\left\{\left|\xi_i\right|\leq \mathbf{R},i=1,\cdots,m\right\}$, it holds that
\begin{align*}
&\P_{\theta_0}\left(\eta\sum_{k=0}^{m-1}\Arrowvert b(\theta_k)\Arrowvert_2^2>x\right)\\
\leq&\P_{\theta_0}\left(\eta\sum_{k=0}^{m-1}\Arrowvert b(\theta_k)\Arrowvert_2^2>x,A\right)+\P_{\theta_0}\left(\eta\sum_{k=0}^{m-1}\Arrowvert b(\theta_k)\Arrowvert_2^2>x,A^c\right)
:=P_1+P_2.
\end{align*}

We first calculate the upper bound of $P_1$. One can write from the Markov inequality
\begin{align*}
P_1
=\P_{\theta_0}\left(\eta\sum_{k=0}^{m-1}\Arrowvert b(\theta_k)\Arrowvert_2^2>x,A\right)
=&\P_{\theta_0}\left(\left(\eta\sum_{k=0}^{m-1}\Arrowvert b(\theta_k)\Arrowvert_2^2\right)I_A>x\right)
\\\leq&\E_{\theta_0}\left(\exp\left\{\left(c\eta\sum_{k=0}^{m-1}\Arrowvert b(\theta_k)\Arrowvert_2^2\right)I_A\right\}\right)\cdot e^{-cx}.
\end{align*}
By (\ref{|g(x)|^2}) and (\ref{|x|^2}), we have
\begin{align*}
&\E_{\theta_0}\left(\exp\left\{\left(c\eta\sum_{k=0}^{m-1}\Arrowvert b(\theta_k)\Arrowvert_2^2\right)I_A\right\}\right)\\
\leq&\E_{\theta_0}\left(\exp\left\{\left(c\eta\sum_{k=0}^{m-1}\left(2L^2\Arrowvert\theta_k\Arrowvert_2^2+2\Arrowvert b(0)\Arrowvert_2^2\right)\right)I_A\right\}\right)\\
\leq&\E_{\theta_0}\left(\exp\left\{\left(c_1\eta\sum_{k=0}^{m-1}\Arrowvert\theta_k\Arrowvert_2^2\right)I_A\right\}\right)\cdot e^{2c\Arrowvert b(0)\Arrowvert_2^2\eta^{-1}}\\
\leq&\E_{\theta_0}\left(\exp\left\{\left(c_1\eta\sum_{k=0}^{m-1}\left(\frac{2}{K_1}C-\frac{2}{K_1}\left<\theta_k,b(\theta_k)\right>\right)\right)I_A\right\}\right)\cdot e^{2c\Arrowvert b(0)\Arrowvert_2^2\eta^{-1}}\\
\leq& e^{c_3\eta^{-1}}\E_{\theta_0}\left(\exp\left\{\left(-c_2\eta\sum_{k=0}^{m-1}\left<\theta_k,b(\theta_k)\right>\right)I_A\right\}\right).
\end{align*}
Since $\theta_{k+1}=\theta_k+\eta b(\theta_k)+\sqrt{\eta}\sigma(\theta_k)\xi_{k+1}+\frac{1}{2}\eta\mathcal{R}(\theta_k,\xi_{k+1}),\;k\geq 0$, it is easy to calculate that
\begin{align}\label{square}
&\Arrowvert\theta_{k+1}\Arrowvert_2^2-\Arrowvert\theta_k\Arrowvert_2^2\nonumber\\
=&2\left<\theta_k,\eta b(\theta_k)\right>+\eta^2\Arrowvert b(\theta_k)\Arrowvert_2^2+2\left<\theta_k+\eta b(\theta_k),\sqrt{\eta}\sigma(\theta_k)\xi_{k+1}\right>+\eta\Arrowvert\sigma(\theta_k)\xi_{k+1}\Arrowvert_2^2\nonumber\\
&+2\left<\theta_k+\eta b(\theta_k)+\sqrt{\eta} \sigma(\theta_k)\xi_{k+1},\frac{1}{2}\eta\mathcal{R}(\theta_k,\xi_{k+1})\right>+\frac{1}{4}\eta^2\Arrowvert\mathcal{R}(\theta_k,\xi_{k+1})\Arrowvert_2^2,
\end{align}
summing the above from $k=0$ to $k=m-1$, we obtain
\begin{align}\label{sum}
&-\sum_{k=0}^{m-1}\left<\theta_k,\eta b(\theta_k)\right>\nonumber\\
\leq&\frac{1}{2}\Arrowvert\theta_0\Arrowvert_2^2+\sum_{k=0}^{m-1}\Big(\frac{1}{2}\eta^2\Arrowvert b(\theta_k)\Arrowvert_2^2+\left<\theta_k+\eta b(\theta_k),\sqrt{\eta} \sigma(\theta_k)\xi_{k+1}\right>+\frac{1}{2}\eta\Arrowvert\sigma(\theta_k)\xi_{k+1}\Arrowvert_2^2\nonumber\\
&+\frac{1}{2}\left<\theta_k+\eta b(\theta_k)+\sqrt{\eta} \sigma(\theta_k)\xi_{k+1},\eta\mathcal{R}(\theta_k,\xi_{k+1})\right>+\frac{1}{8}\eta^2\Arrowvert\mathcal{R}(\theta_k,\xi_{k+1})\Arrowvert_2^2\Big),
\end{align}
which implies
\begin{align*}
&\E_{\theta_0}\left(\exp\left\{\left(-c_2\eta\sum_{k=0}^{m-1}\left<\theta_k,b(\theta_k)\right>\right)I_A\right\}\right)\\
\leq& e^{\frac{1}{2}c_2\Arrowvert\theta_0\Arrowvert_2^2}\cdot\E_{\theta_0}\Bigg(\exp\left\{\Bigg(c_2\sum_{k=0}^{m-1}\Bigg(\frac{1}{2}\eta^2\Arrowvert b(\theta_k)\Arrowvert_2^2+\left<\theta_k+\eta b(\theta_k),\sqrt{\eta} \sigma(\theta_k)\xi_{k+1}\right>+\frac{1}{2}\eta\Arrowvert\sigma(\theta_k)\xi_{k+1}\Arrowvert_2^2\right.\\
&\left.\quad\qquad\qquad+\frac{1}{2}\left<\theta_k+\eta b(\theta_k)+\sqrt{\eta} \sigma(\theta_k)\xi_{k+1},\eta\mathcal{R}(\theta_k,\xi_{k+1})\right>+\frac{1}{8}\eta^2\Arrowvert\mathcal{R}(\theta_k,\xi_{k+1})\Arrowvert_2^2\Bigg)\Bigg)I_A\right\}\Bigg).
\end{align*}
By Lemma \ref{E_k} with $\Psi_1(\theta_k)=c_2\left(\sqrt{\eta}\theta_k+\eta^{\frac{3}{2}}b(\theta_k)\right)$, we have
\begin{align*}
&\E_{\theta_0}\exp\left\{\Bigg(c_2\sum_{k=0}^{m-1}\Bigg(\frac{1}{2}\eta^2\Arrowvert b(\theta_k)\Arrowvert_2^2+\left<\theta_k+\eta b(\theta_k),\sqrt{\eta} \sigma(\theta_k)\xi_{k+1}\right>+\frac{1}{2}\eta\Arrowvert\sigma(\theta_k)\xi_{k+1}\Arrowvert_2^2\right.\\
&\left.\qquad\qquad\quad+\frac{1}{2}\left<\theta_k+\eta b(\theta_k)+\sqrt{\eta} \sigma(\theta_k)\xi_{k+1},\eta\mathcal{R}(\theta_k,\xi_{k+1})\right>+\frac{1}{8}\eta^2\Arrowvert\mathcal{R}(\theta_k,\xi_{k+1})\Arrowvert_2^2\Bigg)\Bigg)I_A\right\}\\
\leq&\Bigg(\E_{\theta_0}\exp\left\{\Bigg(2c_2\sum_{k=0}^{m-1}\Bigg(\eta\Arrowvert\theta_k+\eta b(\theta_k)\Arrowvert_2^2\cdot\Arrowvert\sigma\Arrowvert_{\infty}^2+\frac{1}{2}\eta^2\Arrowvert b(\theta_k)\Arrowvert_2^2+\frac{1}{2}\eta\Arrowvert\sigma(\theta_k)\xi_{k+1}\Arrowvert_2^2\right.\\
&\left.\quad\qquad\qquad+\frac{1}{2}\left<\theta_k+\eta b(\theta_k)+\sqrt{\eta} \sigma(\theta_k)\xi_{k+1},\eta\mathcal{R}(\theta_k,\xi_{k+1})\right>+\frac{1}{8}\eta^2\Arrowvert\mathcal{R}(\theta_k,\xi_{k+1})\Arrowvert_2^2\Bigg)\Bigg)I_A\right\}\Bigg)^{\frac{1}{2}}\\
\leq&\Bigg(\E_{\theta_0}\exp\left\{\Bigg(2c_2\sum_{k=0}^{m-1}\Bigg(\eta\Arrowvert\theta_k+\eta b(\theta_k)\Arrowvert_2^2\cdot\Arrowvert\sigma\Arrowvert_{\infty}^2+\frac{1}{2}\eta^2\Arrowvert b(\theta_k)\Arrowvert_2^2+\frac{1}{4}\eta\Arrowvert\theta_k+\eta b(\theta_k)\Arrowvert_2^2\right.\\
&\left.\qquad\qquad\qquad\qquad+\frac{1}{4}\eta\Arrowvert\mathcal{R}(\theta_k,\xi_{k+1})\Arrowvert_2^2+\frac{3}{4}\eta\Arrowvert\sigma(\theta_k)\xi_{k+1}\Arrowvert_2^2+\frac{3}{8}\eta^2\Arrowvert\mathcal{R}(\theta_k,\xi_{k+1})\Arrowvert_2^2\Bigg)\Bigg)I_A\right\}\Bigg)^{\frac{1}{2}}\\
\leq&\left(\E_{\theta_0}\exp\left\{\Bigg(4c_2\sum_{k=0}^{m-1}\Bigg(\eta\Arrowvert\theta_k+\eta b(\theta_k)\Arrowvert_2^2\cdot\Arrowvert\sigma\Arrowvert_{\infty}^2+\frac{1}{2}\eta^2\Arrowvert b(\theta_k)\Arrowvert_2^2+\frac{1}{4}\eta\Arrowvert\theta_k+\eta b(\theta_k)\Arrowvert_2^2\Bigg)\Bigg)I_A\right\}\right)^{\frac{1}{4}}\\
&\cdot\left(\E_{\theta_0}\exp\left\{6c_2\sum_{k=0}^{m-1}\eta\Arrowvert\sigma(\theta_k)\xi_{k+1}\Arrowvert_2^2\right\}\right)^{\frac{1}{8}}\cdot\left(\E_{\theta_0}\exp\left\{5c_2\sum_{k=0}^{m-1}\eta\Arrowvert\mathcal{R}(\theta_k,\xi_{k+1})\Arrowvert_2^2I_{\{|\xi_{k+1}|\leq \mathbf{R}\}}\right\}\right)^{\frac{1}{8}},
\end{align*}
where the last inequality follows from H\"older's inequality. For the first expectation, since
\begin{align*}
&\eta\Arrowvert\theta_k+\eta b(\theta_k)\Arrowvert_2^2\cdot\Arrowvert\sigma\Arrowvert_{\infty}^2+\frac{1}{2}\eta^2\Arrowvert b(\theta_k)\Arrowvert_2^2+\frac{1}{4}\eta\Arrowvert\theta_k+\eta b(\theta_k)\Arrowvert_2^2\\    \leq&2\Arrowvert\sigma\Arrowvert_{\infty}^2\eta\Arrowvert\theta_k\Arrowvert_2^2+2\Arrowvert\sigma\Arrowvert_{\infty}^2\eta^3\Arrowvert b(\theta_k)\Arrowvert_2^2+\frac{1}{2}\eta^2\Arrowvert b(\theta_k)\Arrowvert_2^2+\frac{1}{2}\eta\Arrowvert\theta_k\Arrowvert_2^2+\frac{1}{2}\eta^3\Arrowvert b(\theta_k)\Arrowvert_2^2\\
\leq& c_1\eta\Arrowvert\theta_k\Arrowvert_2^2+c_2\eta^2\Arrowvert b(\theta_k)\Arrowvert_2^2\leq c_1\eta\Arrowvert\theta_k\Arrowvert_2^2+c_2\eta^2\left(2L^2\Arrowvert\theta_k\Arrowvert_2^2+2\Arrowvert b(0)\Arrowvert_2^2\right)\\
\leq& c_1\eta\Arrowvert\theta_k\Arrowvert_2^2+c_3\eta^2\Arrowvert b(0)\Arrowvert_2^2,
\end{align*}
we obtain
\begin{align}
&\left(\E_{\theta_0}\exp\left\{\Bigg(4c_2\sum_{k=0}^{m-1}\Bigg(\eta\Arrowvert\theta_k+\eta b(\theta_k)\Arrowvert_2^2\cdot\Arrowvert\sigma\Arrowvert_{\infty}^2+\frac{1}{2}\eta^2\Arrowvert b(\theta_k)\Arrowvert_2^2+\frac{1}{4}\eta\Arrowvert\theta_k+\eta b(\theta_k)\Arrowvert_2^2\Bigg)\Bigg)I_A\right\}\right)^{\frac{1}{4}}\nonumber\\
\leq&\left(\E_{\theta_0}\exp\left\{\Bigg(4c_2\sum_{k=0}^{m-1}\Bigg(c_1\eta\Arrowvert\theta_k\Arrowvert_2^2+c_3\eta^2\Arrowvert b(0)\Arrowvert_2^2\Bigg)\Bigg)I_A\right\}\right)^{\frac{1}{4}}\nonumber\\
\leq &c_4\left(\E_{\theta_0}\exp\left\{\left(c_1\sum_{k=0}^{m-1}\eta\Arrowvert\theta_k\Arrowvert_2^2\right)I_A\right\}\right)^{\frac{1}{4}}.
\end{align}
For the second expectation, we take some $\eta_0$ such that $1-2c\eta\Arrowvert\sigma\Arrowvert_{\infty}^2>0$ for any $\eta<\eta_0$,
\begin{align*}
\E_{\theta_0}\exp\left\{6c_2\sum_{k=0}^{m-1}\eta\Arrowvert\sigma(\theta_k)\xi_{k+1}\Arrowvert_2^2\right\}
\leq&\E\exp\left\{c\sum_{k=0}^{m-1}\eta\Arrowvert\sigma\Arrowvert_{\infty}^2\Arrowvert\xi_{k+1}\Arrowvert_2^2\right\}\\
=&\left(\int_{-\infty}^{\infty}\frac{1}{\sqrt{2\pi}}\exp\left\{c\eta\Arrowvert\sigma\Arrowvert_{\infty}^2x^2-\frac{1}{2}x^2\right\}\,\mathrm{d}x\right)^{md}\\
=&\left(1-2c\eta\Arrowvert\sigma\Arrowvert_{\infty}^2\right)^{-\frac{1}{2}md}\leq e^{c\Arrowvert\sigma\Arrowvert_{\infty}^2\eta^{-1}}.
\end{align*}
For the third expectation, choose $\mathbf{R}<c\Arrowvert\nabla\sigma\Arrowvert^{-1}_{\infty}\eta^{-\frac{1}{2}}$, then
$$1-2c\eta\Arrowvert\sigma\Arrowvert_{\infty}^2\Arrowvert\nabla\sigma\Arrowvert_{\infty}^2\left(\mathbf{R}^2+1\right)>0.$$
Hence,
\begin{align*}
&\E_{\theta_0}\left(\exp\left\{5c_2\sum_{k=0}^{m-1}\eta\Arrowvert\mathcal{R}(\theta_k,\xi_{k+1})\Arrowvert_2^2I_{\{|\xi_{k+1}|\leq \mathbf{R}\}}\right\}\right)\\
\leq&\E\left(\exp\left\{c\sum_{k=0}^{m-1}\eta\Arrowvert\sigma\Arrowvert_{\infty}^2\Arrowvert\nabla\sigma\Arrowvert_{\infty}^2\left(\Arrowvert\xi_{k+1}\Arrowvert_2^4+1\right)I_{\{|\xi_{k+1}|\leq \mathbf{R}\}}\right\}\right)\\
\leq&\E\left(\exp\left\{c\sum_{k=0}^{m-1}\eta\Arrowvert\sigma\Arrowvert_{\infty}^2\Arrowvert\nabla\sigma\Arrowvert_{\infty}^2\left(\mathbf{R}^2+1\right)\left(\Arrowvert\xi_{k+1}\Arrowvert_2^2+1\right)I_{\{|\xi_{k+1}|\leq \mathbf{R}\}}\right\}\right)\\
\leq&\left(\E\exp\left\{c\eta\Arrowvert\sigma\Arrowvert_{\infty}^2\Arrowvert\nabla\sigma\Arrowvert_{\infty}^2\left(\mathbf{R}^2+1\right)\left(\Arrowvert\xi_{k+1}\Arrowvert_2^2+1\right)\right\}\right)^m\\
=&\left(e^{c\eta\Arrowvert\sigma\Arrowvert_{\infty}^2\Arrowvert\nabla\sigma\Arrowvert_{\infty}^2\left(\mathbf{R}^2+1\right)}\int_{-\infty}^{\infty}\frac{1}{\sqrt{2\pi}}e^{-\frac{1-2c\eta\Arrowvert\sigma\Arrowvert_{\infty}^2\Arrowvert\nabla\sigma\Arrowvert_{\infty}^2\left(\mathbf{R}^2+1\right)}{2}x^2}\,\mathrm{d}x\right)^{md}\\
=& e^{cd\eta^{-1}\Arrowvert\sigma\Arrowvert_{\infty}^2\Arrowvert\nabla\sigma\Arrowvert_{\infty}^2\left(\mathbf{R}^2+1\right)}\left(1-2c\eta\Arrowvert\sigma\Arrowvert_{\infty}^2\Arrowvert\nabla\sigma\Arrowvert_{\infty}^2\left(\mathbf{R}^2+1\right)\right)^{-\frac{1}{2}md}
\leq e^{c\eta^{-1}\Arrowvert\sigma\Arrowvert_{\infty}^2\Arrowvert\nabla\sigma\Arrowvert_{\infty}^2\left(\mathbf{R}^2+1\right)}.
\end{align*}
Hence, for $\eta<\eta_0$ small enough and $\mathbf{R}<c\Arrowvert\nabla\sigma\Arrowvert^{-1}_{\infty}\eta^{-\frac{1}{2}}$, we have
\begin{align*}
&\E_{\theta_0}\left(\exp\left\{\left(c_1\eta\sum_{k=0}^{m-1}\Arrowvert\theta_k\Arrowvert_2^2\right)I_A\right\}\right)\\
\leq& c_4\left(\E_{\theta_0}\exp\left\{\left(c_1\eta\sum_{k=0}^{m-1}\Arrowvert\theta_k\Arrowvert_2^2\right)I_A\right\}\right)^{\frac{1}{4}}\cdot e^{\frac{1}{2}c_2\Arrowvert\theta_0\Arrowvert^2}\cdot e^{c\Arrowvert\sigma\Arrowvert_{\infty}^2\eta^{-1}\left(1+\Arrowvert\nabla\sigma\Arrowvert_{\infty}^2\left(\mathbf{R}^2+1\right)\right)},
\end{align*}
which implies that
\begin{align*}
\E_{\theta_0}\left(\exp\left\{\left(c_1\eta\sum_{k=0}^{m-1}\Arrowvert\theta_k\Arrowvert_2^2\right)I_A\right\}\right)\leq Ce^{c_1\Arrowvert\theta_0\Arrowvert_2^2}\cdot e^{c_2\Arrowvert\sigma\Arrowvert_{\infty}^2\eta^{-1}\left(1+\Arrowvert\nabla\sigma\Arrowvert_{\infty}^2\left(\mathbf{R}^2+1\right)\right)}.
\end{align*}
This leads to
\begin{align}\label{Etheta0}
\E_{\theta_0}\left(\exp\left\{\left(c\eta\sum_{k=0}^{m-1}\Arrowvert b(\theta_k)\Arrowvert_2^2\right)I_A\right\}\right)
&\leq\E_{\theta_0}\left(\exp\left\{\left(c_1\eta\sum_{k=0}^{m-1}\Arrowvert\theta_k\Arrowvert_2^2\right)I_A\right\}\right)\cdot e^{2c\Arrowvert b(0)\Arrowvert_2^2\eta^{-1}}\nonumber\\
&\leq Ce^{c_1\Arrowvert\theta_0\Arrowvert_2^2}\cdot e^{c_2\Arrowvert\sigma\Arrowvert_{\infty}^2\eta^{-1}\left(1+\Arrowvert\nabla\sigma\Arrowvert_{\infty}^2\left(\mathbf{R}^2+1\right)\right)},
\end{align}
and then
\begin{align}\label{P1}
P_1
\leq&\E_{\theta_0}\left(\exp\left\{\left(c\eta\sum_{k=0}^{m-1}\Arrowvert b(\theta_k)\Arrowvert_2^2\right)I_A\right\}\right)\cdot e^{-cx}\nonumber\\
\leq &Ce^{c_1\Arrowvert\theta_0\Arrowvert
_2^2}\cdot e^{c_2\Arrowvert\sigma\Arrowvert_{\infty}^2\eta^{-1}\left(1+\Arrowvert\nabla\sigma\Arrowvert_{\infty}^2\left(\mathbf{R}^2+1\right)\right)}e^{-c_3x}.
\end{align}

Next, for $P_2$, we have
\begin{align*}
P_2
=\P_{\theta_0}\left(\eta\sum_{k=0}^{m-1}\Arrowvert b(\theta_k)\Arrowvert_2^2>x,A^c\right)\leq\P\left(A^c\right)=\P\left(\max_{1\leq i\leq m}|\xi_i|>\mathbf{R}\right)=1-\prod_{i=1}^m\P\left(|\xi_i|^2\leq \mathbf{R}^2\right).
\end{align*}
Since $|\xi_{i}|^2\sim\chi^2(d)$, we can obtain $P\left(|\xi_i|^2> \mathbf{R}^2\right)\leq (1+\frac{\mathbf{R}^2}{d})^{-\frac{d}{2}}e^{-\frac{\mathbf{R}^2}{2}}$, thus
\begin{align*}
\P\left(|\xi_i|^2\leq \mathbf{R}^2\right)=1-\P\left(|\xi_i|^2> \mathbf{R}^2\right)\geq 1-\left(1+\frac{\mathbf{R}^2}{d}\right)^{-\frac{d}{2}}e^{-\frac{\mathbf{R}^2}{2}}.
\end{align*}
Choose $\mathbf{R}>c_1\eta^{-a}$ for any $a>0$. Hence we have
\begin{align}\label{P2}
P_2
\leq 1-\prod_{i=1}^m\P\left(|\xi_i|^2\leq \mathbf{R}^2\right)
\leq 1-\left(1-\left(1+\frac{\mathbf{R}^2}{d}\right)^{-\frac{d}{2}}e^{-\frac{\mathbf{R}^2}{2}}\right)^m\leq Ce^{-\frac{\mathbf{R}^2}{2}}.
\end{align}
By \eqref{P1} and \eqref{P2}, we have \eqref{gP0}.

Finally, with $\theta_0$ constrained by {\bf Assumption} {\bf\ref{conditions2}}, we can obtain by \eqref{Etheta0} that
\begin{align*}
    \E\left(\exp\left\{\left(c\eta\sum_{k=0}^{m-1}\Arrowvert b(\theta_k)\Arrowvert_2^2\right)I_A\right\}\right)=&\E\left(\E_{\theta_0}\left(\exp\left\{\left(c\eta\sum_{k=0}^{m-1}\Arrowvert b(\theta_k)\Arrowvert_2^2\right)I_A\right\}\right)\right)\\\leq &c_1 e^{c_2\Arrowvert\sigma\Arrowvert_{\infty}^2\eta^{-1}\left(1+\Arrowvert\nabla\sigma\Arrowvert_{\infty}^2\left(\mathbf{R}^2+1\right)\right)}.
\end{align*}
Following the same procedures as for $\P_{\theta_0}\left(\eta\sum_{k=0}^{m-1}\Arrowvert b(\theta_k)\Arrowvert_2^2>x\right)$, we can get \eqref{gP}. 

\section{Proof of Proposition \ref{remainder}}\label{AppendixB}
We now give some estimations for the tail probability $\P\big(|\mathcal{R}_\eta|\geq y\big)$ for $0<y=o(\eta^{-1/2})$. Clearly, it holds
\begin{align*}
\P\big(|\mathcal{R}_\eta|\geq y\big)\leq\sum_{i=1}^{13}\P\big(|\mathcal{R}_{\eta,i}|\geq y/13\big)=:\sum_{i=1}^{13} I_i.
\end{align*}

Then we give estimates for $I_i,i=1,2,\cdots,13$.

\textbf{(1) Estimation of $I_1$.} By the boundedness of $\Arrowvert f_h\Arrowvert_2$, as established in \eqref{boundedness}, we have $|\mathcal{R}_{\eta}|\leq c_1\eta^{1/2}$, and thus for all $y> c\eta^{1/2}$,
$$
\P\big(\lvert\mathcal{R}_{\eta,1}\rvert\geq y/13\big)\leq e^{c_1^2}e^{-y^2\eta^{-1}}.
$$

\textbf{(2) Estimation of $I_2$.} Denote $\zeta_{k+1}=\left<\nabla^2f_h,(\sigma(\theta_k)\xi_{k+1})(\sigma(\theta_k)\xi_{k+1})^{\rm T}-\sigma(\theta_k)\sigma(\theta_k)^{\rm T}\right>_{\rm HS}$. Then it is easy to see that $\left(\zeta_i,\mathcal{F}_i\right)_{i\geq 1}$ is a sequence of martingale differences, and by \eqref{boundedness}, we have
\begin{align*}
    |\zeta_{k+1}|
    &\leq \Arrowvert\nabla^2f_h(\theta_k)\Arrowvert\cdot\Arrowvert(\sigma(\theta_k)\xi_{k+1})(\sigma(\theta_k)\xi_{k+1})^{\rm T}-\sigma(\theta_k)\sigma(\theta_k)^{\rm T}\Arrowvert\\
    &\leq c\left(1+\Arrowvert\xi_{k+1}\Arrowvert_2^2\right),
\end{align*}
where $c=\Arrowvert\nabla^2f_h(\theta_k)\Arrowvert\cdot\Arrowvert\sigma\Arrowvert_{\infty}^2$ and which implies that there exists some positive constant $c$ such that
$$
\E\left(|\zeta_{k+1}|^2\exp\left\{c|\zeta_{k+1}|\right\}|\mathcal{F}_k\right)<\infty.
$$
Therefore, using Proposition \ref{martingale differences} with $\alpha=1$, we have for all $y>c\eta^{1/2}$,
\begin{align*}
    I_2&=\P\left(\left|\frac{\eta^{\frac{3}{2}}}{2}\sum_{k=0}^{m-1}\left<\nabla^2f_h(\theta_k),\left(\sigma(\theta_k)\xi_{k+1}\right)\left(\sigma(\theta_k)\xi_{k+1}\right)^{\rm T}-\sigma(\theta_k)\sigma(\theta_k)^{\rm T}\right>_{\rm{HS}}\right|\geq y/13\right)\\
    &=\P\left(\left|\sum_{k=0}^{m-1}\zeta_{k+1}\right|\geq \frac{2}{13}y\eta^{-3/2}\right)\leq c_1\exp\left\{-\frac{\left(y\eta^{-3/2}\right)^2}{c_2\left(\eta^{-2}+y\eta^{-3/2}\right)}\right\}\\
    &\leq c_1\exp\left\{-\frac{y^2\eta^{-1}}{c_2(1+\eta^{1/2}y)}\right\}.
\end{align*}

\textbf{(3) Estimation of $I_3$.} According to $A=\left\{\left|\xi_i\right|\leq \mathbf{R},i=1,\cdots,m\right\}$, we have
\begin{align*}
\P\left(\mathcal{R}_{\eta,3}\geq y/13\right)
\leq\P\left(\mathcal{R}_{\eta,3}\geq y/13,A\right)+\P\left(A^c\right)
\leq\P\left(\mathcal{R}_{\eta,3}\geq y/13,A\right)+Ce^{-\frac{\mathbf{R}^2}{2}}.
\end{align*}
For $\P\left(\mathcal{R}_{\eta,3}\geq y/13,A\right)$, applying Markov inequality,
\begin{align*}
&\P\left(\mathcal{R}_{\eta,3}\geq y/13,A\right)\\
\leq&\E\exp\left\{\left(\sum_{k=0}^{m-1}\left<(c\eta)^{\frac{1}{2}}\left(\left(\nabla^2f_h(\theta_k)\right)^{\rm T}+\nabla^2f_h(\theta_k)\right)b(\theta_k),\sigma(\theta_k)\xi_{k+1}\right>\right) I_A\right\}e^{-cy\eta^{-\frac{3}{2}}}.
\end{align*}
Taking $\Psi_1(\theta_k)=(c\eta)^{\frac{1}{2}}\left(\left(\nabla^2f_h(\theta_k)\right)^{\rm T}+\nabla^2f_h(\theta_k)\right)b(\theta_k)$ and $\Psi_2(\theta_k,\xi_{k+1})=0$, we obtain by Lemma \ref{E_k} and Proposition \ref{g},
\begin{align*}
&\E\exp\left\{\left(\sum_{k=0}^{m-1}\left<(c\eta)^{\frac{1}{2}}\left(\left(\nabla^2f_h(\theta_k)\right)^{\rm T}+\nabla^2f_h(\theta_k)\right)b(\theta_k),\sigma(\theta_k)\xi_{k+1}\right>\right)I_A\right\}\\
&\leq\left(\E\exp\left\{\left(\sum_{k=0}^{m-1}2c\eta\left\Arrowvert\left(\left(\nabla^2f_h(\theta_k)\right)^{\rm T}+\nabla^2f_h(\theta_k)\right)b(\theta_k)\right\Arrowvert_2^2\Arrowvert\sigma\Arrowvert_{\infty}^2\right)I_A\right\}\right)^{\frac{1}{2}}\\
&\leq\left(\E\exp\left\{\left(\sum_{k=0}^{m-1}c\eta\Arrowvert b(\theta_k)\Arrowvert_2^2\right)I_A\right\}\right)^{\frac{1}{2}}\leq c_1e^{c_2\Arrowvert\sigma\Arrowvert_{\infty}^2\eta^{-1}\left(1+\Arrowvert\nabla\sigma\Arrowvert_{\infty}^2(\mathbf{R}^2+1)\right)}.
\end{align*}
Thus we have
\begin{align*}
\P\left(\mathcal{R}_{\eta,3}\geq y/13\right)
\leq c_1e^{c_2\Arrowvert\sigma\Arrowvert_{\infty}^2\eta^{-1}\left(1+\Arrowvert\nabla\sigma\Arrowvert_{\infty}^2(\mathbf{R}^2+1)\right)}e^{-c_3y\eta^{-\frac{3}{2}}}+Ce^{-\frac{\mathbf{R}^2}{2}}.
\end{align*}

By the same argument, we obtain the same bound for $\P\left(\mathcal{R}_{\eta,3}\leq -y/13\right)$. Hence for $\mathbf{R}$ satisfying \eqref{R} and $y>\max\left\{c\eta^{1/2},\eta^{\frac{1}{2}}\mathbf{R}^2\Arrowvert\nabla\sigma\Arrowvert_{\infty}^2\right\}$, we have
\begin{align*}
I_3=\P\left(\left|\mathcal{R}_{\eta,3}\right|\geq y/13\right)
\leq c_1e^{c_2\Arrowvert\sigma\Arrowvert_{\infty}^2\eta^{-1}\left(1+\Arrowvert\nabla\sigma\Arrowvert_{\infty}^2(\mathbf{R}^2+1)\right)}e^{-c_3y\eta^{-\frac{3}{2}}}+Ce^{-\frac{\mathbf{R}^2}{2}}.
\end{align*}

\textbf{(4) Estimation of $I_4$}. For all $y>c\eta^{\frac{1}{2}}$,
\begin{align*}
&\P\left(\left|\mathcal{R}_{\eta,4}\right|\geq y/13\right)\\
=&\P\left(\left|\frac{\eta^2}{2}\sum_{k=0}^{m-1}\int_0^1\left(1-t\right)^2\nabla_{\sigma(\theta_k)\xi_{k+1}}\nabla_{\sigma(\theta_k)\xi_{k+1}}\nabla_{\sigma(\theta_k)\xi_{k+1}}f_h(\theta_k+t\triangle\theta_k)\mathrm{d}t\right|\geq y/13\right)\\
\leq&\P\Bigg(\left|\sum_{k=0}^{m-1}\int_0^1\nabla_{\sigma(\theta_k)\xi_{k+1}}\nabla_{\sigma(\theta_k)\xi_{k+1}}\nabla_{\sigma(\theta_k)\xi_{k+1}}f_h(\theta_k+t\triangle\theta_k)\right.\\&\left.\qquad\qquad\qquad\qquad\qquad\qquad-\nabla_{\sigma(\theta_k)\xi_{k+1}}\nabla_{\sigma(\theta_k)\xi_{k+1}}\nabla_{\sigma(\theta_k)\xi_{k+1}}f_h(\theta_k)\mathrm{d}t\right|\geq \frac{y\eta^{-2}}{13}\Bigg)\\
&+\P\left(\left|\sum_{k=0}^{m-1}\int_0^1\nabla_{\sigma(\theta_k)\xi_{k+1}}\nabla_{\sigma(\theta_k)\xi_{k+1}}\nabla_{\sigma(\theta_k)\xi_{k+1}}f_h(\theta_k)\mathrm{d}t\right|\geq \frac{y\eta^{-2}}{13}\right)\\
:=&I_{\eta,4,1}+I_{\eta,4,2}.
\end{align*}

For $I_{\eta,4,1}$, by \eqref{boundedness} and the fact $\triangle\theta_k=\eta b(\theta_k)+\sqrt{\eta}\sigma(\theta_k)\xi_{k+1}+\frac{1}{2}\eta\mathcal{R}(\theta_k,\xi_{k+1})$,  we get
\begin{align*}
I_{\eta,4,1}
=&\P\left(\left|\sum_{k=0}^{m-1}\int_0^1\int_0^1\nabla_{t\triangle\theta_k}\nabla_{\sigma(\theta_k)\xi_{k+1}}\nabla_{\sigma(\theta_k)\xi_{k+1}}\nabla_{\sigma(\theta_k)\xi_{k+1}}f_h(\theta_k+tt^{\prime}\triangle\theta_k)\mathrm{d}t^{\prime}\mathrm{d}t\right|\geq cy\eta^{-2}\right)\\
\leq&\P\left(\sum_{k=0}^{m-1}\Arrowvert\eta b(\theta_k)+\sqrt{\eta}\sigma(\theta_k)\xi_{k+1}+\frac{1}{2}\eta\mathcal{R}(\theta_k,\xi_{k+1})\Arrowvert_2\Arrowvert\sigma(\theta_k)\xi_{k+1}\Arrowvert_2^3\geq 3c_1y\eta^{-2}\right)\\
\leq&\P\left(\sum_{k=0}^{m-1}\sqrt{\eta}\Arrowvert\sigma(\theta_k)\xi_{k+1}\Arrowvert_2^4\geq c_1y\eta^{-2}\right)+\P\left(\sum_{k=0}^{m-1}\Arrowvert\eta b(\theta_k)\Arrowvert_2\Arrowvert\sigma(\theta_k)\xi_{k+1}\Arrowvert_2^3\geq c_1y\eta^{-2}\right)\\
&+\P\left(\sum_{k=0}^{m-1}\Arrowvert\eta\mathcal{R}(\theta_k,\xi_{k+1})\Arrowvert_2\Arrowvert\sigma(\theta_k)\xi_{k+1}\Arrowvert_2^3\geq 2c_1y\eta^{-2}\right):=I_{\eta,4,1}^{\prime}+I_{\eta,4,1}^{\prime\prime}+I_{\eta,4,1}^{\prime\prime\prime}.
\end{align*}
First, estimate $I_{\eta,4,1}^{\prime}$. In fact
\begin{align*}
I_{\eta,4,1}^{\prime}
&=\P\left(\sum_{k=0}^{m-1}\sqrt{\eta}\Arrowvert\sigma(\theta_k)\xi_{k+1}\Arrowvert_2^4\geq c_1y\eta^{-2}\right)
=\P\left(\sum_{k=0}^{m-1}\Arrowvert\sigma(\theta_k)\xi_{k+1}\Arrowvert_2^4\geq c_1y\eta^{-\frac{5}{2}}\right)\\
&=\P\left(\sum_{k=0}^{m-1}\left(\Arrowvert\sigma(\theta_k)\xi_{k+1}\Arrowvert_2^4-\E\Arrowvert\sigma(\theta_k)\xi_{k+1}\Arrowvert_2^4\right)\geq \big(c_1y\eta^{-\frac{5}{2}}-\sum_{k=0}^{m-1}\E\Arrowvert\sigma(\theta_k)\xi_{k+1}\Arrowvert_2^4\big)\right).
\end{align*}
Note that there exists some positive constant $c$ such that
\begin{align*}
\E\left(\left(\Arrowvert\sigma(\theta_k)\xi_{k+1}\Arrowvert_2^4-\E\Arrowvert\sigma(\theta_k)\xi_{k+1}\Arrowvert_2^4\right)^2\exp\left\{c\left|\Arrowvert\sigma(\theta_k)\xi_{k+1}\Arrowvert_2^4-\E\Arrowvert\sigma(\theta_k)\xi_{k+1}\Arrowvert_2^4\right|^{1/2}\right\}\bigg|\mathcal{F}_k\right)<\infty.
\end{align*}
Using Proposition \ref{martingale differences} with $\alpha=1/2$, we have for all $y>c\eta^{1/2}$,
\begin{align*}
I_{\eta,4,1}^{\prime}
&\leq c\exp\left\{-\frac{\left(y\eta^{-5/2}\right)^2}{c_1\left(\eta^{-2}+\left(y\eta^{-5/2}\right)^{3/2}\right)}\right\}\leq ce^{-c_2y^{1/2}\eta^{-5/4}}.
\end{align*}
Second, we estimate $I_{\eta,4,1}^{\prime\prime}$. Using H\"older's inequality, we get for all $y>c\eta^{1/2}$,
\begin{align*}
&I_{\eta,4,1}^{\prime\prime}
=\P\left(\sum_{k=0}^{m-1}\Arrowvert\eta b(\theta_k)\Arrowvert_2\Arrowvert\sigma(\theta_k)\xi_{k+1}\Arrowvert_2^3\geq c_1y\eta^{-2}\right)\\
\leq&\P\left(\eta\left(\sum_{k=0}^{m-1}\Arrowvert b(\theta_k)\Arrowvert_2^2\right)^{\frac{1}{2}}\left(\sum_{k=0}^{m-1}\Arrowvert\sigma(\theta_k)\xi_{k+1}\Arrowvert_2^6\right)^{\frac{1}{2}}\geq c_1y\eta^{-2}\right)\\
\leq&\P\left(\eta\left(\sum_{k=0}^{m-1}\Arrowvert b(\theta_k)\Arrowvert_2^2\right)^{\frac{1}{2}}\left(\sum_{k=0}^{m-1}\Arrowvert\sigma(\theta_k)\xi_{k+1}\Arrowvert_2^6\right)^{\frac{1}{2}}\geq c_1y\eta^{-2},\eta\sum_{k=0}^{m-1}\Arrowvert b(\theta_k)\Arrowvert_2^2\geq Cy^{\frac{2}{3}}\eta^{-\frac{4}{3}}\right)\\
&+\P\left(\eta\left(\sum_{k=0}^{m-1}\Arrowvert b(\theta_k)\Arrowvert_2^2\right)^{\frac{1}{2}}\left(\sum_{k=0}^{m-1}\Arrowvert\sigma(\theta_k)\xi_{k+1}\Arrowvert_2^6\right)^{\frac{1}{2}}\geq c_1y\eta^{-2},\eta\sum_{k=0}^{m-1}\Arrowvert b(\theta_k)\Arrowvert_2^2< Cy^{\frac{2}{3}}\eta^{-\frac{4}{3}}\right)\\
\leq&\P\left(\eta\sum_{k=0}^{m-1}\Arrowvert b(\theta_k)\Arrowvert_2^2\geq Cy^{\frac{2}{3}}\eta^{-\frac{4}{3}}\right)+\P\left(\sum_{k=0}^{m-1}\Arrowvert\sigma(\theta_k)\xi_{k+1}\Arrowvert_2^6\geq c_1y^{\frac{4}{3}}\eta^{-\frac{11}{3}}\right).
\end{align*}
Since there exists some positive constant c such that
\begin{align*}
\E\left(\left(\Arrowvert\sigma(\theta_k)\xi_{k+1}\Arrowvert_2^6-\E\Arrowvert\sigma(\theta_k)\xi_{k+1}\Arrowvert_2^6\right)^2\exp\left\{c\left|\Arrowvert\sigma(\theta_k)\xi_{k+1}\Arrowvert_2^6-\E\Arrowvert\sigma(\theta_k)\xi_{k+1}\Arrowvert_2^6\right|^{1/3}\right\}\bigg|\mathcal{F}_k\right)<\infty,
\end{align*}
using Proposition \ref{martingale differences} with $\alpha=1/3$,  we have for all $y>c\eta^{1/2}$,
\begin{align*}
&\P\left(\sum_{k=0}^{m-1}\Arrowvert\sigma(\theta_k)\xi_{k+1}\Arrowvert_2^6\geq c_1y^{\frac{4}{3}}\eta^{-\frac{11}{3}}\right)\\
=&\P\left(\sum_{k=0}^{m-1}\left(\Arrowvert\sigma(\theta_k)\xi_{k+1}\Arrowvert_2^6-\E\Arrowvert\sigma(\theta_k)\xi_{k+1}\Arrowvert_2^6\right)\geq \big(c_1y^{\frac{4}{3}}\eta^{-\frac{11}{3}}-\sum_{k=0}^{m-1}\E\Arrowvert\sigma(\theta_k)\xi_{k+1}\Arrowvert_2^6\big)\right)\\
\leq& c\exp\left\{-\frac{\left(y^{4/3}\eta^{-11/3}\right)^2}{c_1\left(\eta^{-2}+\left(y^{4/3}\eta^{-11/3}\right)^{5/3}\right)}\right\}\leq ce^{-c_2y^{4/9}\eta^{-11/9}}.
\end{align*}
For any $a>0$ and $\mathbf{R}$ satisfying
\begin{align}\label{R2}
    c_1\eta^{-a}<\mathbf{R}<c_2\Arrowvert\nabla\sigma\Arrowvert_{\infty}^{-1}\eta^{-\frac{1}{3}},
\end{align}
together with Proposition \ref{g}, we deduce that for all $y>\max\left\{c\eta^{1/2},\eta^{1/2}\mathbf{R}^3\Arrowvert\nabla\sigma\Arrowvert_{\infty}^3\right\}$,
\begin{align*}
I_{\eta,4,1}^{\prime\prime}
\leq c_1e^{c_2\Arrowvert\sigma\Arrowvert_{\infty}^2\eta^{-1}\left(1+\Arrowvert\nabla\sigma\Arrowvert_{\infty}^2(\mathbf{R}^2+1)\right)}e^{-c_3y^{\frac{2}{3}}\eta^{-\frac{4}{3}}}+c_4e^{-\frac{\mathbf{R}^2}{2}}+ce^{-c_5y^{4/9}\eta^{-11/9}}.
\end{align*}
Third, we give an estimation for $I_{\eta,4,1}^{\prime\prime\prime}$. Denote $T_{k+1}=\Arrowvert\mathcal{R}(\theta_k,\xi_{k+1})\Arrowvert_2\Arrowvert\sigma(\theta_k)\xi_{k+1}\Arrowvert_2^3$, then $\left(T_{k+1},\mathcal{F}_{k+1}\right)_{k\geq 0}$ is a sequence of martingale differences and $\left|T_{k+1}\right|\leq c\left(\Arrowvert\xi_{k+1}\Arrowvert_2^5+\Arrowvert\xi_{k+1}\Arrowvert_2^3\right)$. Therefore, there exists some positive constant c such that
\begin{align*}
\E\left(T_{k+1}^2\exp\left\{c|T_{k+1}|^{2/5}\right\}\bigg|\mathcal{F}_k\right)<\infty.
\end{align*}
For all $y>c\eta^{1/2}$, one can write by Proposition \ref{martingale differences} with $\alpha=2/5$ that
\begin{align*}
I_{\eta,4,1}^{\prime\prime\prime}
=&\P\left(\sum_{k=0}^{m-1}\Arrowvert\mathcal{R}(\theta_k,\xi_{k+1})\Arrowvert_2\Arrowvert\sigma(\theta_k)\xi_{k+1}\Arrowvert_2^3\geq 2c_1y\eta^{-3}\right)
=\P\left(\sum_{k=0}^{m-1}T_{k+1}\geq 2c_1y\eta^{-3}\right)\\
\leq& c\exp\left\{-\frac{\left(y\eta^{-3}\right)^2}{c_1\left(\eta^{-2}+\left(y\eta^{-3}\right)^{8/5}\right)}\right\}\leq ce^{-c_2y^{2/5}\eta^{-6/5}}.
\end{align*}
Finally, combining the estimations of $I_{\eta,4,1}^{\prime}$, $I_{\eta,4,1}^{\prime\prime}$ and $I_{\eta,4,1}^{\prime\prime\prime}$, for $\mathbf{R}$ satisfying \eqref{R2} and $y>\max\left\{c\eta^{1/2},\eta^{1/2}\mathbf{R}^3\Arrowvert\nabla\sigma\Arrowvert_{\infty}^3\right\}$, we have
\begin{align*}
I_{\eta,4,1}\leq c_1e^{c_2\Arrowvert\sigma\Arrowvert_{\infty}^2\eta^{-1}\left(1+\Arrowvert\nabla\sigma\Arrowvert_{\infty}^2(\mathbf{R}^2+1)\right)}e^{-c_3y^{\frac{2}{3}}\eta^{-\frac{4}{3}}}+c_4e^{-\frac{\mathbf{R}^2}{2}}+ce^{-c_5y^{4/9}\eta^{-11/9}}+ce^{-c_6y^{2/5}\eta^{-6/5}}.
\end{align*}

On the other hand, we estimate $I_{\eta,4,2}$. Denote
$$
h_{k+1}=\int_0^1\nabla_{\sigma(\theta_k)\xi_{k+1}}\nabla_{\sigma(\theta_k)\xi_{k+1}}\nabla_{\sigma(\theta_k)\xi_{k+1}}f_h(\theta_k)\mathrm{d}t.
$$
Then $\left(h_{k+1},\mathcal{F}_{k+1}\right)_{k\geq 0}$ is a sequence of martingale differences. Moreover, by the boundedness of $\Arrowvert\nabla^3f_h\Arrowvert$, as established in \eqref{boundedness}, and together with \textbf{Assumption} \ref{conditions}, we have $|h_{k+1}|\leq c\Arrowvert\xi_{k+1}\Arrowvert_2^3$ where $c=\Arrowvert\nabla^3f_h(\theta_k)\Arrowvert\Arrowvert\sigma(\theta_k)\Arrowvert^3$. Therefore, there exists some positive constant $c$ such that $\E\left(|h_{k+1}|^2\exp\{c|h_{k+1}|^{2/3}\}|\mathcal{F}_k\right)<\infty$. Using Proposition \ref{martingale differences} with $\alpha=2/3$, we have for all $y>c\eta^{1/2}$,
\begin{align*}
    I_{\eta,4,2}
    &\leq \P\left(\left|\sum_{k=0}^{m-1}h_{k+1}\right|\geq cy\eta^{-2}\right)\leq c\exp\left\{-\frac{\left(y\eta^{-2}\right)^2}{c_1\left(\eta^{-2}+(y\eta^{-2})^{4/3}\right)}\right\}\\
    &\leq ce^{-c_2y^{2/3}\eta^{-4/3}}. 
\end{align*}

Hence, we get for all $y>\max\left\{c\eta^{1/2},\eta^{1/2}\mathbf{R}^3\Arrowvert\nabla\sigma\Arrowvert_{\infty}^3\right\}$ and $\mathbf{R}$ satisfying \eqref{R2},
\begin{align*}
I_4\leq c_1e^{c_2\Arrowvert\sigma\Arrowvert_{\infty}^2\eta^{-1}\left(1+\Arrowvert\nabla\sigma\Arrowvert_{\infty}^2(\mathbf{R}^2+1)\right)}e^{-c_3y^{\frac{2}{3}}\eta^{-\frac{4}{3}}}+c_4e^{-\frac{\mathbf{R}^2}{2}}+ce^{-c_5y^{4/9}\eta^{-11/9}}+ce^{-c_6y^{2/5}\eta^{-6/5}}.
\end{align*}

\textbf{(5) Estimation of $I_5$.} By \eqref{boundedness} and Proposition \ref{g}, for all $y>\max\left\{c\eta^{1/2},\eta^{1/2}\mathbf{R}^3\Arrowvert\nabla\sigma\Arrowvert_{\infty}^3\right\}$ and $\mathbf{R}$ satisfying \eqref{R2}, we have
\begin{align*}
&\P\left(\left|\mathcal{R}_{\eta,5}\right|\geq y/13\right)\\\leq&\P\left(\eta^{\frac{5}{2}}\sum_{k=0}^{m-1}\Arrowvert b(\theta_k)\Arrowvert_2^2\geq cy\right)+\P\left(\eta^{\frac{7}{2}}\sum_{k=0}^{m-1}\Arrowvert b(\theta_k)\Arrowvert_2^3\geq cy\right)\\
\leq&\P\left(\eta\sum_{k=0}^{m-1}\Arrowvert b(\theta_k)\Arrowvert_2^2\geq cy\eta^{-\frac{3}{2}}\right)+\P\left(\eta\sum_{k=0}^{m-1}\Arrowvert b(\theta_k)\Arrowvert_2^2\geq cy^{\frac{2}{3}}\eta^{-\frac{4}{3}}\right)\\
\leq &2\P\left(\eta\sum_{k=0}^{m-1}\Arrowvert b(\theta_k)\Arrowvert_2^2\geq cy^{\frac{2}{3}}\eta^{-\frac{4}{3}}\right)
\leq c_1e^{c_2\Arrowvert\sigma\Arrowvert_{\infty}^2\eta^{-1}\left(1+\Arrowvert\nabla\sigma\Arrowvert_{\infty}^2(\mathbf{R}^2+1)\right)}e^{-c_3y^{\frac{2}{3}}\eta^{-\frac{4}{3}}}+c_4e^{-\frac{\mathbf{R}^2}{2}}.
\end{align*}

\textbf{(6) Estimation of $I_6$.} By \eqref{boundedness}, we have for all $y>c\eta^{1/2}$,
\begin{align*}
I_6
\leq&\P\left(\eta^{\frac{5}{2}}\sum_{k=0}^{m-1}\left(\Arrowvert b(\theta_k)\Arrowvert_2\Arrowvert\sigma(\theta_k)\xi_{k+1}\Arrowvert_2^2+\sqrt{\eta}\Arrowvert b(\theta_k)\Arrowvert_2^2\Arrowvert\sigma(\theta_k)\xi_{k+1}\Arrowvert_2\right)\geq cy\right)\\
\leq &\P\left(\sum_{k=0}^{m-1}\Arrowvert b(\theta_k)\Arrowvert_2\Arrowvert\sigma(\theta_k)\xi_{k+1}\Arrowvert_2^2\geq cy\eta^{-\frac{5}{2}}\right)+\P\left(\sum_{k=0}^{m-1}\Arrowvert b(\theta_k)\Arrowvert_2^2\Arrowvert\sigma(\theta_k)\xi_{k+1}\Arrowvert_2\geq cy\eta^{-3}\right)\\
:=&I_{6,1}+I_{6,2}.
\end{align*}

For $I_{6,1}$, by the same argument as the estimation of $I_{\eta,4,1}^{\prime\prime}$, we have
\begin{align*}
I_{6,1}
&\leq\P\left(\eta\sum_{k=0}^{m-1}\Arrowvert b(\theta_k)\Arrowvert_2^2\geq Cy^{\frac{2}{3}}\eta^{-\frac{4}{3}}\right)+\P\left(\sum_{k=0}^{m-1}\Arrowvert\sigma(\theta_k)\xi_{k+1}\Arrowvert_2^4\geq cy^{4/3}\eta^{-8/3}\right).
\end{align*}
There exists some positive constant c such that
\begin{align*}
\E\left(\left|\Arrowvert\sigma(\theta_k)\xi_{k+1}\Arrowvert_2^4-\E\Arrowvert\sigma(\theta_k)\xi_{k+1}\Arrowvert_2^4\right|^2\exp\left\{c\left|\Arrowvert\sigma(\theta_k)\xi_{k+1}\Arrowvert_2^4-\E\Arrowvert\sigma(\theta_k)\xi_{k+1}\Arrowvert_2^4\right|^{1/2}\right\}\bigg|\mathcal{F}_k\right)<\infty.
\end{align*}
Using Proposition \ref{martingale differences} with $\alpha=1/2$, we have for all $y>c\eta^{1/2}$,
\begin{align*}
&\P\left(\sum_{k=0}^{m-1}\Arrowvert\sigma(\theta_k)\xi_{k+1}\Arrowvert_2^4\geq cy^{4/3}\eta^{-8/3}\right)\\
=&\P\left(\sum_{k=0}^{m-1}\left(\Arrowvert\sigma(\theta_k)\xi_{k+1}\Arrowvert_2^4-\E\Arrowvert\sigma(\theta_k)\xi_{k+1}\Arrowvert_2^4\right)\geq \big(cy^{4/3}\eta^{-8/3}-\sum_{k=0}^{m-1}\E\Arrowvert\sigma(\theta_k)\xi_{k+1}\Arrowvert_2^4\big)\right)\\
\leq& c\exp\left\{-\frac{\left(y^{4/3}\eta^{-8/3}\right)^2}{c_1\left(\eta^{-2}+\left(y^{4/3}\eta^{-8/3}\right)^{3/2}\right)}\right\}\leq ce^{-c_2y^{2/3}\eta^{-4/3}}.
\end{align*}
Hence, for all $y>\max\left\{c\eta^{1/2},\eta^{1/2}\mathbf{R}^3\Arrowvert\nabla\sigma\Arrowvert_{\infty}^3\right\}$ and $\mathbf{R}$ satisfying \eqref{R2}, we have
\begin{align*}
I_{6,1}=c_1e^{c_2\Arrowvert\sigma\Arrowvert_{\infty}^2\eta^{-1}\left(1+\Arrowvert\nabla\sigma\Arrowvert_{\infty}^2(\mathbf{R}^2+1)\right)}e^{-c_3y^{\frac{2}{3}}\eta^{-\frac{4}{3}}}+c_4e^{-\frac{\mathbf{R}^2}{2}}.
\end{align*}

Next, for $I_{6,2}$, one can write from Proposition \ref{g} that for $y>\max\left\{c\eta^{1/2},\eta^{1/2}\mathbf{R}^3\Arrowvert\nabla\sigma\Arrowvert_{\infty}^3\right\}$,
\begin{align*}
I_{6,2}
&=\P\left(\sum_{k=0}^{m-1}\Arrowvert b(\theta_k)\Arrowvert_2^2\Arrowvert\sigma(\theta_k)\xi_{k+1}\Arrowvert_2\geq cy\eta^{-3}\right)\\
&\leq\P\left(\sum_{k=0}^{m-1}\Arrowvert b(\theta_k)\Arrowvert_2^2Cy^{1/3}\eta^{-2/3}\geq cy\eta^{-3}\right)+\sum_{k=0}^{m-1}\P\left(\Arrowvert\sigma(\theta_k)\xi_{k+1}\Arrowvert_2\geq Cy^{1/3}\eta^{-2/3}\right)\\
&=\P\left(\eta\sum_{k=0}^{m-1}\Arrowvert b(\theta_k)\Arrowvert_2^2\geq cy^{2/3}\eta^{-4/3}\right)+\eta^{-2}\exp\left\{-\frac{y^{2/3}\eta^{-4/3}}{2}\right\}\\
&\leq c_1e^{c_2\Arrowvert\sigma\Arrowvert_{\infty}^2\eta^{-1}\left(1+\Arrowvert\nabla\sigma\Arrowvert_{\infty}^2(\mathbf{R}^2+1)\right)}e^{-c_3y^{\frac{2}{3}}\eta^{-\frac{4}{3}}}+c_4e^{-\frac{\mathbf{R}^2}{2}}.
\end{align*}

Hence, for all $y>\max\left\{c\eta^{1/2},\eta^{1/2}\mathbf{R}^3\Arrowvert\nabla\sigma\Arrowvert_{\infty}^3\right\}$ and $\mathbf{R}$ satisfying \eqref{R2}, we can conclude
\begin{align*}
I_6\leq c_1e^{c_2\Arrowvert\sigma\Arrowvert_{\infty}^2\eta^{-1}\left(1+\Arrowvert\nabla\sigma\Arrowvert_{\infty}^2(\mathbf{R}^2+1)\right)}e^{-c_3y^{\frac{2}{3}}\eta^{-\frac{4}{3}}}+c_4e^{-\frac{\mathbf{R}^2}{2}}.
\end{align*}

\textbf{(7) Estimation of $I_7$.} Denote $\zeta_{k+1}=\big<\nabla f_h(\theta_k),\mathcal{R}(\theta_k,\xi_{k+1})\big>$, then it is easy to see that $\left(\zeta_{k+1},\mathcal{F}_{k+1}\right)_{k\geq0}$ is a sequence of martingale differences and one can derive from \eqref{boundedness}
\begin{align*}
\lvert \zeta_{k+1}\rvert
\leq \Arrowvert\nabla f_h(\theta_k)\Arrowvert_2\cdot\Arrowvert\mathcal{R}(\theta_k,\xi_{k+1})\Arrowvert_2
\leq c\big(1+\Arrowvert\xi_{k+1}\Arrowvert_2^2\big),
\end{align*}
where $c=\Arrowvert\nabla f_h(\theta_k)\Arrowvert_2\Arrowvert\sigma(\theta_k)\Arrowvert\Arrowvert\nabla\sigma(\theta_k)\Arrowvert$. There exists some positive constant c such that
$$
\E\big(\lvert\zeta_{k+1}\rvert^2\exp\{c\lvert\zeta_{k+1}\rvert\}\big|\mathcal{F}_k\big)<\infty.
$$
Therefore, we have for all $y>c\eta^{1/2}$,
\begin{align*}
I_7
&=\P\big(\lvert\mathcal{R}_{\eta,7}\rvert\geq cy\big)=\P\Big(\Big\lvert\sum_{k=0}^{m-1}\zeta_{k+1}\Big\rvert\geq cy\eta^{-3/2}\Big)\\
&\leq c_1\exp\Big\{-\frac{(y\eta^{-3/2})^2}{c_2\big(\eta^{-2}+cy\eta^{-3/2}\big)}\Big\}\leq c_1\exp\Big\{-\frac{y^2\eta^{-1}}{c_2(1+\eta^{1/2}y)}\Big\}.
\end{align*}

\textbf{(8) Estimation of $I_8$.} By \eqref{boundedness} and following the same procedure as in the estimation of $I_{\eta,4,1}^{\prime\prime}$, we get for all $y>c\eta^{1/2}$,
\begin{align*}
\P\big(\lvert\mathcal{R}_{\eta,8}\rvert\geq cy\big)
&\leq\P\left(\sum_{k=0}^{m-1}\Arrowvert b(\theta_k)\Arrowvert_2\Arrowvert\mathcal{R}(\theta_k,\xi_{k+1})\Arrowvert_2\geq cy\eta^{-5/2}\right)\\
&\leq\P\left(\eta\sum_{k=0}^{m-1}\Arrowvert b(\theta_k)\Arrowvert_2^2\geq Cy^{2/3}\eta^{-4/3}\right)+\P\left(\sum_{k=0}^{m-1}\Arrowvert\mathcal{R}(\theta_k,\xi_{k+1})\Arrowvert_2^2\geq cy^{4/3}\eta^{-8/3}\right).
\end{align*}

Next, we give an estimation for the second term in the last inequality. Denote $T_{k+1}=\Arrowvert\mathcal{R}(\theta_k,\xi_{k+1})\Arrowvert_2^2-\E\Arrowvert\mathcal{R}(\theta_k,\xi_{k+1})\Arrowvert_2^2$, then it is easy to see that there exists some positive constant $c$ such that
$$
\E\big(T_{k+1}^2\exp{\big\{c\lvert T_{k+1}\rvert^{1/2}\big\}}\big\lvert\mathcal{F}_k\big)<\infty.
$$
Using Proposition \ref{martingale differences} with $\alpha=1/2$, we have for all $y>c\eta^{1/2}$,
\begin{align*}
&\P\Big(\sum_{k=0}^{m-1}\Arrowvert\mathcal{R}(\theta_k,\xi_{k+1})\Arrowvert_2^2\geq cy^{4/3}\eta^{-8/3}\Big)\\
=&\P\Big(\sum_{k=0}^{m-1}T_{k+1}\geq \big(cy^{4/3}\eta^{-8/3}-\sum_{k=0}^{m-1}\E\Arrowvert\mathcal{R}(\theta_k,\xi_{k+1})\Arrowvert_2^2\big)\Big)\\
\leq &c\exp{\left\{-\frac{(y^{4/3}\eta^{-8/3})^2}{c_1\left(\eta^{-2}+\left(y^{4/3}\eta^{-8/3}\right)^{3/2}\right)}\right\}}
\leq c\exp{\big\{-c_2y^{2/3}\eta^{-4/3}\big\}}.
\end{align*}

Hence, for all $y>\max\left\{c\eta^{1/2},\eta^{1/2}\mathbf{R}^3\Arrowvert\nabla\sigma\Arrowvert_{\infty}^3\right\}$ and $\mathbf{R}$ satisfying \eqref{R2}, one can derive
\begin{align*}
I_8&\leq c_1e^{c_2\Arrowvert\sigma\Arrowvert_{\infty}^2\eta^{-1}\left(1+\Arrowvert\nabla\sigma\Arrowvert_{\infty}^2(\mathbf{R}^2+1)\right)}e^{-c_3y^{\frac{2}{3}}\eta^{-\frac{4}{3}}}+c_4e^{-\frac{\mathbf{R}^2}{2}}.
\end{align*}

\textbf{(9) Estimation of $I_9$.} Denote
$$
h_{k+1}
=\big<\nabla^2 f_h(\theta_k),\sigma(\theta_k)\xi_{k+1}\big(\mathcal{R}(\theta_k,\xi_{k+1})\big)^{\rm T}\big>_{\rm{HS}}
+\big<\nabla^2 f_h(\theta_k),\mathcal{R}(\theta_k,\xi_{k+1})\big(\sigma(\theta_k)\xi_{k+1}\big)^{\rm T}\big>_{\rm{HS}},
$$
then $\big(h_{k+1},\mathcal{F}_{k+1}\big)_{k\geq0}$ is a sequence of martingale differences. Moreover, by \eqref{boundedness}, we further have
$$
\lvert h_{k+1}\rvert\leq c\big(\Arrowvert\xi_{k+1}\Arrowvert_2^3+\Arrowvert\xi_{k+1}\Arrowvert_2\big),
$$
where $c=\Arrowvert\nabla^2f_h(\theta_k)\Arrowvert\Arrowvert\sigma(\theta_k)\Arrowvert^2\Arrowvert\nabla\sigma(\theta_k)\Arrowvert$. Therefore, there exists some positive constant $c$ such that
$$
\E\big(h_{k+1}^2\exp{\big\{c\lvert h_{k+1}\rvert^{2/3}\big\}}\big\lvert\mathcal{F}_k\big)<\infty.
$$
Using Proposition \ref{martingale differences} with $\alpha=2/3$, we have for all $y\geq c\eta^{1/2}$,
\begin{align*}
I_9
&\leq\P\Big(\sum_{k=0}^{m-1}h_{k+1}\geq cy\eta^{-2}\Big)=\P\Big(\sum_{k=0}^{m-1}\big(h_{k+1}-\E h_{k+1}\big)\geq\big(cy\eta^{-2}-\sum_{k=0}^{m-1}\E h_{k+1}\big)\Big)\\
&\leq c\exp{\left\{-\frac{(y\eta^{-2})^2}{c(\eta^{-2}+(y\eta^{-2})^{4/3})}\right\}}
\leq c\exp{\big\{-c_1y^{2/3}\eta^{-4/3}\big\}}.
\end{align*}

\textbf{(10) Estimation of $I_{10}$.} For all $y>c\eta^{1/2}$ with $c$ large enough, one can derive from \eqref{boundedness} that
\begin{align*}
&\P\big(\lvert\mathcal{R}_{\eta,10}\rvert\geq y/13\big)\\
\leq&\P\left(\eta^{\frac{5}{2}}\sum_{k=0}^{m-1}\big<\nabla^2f_h(\theta_k),\mathcal{R}(\theta_k,\xi_{k+1})\big(\mathcal{R}(\theta_k,\xi_{k+1})\big)^{\rm T}\big>_{\rm{HS}}\geq \frac{4y}{13}\right)\\
&+\P\left(\eta^{\frac{7}{2}}\sum_{k=0}^{m-1}\int_{0}^{1}(1-t)^2\nabla_{\mathcal{R}(\theta_k,\xi_{k+1})}\nabla_{\mathcal{R}(\theta_k,\xi_{k+1})}\nabla_{\mathcal{R}(\theta_k,\xi_{k+1})}f_h\big(\theta_k+t\triangle\theta_k\big)\mathrm{d}t\geq \frac{8y}{13}\right)\\
\leq&\P\left(\sum_{k=0}^{m-1}\Arrowvert\mathcal{R}(\theta_k,\xi_{k+1})\Arrowvert_2^2\geq cy\eta^{-5/2}\right)+\P\left(\sum_{k=0}^{m-1}\Arrowvert\mathcal{R}(\theta_k,\xi_{k+1})\Arrowvert_2^3\geq cy\eta^{-7/2}\right)\\
:=&I_{\eta,10,1}+I_{\eta,10,2}.
\end{align*}

First, we estimate $I_{\eta,10,1}$. Denote $T_{k+1}=\Arrowvert\mathcal{R}(\theta_k,\xi_{k+1})\Arrowvert_2^2-\E\Arrowvert\mathcal{R}(\theta_k,\xi_{k+1})\Arrowvert_2^2$. Clearly, there exists some positive constant $c$ such that
$$
\E\big(T_{k+1}^2\exp{\big\{c\lvert T_{k+1}\rvert^{1/2}\big\}}\big\lvert\mathcal{F}_k\big)<\infty.
$$
Using Proposition \ref{martingale differences} with $\alpha=1/2$, we have for all $y>c\eta^{1/2}$,
\begin{align*}
I_{\eta,10,1}
&=\P\left(\sum_{k=0}^{m-1}\Arrowvert\mathcal{R}(\theta_k,\xi_{k+1})\Arrowvert_2^2\geq cy\eta^{-5/2}\right)\\
&=\P\left(\sum_{k=0}^{m-1}T_{k+1}\geq\big(cy\eta^{-5/2}-\sum_{k=0}^{m-1}\E\Arrowvert\mathcal{R}(\theta_k,\xi_{k+1})\Arrowvert_2^2\big)\right)\\
&\leq c\exp{\left\{-\frac{(y\eta^{-5/2})^2}{c\big(\eta^{-2}+(y\eta^{-5/2})^{3/2}\big)}\right\}}\leq c\exp{\big\{-c_1y^{1/2}\eta^{-5/4}\big\}}.
\end{align*}

Second, we estimate $I_{\eta,10,2}$. Denote $T_{k+1}=\Arrowvert\mathcal{R}(\theta_k,\xi_{k+1})\Arrowvert_2^3-\E\Arrowvert\mathcal{R}(\theta_k,\xi_{k+1})\Arrowvert_2^3$. By the same procedure as above with $\alpha=1/3$, we have
\begin{align*}
I_{\eta,10,2}
&=\P\left(\sum_{k=0}^{m-1}\Arrowvert\mathcal{R}(\theta_k,\xi_{k+1})\Arrowvert_2^3\geq cy\eta^{-7/2}\right)
\\&\leq c\exp{\left\{-\frac{(y\eta^{-7/2})^2}{c\big(\eta^{-2}+(y\eta^{-7/2})^{5/3}\big)}\right\}}\leq c\exp{\big\{-c_1y^{1/3}\eta^{-7/6}\big\}}.
\end{align*}

Finally, combining the estimations of $I_{\eta,10,1}$ and $I_{\eta,10,2}$, one can conclude that for all $y> c\eta^{1/2}$ with $c$ large enough,
$$
I_{10}\leq I_{\eta,10,1}+I_{\eta,10,2}\leq c\exp{\big\{-c_1y^{1/3}\eta^{-7/6}\big\}}.
$$

\textbf{(11) Estimation of $I_{11}$.} By \eqref{boundedness} and the Cauchy-Schwarz inequality, we have
\begin{align*}
I_{11}
\leq&\P\Big(\sum_{k=0}^{m-1}\Arrowvert b(\theta_k)\Arrowvert_2^2\Arrowvert\mathcal{R}(\theta_k,\xi_{k+1})\Arrowvert_2\geq cy\eta^{-7/2}\Big)+\P\Big(\sum_{k=0}^{m-1}\Arrowvert b(\theta_k)\Arrowvert_2\Arrowvert\mathcal{R}(\theta_k,\xi_{k+1})\Arrowvert_2^2\geq cy\eta^{-7/2}\Big)\\:=&I_{\eta,11,1}+I_{\eta,11,2}.
\end{align*}

First, we estimate $I_{\eta,11,1}$. Following the same procedure as in the estimation of $I_{\eta,4,1}^{\prime\prime}$, for all $y>\max\left\{c\eta^{1/2},\eta^{1/2}\mathbf{R}^3\Arrowvert\nabla\sigma\Arrowvert_{\infty}^3\right\}$ and $\mathbf{R}$ satisfying \eqref{R2}, we have
\begin{align*}
I_{\eta,11,1}
=&\P\left(\sum_{k=0}^{m-1}\Arrowvert b(\theta_k)\Arrowvert_2^2\Arrowvert\mathcal{R}(\theta_k,\xi_{k+1})\Arrowvert_2\geq cy\eta^{-7/2}\right)\\
\leq&\P\left(\eta\sum_{k=0}^{m-1}\Arrowvert b(\theta_k)\Arrowvert_2^2\geq Cy^{2/3}\eta^{-4/3}\right)+
\P\left(\sum_{k=0}^{m-1}\Arrowvert\mathcal{R}(\theta_k,\xi_{k+1})\Arrowvert_2^2\geq cy^{2/3}\eta^{-7/3}\right)\\
\leq&\P\left(\eta\sum_{k=0}^{m-1}\Arrowvert b(\theta_k)\Arrowvert_2^2\geq Cy^{2/3}\eta^{-4/3}\right)+c\exp{\left\{-\frac{(y^{2/3}\eta^{-7/3})^2}{c\big(\eta^{-2}+(y^{2/3}\eta^{-7/3})^{3/2}\big)}\right\}}\\
&\leq c_1e^{c_2\Arrowvert\sigma\Arrowvert_{\infty}^2\eta^{-1}\left(1+\Arrowvert\nabla\sigma\Arrowvert_{\infty}^2(\mathbf{R}^2+1)\right)}e^{-c_3y^{\frac{2}{3}}\eta^{-\frac{4}{3}}}+c_4e^{-\frac{\mathbf{R}^2}{2}}+ce^{-c_5y^{\frac{1}{3}}\eta^{-\frac{7}{6}}}.
\end{align*}

Similarly, we get for all $y> c\eta^{1/2}$ with $c$ large enough,
\begin{align*}
I_{\eta,11,2}
=&\P\left(\sum_{k=0}^{m-1}\Arrowvert b(\theta_k)\Arrowvert_2\Arrowvert\mathcal{R}(\theta_k,\xi_{k+1})\Arrowvert_2^2\geq cy\eta^{-7/2}\right)\\
\leq&\P\Big(\eta\sum_{k=0}^{m-1}\Arrowvert b(\theta_k)\Arrowvert_2^2\geq Cy^{2/3}\eta^{-4/3}\Big)+\P\Big(\sum_{k=0}^{m-1}\Arrowvert\mathcal{R}(\theta_k,\xi_{k+1})\Arrowvert_2^4\geq cy^{4/3}\eta^{-14/3}\Big).
\end{align*}
Furthermore, denote $T_{k+1}=\Arrowvert\mathcal{R}(\theta_k,\xi_{k+1})\Arrowvert_2^4-\E\Arrowvert\mathcal{R}(\theta_k,\xi_{k+1})\Arrowvert_2^4$ and then using Proposition \ref{martingale differences} with $\alpha=1/4$, we have
\begin{align*}
I_{\eta,11,2}
&\leq \P\Big(\eta\sum_{k=0}^{m-1}\Arrowvert b(\theta_k)\Arrowvert_2^2\geq Cy^{2/3}\eta^{-4/3}\Big)+c\exp\left\{-\frac{(y^{4/3}\eta^{-14/3})^2}{c_1\big(\eta^{-2}+(y^{4/3}\eta^{-14/3})^{7/4}\big)}\right\}\\
&\leq\P\Big(\eta\sum_{k=0}^{m-1}\Arrowvert b(\theta_k)\Arrowvert_2^2\geq Cy^{2/3}\eta^{-4/3}\Big)+c\exp\Big\{-y^{1/3}\eta^{-7/6}\Big\}.
\end{align*}

Hence, we get for all $y>\max\left\{c\eta^{1/2},\eta^{1/2}\mathbf{R}^3\Arrowvert\nabla\sigma\Arrowvert_{\infty}^3\right\}$ and $\mathbf{R}$ satisfying \eqref{R2},
$$
I_{11}\leq c_1e^{c_2\Arrowvert\sigma\Arrowvert_{\infty}^2\eta^{-1}\left(1+\Arrowvert\nabla\sigma\Arrowvert_{\infty}^2(\mathbf{R}^2+1)\right)}e^{-c_3y^{\frac{2}{3}}\eta^{-\frac{4}{3}}}+c_4e^{-\frac{\mathbf{R}^2}{2}}+ce^{-c_5y^{\frac{1}{3}}\eta^{-\frac{7}{6}}}.
$$

\textbf{(12) Estimation of $I_{12}$.} By \eqref{boundedness} and H\"older's inequality, we have for all $y>c\eta^{1/2}$ with $c$ large enough,
\begin{align*}
I_{12}
&\leq\P\left(\sum_{k=0}^{m-1}\Arrowvert b(\theta_k)\Arrowvert_2\Arrowvert\sigma(\theta_k)\xi_{k+1}\Arrowvert_2\Arrowvert\mathcal{R}(\theta_k,\xi_{k+1})\Arrowvert_2\geq cy\eta^{-3}\right)\\
&\leq\P\left(\Big(\sum_{k=0}^{m-1}\Arrowvert b(\theta_k)\Arrowvert_2^2\Big)^{1/2}\Big(\sum_{k=0}^{m-1}\Arrowvert\sigma(\theta_k)\xi_{k+1}\Arrowvert_2^2\Arrowvert\mathcal{R}(\theta_k,\xi_{k+1})\Arrowvert_2^2\Big)^{1/2}\geq cy\eta^{-3}\right)\\
&\leq\P\left(\eta\sum_{k=0}^{m-1}\Arrowvert b(\theta_k)\Arrowvert_2^2\geq Cy^{2/3}\eta^{-4/3}\right)+\P\left(\sum_{k=0}^{m-1}\Arrowvert\sigma(\theta_k)\xi_{k+1}\Arrowvert_2^2\Arrowvert\mathcal{R}(\theta_k,\xi_{k+1})\Arrowvert_2^2\geq Cy^{4/3}\eta^{-11/3}\right).
\end{align*}

We now estimate the second term in the expression above. Denote
$$
T_{k+1}=\Arrowvert\sigma(\theta_k)\xi_{k+1}\Arrowvert_2^4-\E\Arrowvert\sigma(\theta_k)\xi_{k+1}\Arrowvert_2^4\quad \text{and} \quad
H_{k+1}=\Arrowvert\mathcal{R}(\theta_k,\xi_{k+1})\Arrowvert_2^4-\E\Arrowvert\mathcal{R}(\theta_k,\xi_{k+1})\Arrowvert_2^4,
$$
using Proposition \ref{martingale differences} with $\alpha_1=1/2$ and $\alpha_2=1/4$ respectively, we obtain
\begin{align*}
&\P\left(\sum_{k=0}^{m-1}\Arrowvert\sigma(\theta_k)\xi_{k+1}\Arrowvert_2^2\Arrowvert\mathcal{R}(\theta_k,\xi_{k+1})\Arrowvert_2^2\geq Cy^{4/3}\eta^{-11/3}\right)\\
\leq&\P\left(\Big(\sum_{k=0}^{m-1}\Arrowvert\sigma(\theta_k)\xi_{k+1}\Arrowvert_2^4\Big)^{1/2}\Big(\sum_{k=0}^{m-1}\Arrowvert\mathcal{R}(\theta_k,\xi_{k+1})\Arrowvert_2^4\Big)^{1/2}\geq Cy^{4/3}\eta^{-11/3},\right.\\&\left.\qquad\qquad\qquad\qquad\qquad\qquad\qquad\qquad\qquad\qquad\sum_{k=0}^{m-1}\Arrowvert\sigma(\theta_k)\xi_{k+1}\Arrowvert_2^4\geq cy^{4/3}\eta^{-8/3}\right)\\
&+\P\left(\Big(\sum_{k=0}^{m-1}\Arrowvert\sigma(\theta_k)\xi_{k+1}\Arrowvert_2^4\Big)^{1/2}\Big(\sum_{k=0}^{m-1}\Arrowvert\mathcal{R}(\theta_k,\xi_{k+1})\Arrowvert_2^4\Big)^{1/2}\geq Cy^{4/3}\eta^{-11/3},\right.\\&\left.\qquad\qquad\qquad\qquad\qquad\qquad\qquad\qquad\qquad\qquad\sum_{k=0}^{m-1}\Arrowvert\sigma(\theta_k)\xi_{k+1}\Arrowvert_2^4<cy^{4/3}\eta^{-8/3}\right)\\
\leq&\P\left(\sum_{k=0}^{m-1}\Arrowvert\sigma(\theta_k)\xi_{k+1}\Arrowvert_2^4\geq cy^{4/3}\eta^{-8/3}\right)+\P\left(\sum_{k=0}^{m-1}\Arrowvert\mathcal{R}(\theta_k,\xi_{k+1})\Arrowvert_2^4\geq cy^{4/3}\eta^{-14/3}\right)\\
\leq& c\exp\Big\{-\frac{(y^{4/3}\eta^{-8/3})^2}{c_1\big(\eta^{-2}+(y^{4/3}\eta^{-8/3})^{3/2}\big)}\Big\}+
c\exp\Big\{-\frac{(y^{4/3}\eta^{-14/3})^2}{c_1\big(\eta^{-2}+(y^{4/3}\eta^{-14/3})^{7/4}\big)}\Big\}\\
\leq& c\exp\big\{-cy^{2/3}\eta^{-4/3}\big\}+c\exp\big\{-cy^{1/3}\eta^{-7/6}\big\}.
\end{align*}

Hence, for all $y>\max\left\{c\eta^{1/2},\eta^{1/2}\mathbf{R}^3\Arrowvert\nabla\sigma\Arrowvert_{\infty}^3\right\}$ and $\mathbf{R}$ satisfying \eqref{R2}, one can obtain
\begin{align*}
I_{12}
&\leq \P\Big(\eta\sum_{k=0}^{m-1}\Arrowvert b(\theta_k)\Arrowvert_2^2\geq Cy^{2/3}\eta^{-4/3}\Big)+c\exp\big\{-cy^{2/3}\eta^{-4/3}\big\}+c\exp\big\{-cy^{1/3}\eta^{-7/6}\big\}
\\&\leq c_1e^{c_2\Arrowvert\sigma\Arrowvert_{\infty}^2\eta^{-1}\left(1+\Arrowvert\nabla\sigma\Arrowvert_{\infty}^2(\mathbf{R}^2+1)\right)}e^{-c_3y^{\frac{2}{3}}\eta^{-\frac{4}{3}}}+c_4e^{-\frac{\mathbf{R}^2}{2}}+ce^{-c_5y^{\frac{1}{3}}\eta^{-\frac{7}{6}}}.
\end{align*}

\textbf{(13) Estimation of $I_{13}$.} By \eqref{boundedness}, it is easy to see that for all $y>c\eta^{1/2}$,
\begin{align*}
I_{13}
=&\P\Big(\lvert\mathcal{R}_{\eta,13}\rvert\geq cy\Big)\\
\leq&\P\left(\sum_{k=0}^{m-1}\Arrowvert\sigma(\theta_k)\xi_{k+1}\Arrowvert_2^2\Arrowvert\mathcal{R}(\theta_k,\xi_{k+1})\Arrowvert_2\geq cy\eta^{-5/2}\right)\\&
+\P\left(\sum_{k=0}^{m-1}\Arrowvert\sigma(\theta_k)\xi_{k+1}\Arrowvert_2\Arrowvert\mathcal{R}(\theta_k,\xi_{k+1})\Arrowvert_2^2\geq cy\eta^{-3}\right)\\
:=&I_{\eta,13,1}+I_{\eta,13,2},
\end{align*}

For the first term, one can derive from the H\"older's inequality that
\begin{align*}
&I_{\eta,13,1}\\
\leq&\P\left(\Big(\sum_{k=0}^{m-1}\Arrowvert\sigma(\theta_k)\xi_{k+1}\Arrowvert_2^4\Big)^{1/2}\Big(\sum_{k=0}^{m-1}\Arrowvert\mathcal{R}(\theta_k,\xi_{k+1})\Arrowvert_2^2\Big)^{1/2}\geq cy\eta^{-5/2},\sum_{k=0}^{m-1}\Arrowvert\sigma(\theta_k)\xi_{k+1}\Arrowvert_2^4\geq Cy\eta^{-5/2}\right)\\
&+\P\left(\Big(\sum_{k=0}^{m-1}\Arrowvert\sigma(\theta_k)\xi_{k+1}\Arrowvert_2^4\Big)^{1/2}\Big(\sum_{k=0}^{m-1}\Arrowvert\mathcal{R}(\theta_k,\xi_{k+1})\Arrowvert_2^2\Big)^{1/2}\geq cy\eta^{-5/2},\sum_{k=0}^{m-1}\Arrowvert\sigma(\theta_k)\xi_{k+1}\Arrowvert_2^4<Cy\eta^{-5/2}\right)\\
\leq&\P\left(\sum_{k=0}^{m-1}\Arrowvert\sigma(\theta_k)\xi_{k+1}\Arrowvert_2^4\geq Cy\eta^{-5/2}\right)+\P\left(\sum_{k=0}^{m-1}\Arrowvert\mathcal{R}(\theta_k,\xi_{k+1})\Arrowvert_2^2\geq Cy\eta^{-5/2}\right).
\end{align*}
Denote
$$
T_{k+1}=\Arrowvert\sigma(\theta_k)\xi_{k+1}\Arrowvert_2^4-\E\Arrowvert\sigma(\theta_k)\xi_{k+1}\Arrowvert_2^4\quad \text{and} \quad
H_{k+1}=\Arrowvert\mathcal{R}(\theta_k,\xi_{k+1})\Arrowvert_2^2-\E\Arrowvert\mathcal{R}(\theta_k,\xi_{k+1})\Arrowvert_2^2,
$$
then using Proposition \ref{martingale differences} with $\alpha=1/2$, we have
\begin{align*}
I_{\eta,13,1}
&\leq 2c\exp\Big\{-\frac{(y\eta^{-5/2})^2}{c_1\big(\eta^{-2}+(y\eta^{-5/2})^{3/2}\big)}\Big\}
\leq 2c\exp\big\{-c_2y^{1/2}\eta^{-5/4}\big\}.
\end{align*}

For the second term, we have
\begin{align*}
I_{\eta,13,2}
\leq&\P\left(\sum_{k=0}^{m-1}\Arrowvert\sigma(\theta_k)\xi_{k+1}\Arrowvert_2\Arrowvert\mathcal{R}(\theta_k,\xi_{k+1})\Arrowvert_2^2\geq cy\eta^{-3},\Arrowvert\sigma(\theta_k)\xi_{k+1}\Arrowvert_2\geq Cy^{1/5}\eta^{-3/5}\right)\\
&+\P\left(\sum_{k=0}^{m-1}\Arrowvert\sigma(\theta_k)\xi_{k+1}\Arrowvert_2\Arrowvert\mathcal{R}(\theta_k,\xi_{k+1})\Arrowvert_2^2\geq cy\eta^{-3},\Arrowvert\sigma(\theta_k)\xi_{k+1}\Arrowvert_2<Cy^{1/5}\eta^{-3/5}\right)\\
\leq&\sum_{k=0}^{m-1}\P\Big(\Arrowvert\sigma(\theta_k)\xi_{k+1}\Arrowvert_2\geq Cy^{1/5}\eta^{-3/5}\Big)+\P\left(\sum_{k=0}^{m-1}\Arrowvert\mathcal{R}(\theta_k,\xi_{k+1})\Arrowvert_2^2\geq cy^{4/5}\eta^{-12/5}\right)\\
\leq&\eta^{-2}\exp\big\{-c_1y^{2/5}\eta^{-6/5}\big\}+c\exp\Big\{-\frac{(y^{4/5}\eta^{-12/5})^2}{c_2\big(\eta^{-2}+\left(y^{4/5}\eta^{-12/5}\right)^{3/2}\big)}\Big\}\\
\leq&\eta^{-2}\exp\big\{-c_1y^{2/5}\eta^{-6/5}\big\}+c\exp\Big\{-y^{2/5}\eta^{-6/5}\Big\},
\end{align*}
where the third inequality follows from Proposition \ref{martingale differences} with $\alpha=1/2$. Thus we obtain that for all $y>c\eta^{1/2}$,
\begin{align*}
I_{13}\leq I_{\eta,13,1}+I_{\eta,13,2}\leq c\exp\big\{-y^{2/5}\eta^{-6/5}\big\}.
\end{align*}

Therefore, by the estimations of $I_i,1\leq i\leq 13$, for $\max\left\{c\eta^{1/2},\eta^{1/2}\mathbf{R}^3\Arrowvert\nabla\sigma\Arrowvert_{\infty}^3\right\}<y=o(\eta^{-1/2})$ and $\mathbf{R}$ satisfying \eqref{R2}, we have
\begin{align}\nonumber
\P\Big(\big\lvert\mathcal{R}_\eta\big\rvert\geq y\Big)\leq c_1\left(\exp\left\{-\frac{y^2\eta^{-1}}{c_2(1+\eta^{1/2}y)}\right\}+\exp\left\{-\frac{\mathbf{R}^2}{2}\right\}+\exp\left\{-c_3y^{1/3}\eta^{-7/6}\right\}\right).
\end{align}
Proposition \ref{remainder} is therefore proved.

\qed
\end{appendix}

\section*{Acknowledgements}

Peng Chen is supported by the National Natural
Science Foundation of China (No. 12301176), the China Postdoctoral Science Foundation (No. 2023M741692) and the Natural Science Foundation of Jiangsu Province Grant (No. BK20220867). Hui Jiang is supported by the National Natural Science Foundation of China (No. 12471146) and the Natural Science
Foundation of Jiangsu Province of China (No. BK20231435).


\begin{thebibliography}{9}

 \bibitem{Abdulle}
Abdulle, A., Vilmart, G. and Zygalakis, K. C. (2014). High order numerical approximation of the invariant measure of ergodic sdes.
\textit{SIAM Journal on Numerical Analysis}, {\bf 52}(4), 1600-1622.

\bibitem{Bao}
Bao, J., Shao, J. and Yuan, C. (2016). Approximation of invariant measures for regime-switching diffusions.
\textit{Potential Analysis}, {\bf 44}(4), 707-727.


\bibitem{Chen-2023}
Chen, P., Deng, C. S., Schilling, R. L. and Xu, L. (2023). Approximation of the invariant measure of stable SDEs by an Euler-Maruyama scheme.
\textit{Stochastic Processes and their Applications}, {\bf 163}, 136-167.

\bibitem{Chen-2024}
Chen, P., Jin, X., Shen, T. and Su, Z. (2024). Variable-step Euler-Maruyama approximations of regime-switching jump diffusion processes.
\textit{Journal of Theoretical Probability}, {\bf 37}(2), 1597-1626.

\bibitem{Chen}
Chen, X., Shao, Q.-M., Wu, W. B. and Xu, L. (2016). Self-normalized Cram\'er-type moderate deviations under dependence.
\textit{Annals of Statistics}, {\bf 44}, 1593-1617.

\bibitem{CZZ} Chen, Z. Q., Zhang, X. and Zhao, G. (2021). Supercritical SDEs driven by multiplicative stable-like L$\acute{e}$vy processes. \textit{Transactions of the American Mathematical Society}, {\bf 374}(11), 7621-7655.



\bibitem{Durmus}
Durmus, A., Moulines, E., Naumov, A., Samsonov, S., and Monica, A. (2024). Probability and moment inequalities for additive functionals of geometrically ergodic Markov chains.
\textit{Journal of Theoretical Probability}, {\bf 37}(3), 2184-2233.

\bibitem{Weinan}
E., W., Li, T. and Vanden-Eijnden, E. (2021). Applied Stochastic Analysis (Vol. 199).
\textit{American Mathematical Society}.

\bibitem{Evan} Evans, L. C. (2012). An introduction to stochastic differential equations (Vol. 82). American Mathematical Society.

\bibitem{Grama}
Fan, X., Grama, I. and Liu, Q. (2013). Cram\'{e}r large deviation expansions for martingales under Bernstein's condition.
\textit{Stochastic Processes and their Applications}, {\bf 123}(11), 3919-3942.

\bibitem{Fan-2020}
Fan, X., Grama, I., Liu, Q. and Shao, Q.-M. (2020). Self-normalized Cram\'er-type moderate deviations for stationary sequences and applications.
\textit{Stochastic Processes and their Applications}, {\bf 130}(8), 5124-5148.

\bibitem{Fan}
Fan, X., Hu, H. and Xu, L. (2024). Normalized and self-normalized Cram\'er-type moderate deviations for the Euler-Maruyama scheme for the SDE.
\textit{Science China Mathematics}, {\bf 67}(8), 1865-1880.

\bibitem{Shao}
Fan, X. and Shao, Q. M. (2023). Cram\'{e}r moderate deviations for martingales with applications.
\textit{Annales de l'Institut Henri Poincare (B) Probabilites et Statistiques}, {\bf 60}(3), 2046-2074.

\bibitem{Fang}
Fang, X., Shao, Q.-M. and Xu, L. (2019). Multivariate approximations in Wasserstein distance by Stein's method and Bismut's formula.
\textit{Probability Theory and Related Fields}, {\bf 174}, 945-979.

\bibitem{FSZ21} Feng, X., Shao, Q. M. and Zeitouni, O. (2021). Self-normalized moderate deviations for random walk in random scenery. \textit{Journal of Theoretical Probability}, {\bf 34}, 103-124.

\bibitem{GJZ23} Gao, F., Jiang, H. and Zhao, X. (2023). Cram\'{e}r type moderate deviations for the Grenander estimator near the boundaries of the support. \textit{Bernoulli}, {\bf 29}(4), 2854-2878.

\bibitem{Gao}
Gao, S., Li, X. and Liu, Z. (2023). Stationary distribution of the milstein scheme for stochastic differential delay equations with first-order convergence.
\textit{Applied Mathematics and Computation}, {\bf 458}, 128224.

\bibitem{Szpruch}
Giles, M. B., Majka, M. B., Szpruch, L., Vollmer, S. J. and Zygalakis, K. C. (2020). Multi-level Monte Carlo methods for the approximation of invariant measures of stochastic differential equations. \textit{Statistics and Computing}, {\bf 30}(3), 507-524.

\bibitem{JPW24} Jiang, H., Pan, Y. and Wei, X. (2024). Self-Normalized Cram\'{e}r-Type Moderate Deviations for Explosive Vasicek Model. \textit{Journal of Theoretical Probability}, {\bf 37}(1), 228-250.

\bibitem{Jiang}
Jiang, Y., Hu, L. and Lu, J. (2020). The stochastic $\theta$ method for stationary distribution of stochastic differential equations with markovian switching.
\textit{Advances in Difference Equations}, {\bf 2020}(1), 1-25.

\bibitem{Jiang-2020}
Jiang, Y., Weng, L. and  Liu, W. (2020). Stationary distribution of the stochastic theta method for nonlinear stochastic differential equations.
\textit{Numerical Algorithms}, {\bf 83}, 1531-1553.

\bibitem{Xing}
Jin, X., Wang, W., Xu, L. and Zhang, T. (2023). Approximation of the ergodic measure of SDEs with singular drift by Euler-Maruyama scheme.
\textit{arXiv:2301.08903.}

\bibitem{Liu}
Liu, W. and Mao, X. (2015). Numerical stationary distribution and its convergence for nonlinear stochastic differential equations.
\textit{Journal of Computational and Applied Mathematics}, {\bf 276}, 16-29.

\bibitem{Liu-2023}
Liu, W., Mao, X. and Wu, Y. (2023). The backward Euler-Maruyama method for invariant measures of stochastic differential equations with super-linear coefficients.
\textit{Applied Numerical Mathematics}, {\bf 184}, 137-150.

\bibitem{Li}
Li, X., Ma, Q., Yang, H. and Yuan, C. (2018). The numerical invariant measure of stochastic differential equations with Markovian switching.
\textit{SIAM Journal on Numerical Analysis}, {\bf 56}(3), 1435-1455.

\bibitem{Lu}
Lu, J., Tan, Y. and Xu, L. (2022). Central limit theorem and self-normalized Cram\'er-type moderate deviation for Euler-Maruyama scheme.
\textit{Bernoulli}, {\bf 28}(2), 937-964.

\bibitem{Mao}
Mao, X., Yuan, C. and Yin, G. (2005). Numerical method for stationary distribution of stochastic differential equations with Markovian switching.
\textit{Journal of Computational and Applied Mathematics}, {\bf 174}(1), 1-27.

\bibitem{RWZ} Ren, J., Wu, J. and Zhang, H. (2014). On existence, uniqueness and convergence of
multi-valued stochastic differential equations driven by continuous semimartingales. \textit{Science
China Mathematics}, {\bf 57}(3), 589-607.

\bibitem{RT96} Roberts, G. O. and Tweedie, R. L. (1996). Exponential convergence of Langevin distributions and their discrete approximations. {\it Bernoulli}, {\bf 2}, 341-363.

\bibitem{SZZ21} Shao, Q. M., Zhang, M. and Zhang, Z. S. (2021). Cram\'{e}r-type moderate deviation theorems for nonnormal approximation. \textit{The Annals of Applied Probability}, {\bf 31}(1), 247-283.

\bibitem{Talay}
Talay, D. (1990). Second-order discretization schemes of stochastic differential systems for the computation of the invariant law. \textit{Stochastics: An International Journal of Probability and Stochastic Processes}, {\bf 29}(1), 13-36.

\bibitem{WL19} Weng, L. and Liu, W. (2019). Invariant measures of the Milstein method for stochastic differential equations with commutative noise. \textit{Applied Mathematics and Computation}, {\bf 358}, 169-176.

\bibitem{Yuan}
Yuan, C. and Mao, X. (2004). Stability in distribution of numerical solutions for stochastic differential equations.
\textit{Stochastic Analysis and Applications}, {\bf 22}(5), 1133-1150.


\end{thebibliography}
\end{document}